\documentclass[12pt,a4paper]{amsart}

\usepackage[fontsize=12pt]{scrextend}
\usepackage[english]{babel}	 							
\usepackage{amsmath,amsfonts,amsthm,mathtools,amssymb,amsxtra,calligra,mathrsfs,mathtools,tikz-cd,tikz,mathrsfs,comment,MnSymbol}					    
\usepackage{stmaryrd}
\usepackage{xcolor}
\usepackage{url}
\usepackage{color}
\usepackage{hyperref}
\usepackage[margin=1in]{geometry}
\usetikzlibrary{arrows}
\usepackage[shortlabels]{enumitem}

\newtheorem{theorem}{Theorem}
\newtheorem{prop}[theorem]{Proposition}
\newtheorem{lemma}[theorem]{Lemma}

\newtheorem{definition}[theorem]{Definition}

\title{Rational curves on K3 surfaces of small genus}
\author{Rijul Saini}
\address{School of Mathematics, Tata Institute of Fundamental Research,  
1 Homi Bhabha Road, Colaba, Mumbai, 400005, India.}
\email{rijul@math.tifr.res.in}
\keywords{Rational curves, K3 surfaces, Monodromy groups, Genus 4 curves}
\subjclass[2020]{Primary 14J28; Secondary 20B25, 14H45}
\begin{document}

\maketitle
\begin{abstract}
    Let $\mathfrak B_g$ denote the moduli space of primitively polarized $K3$ surfaces $(S,H)$ of genus $g$ over $\mathbb C$. It is well-known that $\mathfrak B_g$ is irreducible and that there are only finitely many rational curves in $|H|$ for any primitively polarized $K3$ surface $(S,H)$. So we can ask the question of finding the monodromy group of such curves. The case of $g=2$ essentially follows from the results of Harris \cite{Ha} to be the full symmetric group $S_{324}$, here we solve the case $g=3$ and $4$.
\end{abstract}
\setcounter{tocdepth}{2}
\tableofcontents  
\thispagestyle{empty}
\section{Introduction}
\bigskip
We work over the field $\mathbb C$. Unless otherwise specified, whenever we refer to \textit{genus} of a curve over a field $k$, we always mean the arithmetic genus of the curve. 

Let $\mathfrak B_g$ denote the moduli space of primitively polarized $K3$ surfaces $(S,\mathcal L)$ of genus $g$, i.e. $S$ is a $K3$ surface and $\mathcal L$ is a primitive polarization of $S$ of degree $(\mathcal L)^2 = 2g-2$. Then $\mathfrak B_g$ is irreducible. Due to Chen \cite{Ch} for generic $(S,\mathcal L) \in \mathfrak B_g$, every rational curve in $|H|$ is an integral nodal rational curve. Also for any $(S,\mathcal L) \in \mathfrak B_g$ there are only finitely many rational curves in $|\mathcal L|$ with the number $n_g$ for the generic $(S, \mathcal L)$ given by the Yau-Zaslow formula
$$\sum_{g \ge 0} n_gq^g = \frac q {\Delta(q)} = \prod_{n \ge 1} (1-q^n)^{-24} = 1 + 24 q + 324q^2 + 3200q^3 + 25650 q^4 + \ldots$$
conjectured by Yau and Zaslow in \cite{YZ} and proven by Beauville in \cite{Be} under the assumption that all rational curves in $|\mathcal L|$ are nodal which was proven by Chen in \cite{Ch}. 

Thus, we may study the monodromy group $\Gamma_g$ of these rational curves. That is, we may consider the so called Severi variety $\mathfrak C_g$ of $(S,\mathcal L,C)$ such that $(S,\mathcal L)$ is a primitively polarized $K3$ surface of genus $g$, and $C \in |\mathcal L|$ is a rational curve, and consider the generically finite map $\gamma_g : \mathfrak C_g \to \mathfrak B_g$ and study the monodromy group $\Gamma_g$ of $\gamma_g$.

For $g=2$: The generic primitively polarized $(S,\mathcal L)$ with $(\mathcal L)^2 = 2$ is a double cover $\pi : S \to \mathbb P^2$ of $\mathbb P^2$ ramified over a generic sextic curve $C \subset \mathbb P^2$ with $\mathcal L = \pi^* \mathcal O(1)$. As $H^0(S, \pi^* \mathcal O(1)) \cong H^0(\mathbb P^2, \mathcal O(1))$, all curves in $|\mathcal L|$ are double covers $\pi: C_\ell \to \ell$ of some line $\ell \subset \mathbb P^2$. Thus, for $C_\ell$ to be rational, $\ell$ must be a bitangent to the sextic curve $C$. The monodromy of these bitangent curves was studied by Harris in \cite{Ha}, where he established that the monodromy group of the
bitangents of a generic degree $d$ plane curve is $O_6(\mathbb Z/2 \mathbb Z)$ if $d = 4$, and is the full symmetric group if $d > 4$. Thus, we get that the monodromy group $\Gamma_2$ is the full symmetric group $S_{324}$. (This is observed in \cite[Page 295]{huy})

In this paper, we prove the case of $g=3$ and $g=4$. 

\begin{theorem}[Main Theorem] $\Gamma_g$ is the full symmetric group for $g=3,4$ i.e. $\Gamma_3 \cong S_{3200}$ and $\Gamma_4 \cong S_{25650}$.
\end{theorem}

In the proof of Harris, to establish the existence of a simple transposition one studies flex bitangents to the curve. In our case, this corresponds to a rational curve with one simple cusp and rest of the singularities being nodes. Also, the Plucker formula gets replaced by the Yau-Zaslow formula as proven by Beauville \cite{Be} where he proves it for any $(S,\mathcal L)$ with curves being counted with certain multiplicities.

Note that for generic $(S,\mathcal L) \in \mathfrak B_g$, the Picard group of $S$ is generated by the class of $\mathcal L$, and $\mathcal L$ is very ample if $g \ge 3$. So, we may consider the embedding $\varphi_{|\mathcal L|} : S \hookrightarrow \mathbb P^g$ given by $|\mathcal L|$.

If $g = 3,4,$ or $5$, then it is known that the generic $(S,\mathcal L) \in \mathfrak B_g$ is embedded via $\varphi_{|\mathcal L|}$ as a complete intersection in $ \mathbb P^g$: for $g = 3$ it is given by a quartic in $\mathbb P^3$, for $g=4$ it is given as the intersection of a quadric and a cubic in $\mathbb P^4$, and for $g = 5$ it is given as the intersection of $3$ quadrics in $\mathbb P^5$. Also, the general complete intersection of that type will be a $K3$ surface.

So let $W_g$ be the space of complete intersections in $\mathbb P^g$ of the type given as above. Let $U_g \subset W_g$ be the open subset consisting of smooth surfaces. Then $U_g$ has a natural map to $\mathfrak B_g$. Let $$J_g = \{ (S,H)  \ | \ S \cap H \textup{ is rational } \} \subseteq W_g \times (\mathbb P^g)^{\vee},$$ and let $\pi : J_g \to W_g,  \ \eta : J_g \to (\mathbb P^g)^{\vee}$ be the projection maps.

Then $\pi|_{U_g} : \pi^{-1}(U_g) \to U_g$ is exactly the pull back of $\gamma_g$ to $W_g$. So it is enough to show that the monodromy group of $\pi$ is $S_{n_g}$ to prove that $\Gamma_g \cong S_{n_g}$. We will denote the monodromy group of $\pi$ by $\Pi_g$.

We follow the same lines of argument as in Harris's calculation of monodromy in the paper \cite{Ha}, i.e. we prove that $\Pi_g$ is successively transitive, 2-transitive, and finally contains a simple transposition. 

In the paper \cite{chen1998rational} by Chen, he conjectures (Conjecture 1.2) that the monodromy is always transitive (and also proves it in the cases $g=3,4$ which we consider in this paper). In \cite{CD}, it is proven that the monodromy is transitive for $3 \le g \le 9$ and $g = 11$, by proving that the universal Severi variety $\mathcal V^g_0$ is irreducible.

For transitivity and 2-transitivity we prove that certain subschemes of $J_g$ are irreducible, and their complements do not dominate $W_g$. To prove existence of a simple transposition we prove that there is a surface $S$ on the moduli space and a hyperplane $H$ so that if $C = S \cap H$, then the singularities of $C$ are all simple nodes but for one point where it is a simple cusp, and that any other hyperplane section of $S$ which is a rational curve is actually rational nodal, and that $J_g$ is locally irreducible at this point $(S,H)$ on $J_g$.

The paper is organized as follows: In Section 2 we give an outline of the proof. In Section 3 we discuss the space of maps from a curve to a surface with the image containing some specified points and we discuss a space of rational curves of genus $g$ (with some assumptions on the singularities) which are canonically embedded in $\mathbb P^{g-1}$. In Section 4 we begin by proving the existence of a $K3$ surface in $\mathbb P^3$ with two given curves as hyperplane sections and then give the proof for $g=3$. In section 5 we prove the corresponding lemmas as in section 4 for $g=4$ and then give the proof for $g=4$. 

While writing the paper, we found out that the case of $g=3$ was proven by Sailun Zhan \cite{SZ}. Our proof is different in this case and we also prove the substantially more difficult case of $g=4$.

\textbf{Acknowledgement:} I would like to thank my advisor Prof. N. Fakhruddin for his guidance and patience. This paper would not be possible without his help.

\section{Outline of the proof}

Given any scheme $S$, and $S-$schemes $X,T$, we denote by $X_T$ the base change $X \times_S T$ of $X$ to $T$. If $T = \textup{Spec} (k)$ for a field $k$, we sometimes write $X_T$ as $X_k$ instead.

\subsection{K3 surfaces}
A $K3$ surface $S$ is a smooth projective surface with $\Omega_S^2 \cong \mathcal O_S$ and $h_1(S, \mathcal O_S) = 0$. A primitively polarized $K3$ surface of genus $g$ is a pair $(S, \mathcal L)$, where $X$ is a $K3$ surface, and $\mathcal L$ is an indivisible, nef line bundle on $X$, such that $|\mathcal L|$ is without fixed component and $\mathcal L^2 = 2g - 2$ (hence $g \ge 1$). Given such a pair, $|\mathcal L|$ is base point free, and the morphism $\varphi_{|\mathcal L|}$ determined by this linear system is birational if and only if $\mathcal L^2 > 0$ and $|\mathcal L|$ does not contain any hyperelliptic curve (hence $g \ge 3$) (see \cite{SD}).

For any given non-negative integer $g$ one considers the moduli functor
$$B_g: \ (Sch/ \mathbb C)^{opp} \to (Sets), \ T \mapsto \{(f: X \to T, \mathcal L) \}$$
that sends a scheme $T$ of finite type over $\mathbb C$ to the set $B_g(T)$ of equivalence classes of pairs $(f,\mathcal L)$ with $f : X \to T$ a smooth proper morphism and $\mathcal L \in \textup{Pic}_{X/T} (T)$ such that for all geometric points $\mathrm{Spec}(k) \to T$, i.e. $k$ an algebraically closed field, the base change
yields a $K3$ surface $X_k$ with a primitive ample line bundle $\mathcal L_{X_k}$ such that $(\mathcal L_{X_k})^2 = 2g-2$, i.e. $(X_k,\mathcal L_k)$ is a primitively polarized $K3$ surface of genus $g$.

By definition, $(f,\mathcal L) \sim (f',\mathcal L')$ if there exists a T-isomorphism  $\psi : X \xrightarrow[]{~} X'$ and a line bundle $\mathcal L_0$ on $T$ such that $\psi^* \mathcal L_0 \cong  \mathcal L \otimes f^* \mathcal L_0$.

Then we have, 
\begin{prop}
\cite{PS} For every $g \ge 1$, the moduli functor $B_g$ can be coarsely represented
by an irreducible quasi-projective variety $\mathfrak B_g$ of dimension 19.  If $g \ge 3$, for $(X, \mathcal L)$ very general in $\mathfrak B_g$, the Picard group of $X$ is generated by the class of $\mathcal L$, and $\mathcal L$ is very ample.
\end{prop}

\begin{definition}[Rational curve]
We call a curve $C$ over a field $k$ rational if all irreducible components of $C_{\overline k}$ have geometric genus $0$ where $\overline k$ is an algebraic closure of $k$ and $C_{\overline k}$ is the base change of $C$ to $\textup{Spec} (\overline k)$.
\end{definition}

For any $(X,\mathcal L) \in \mathfrak B_g (\mathbb C)$, there are only finitely many rational curves given by sections of $L$. Let $n_g$ denote this number for a generic $(X,\mathcal L)$. Then we have Yau-Zaslow's formula proven by Beauville in \cite{Be} assuming a result later proven by Chen \cite{Ch}.

\begin{prop}[Yau-Zaslow's Formula] \label{YZF}
 $$\sum_{g \ge 0} n_gq^g = \frac q {\Delta(q)} = \prod_{n \ge 1} (1-q^n)^{-24} = 1 + 24 q + 324q^2 + 3200q^3 + 25650 q^4 + \ldots$$
\end{prop}

More specifically, Beauville showed that if we count curves with multiplicities, then $n_g$ will be the number of curves counted with multiplicity for any primitively polarized $(X,\mathcal L)$ of genus $g$. For any rational curve $C \in |\mathcal L|$, the multiplicity with which it is counted is $\prod_{x \in C} \epsilon(x)$. If $x$ is a smooth point or a node, $\epsilon(x) = 1$, and if $x$ is a  singular point with singularity of the form $x^p - y^q$ with $p,q$ coprime, then we have

\begin{prop} \label{Mult}
 (\cite{Be} Proposition 4.3) If $x$ is a  singular point with singularity of the form $x^p - y^q$ with $p,q$ coprime then $$\epsilon(x) = \frac{1}{p+q} \binom{p+q}{q}$$
\end{prop}

The result proven by Chen in \cite{Ch} is

\begin{prop} \label{chennodal}
\cite{Ch} For $(X,\mathcal L)$ very general in $\mathfrak B_g$, any rational curve in $|\mathcal L|$ is nodal. 
\end{prop}

\subsection{Monodromy of a generically finite map}

Throughout the section, let $X, Y$ be two algebraic varieties of the same dimension over $\mathbb C$ with $X$ irreducible, and $\pi : Y \to X$ a generically finite map of degree $d$. Let $p \in X$ be a generic point so that $\pi^{-1}(p)$ consists of $d$ distinct points $q_1,\cdots , q_d$.

\subsubsection{Monodromy group.} This is defined similar to monodromy groups arising from covering maps in many topological and geometric situations. Let $U \subset X$ be a sufficiently small Zariski open set so that $\pi$ is an unbranched covering map of degree $d$ restricted to $V = \pi^{-1}(U)$. We
may also assume $p \in U$. For any loop $\gamma : [0, 1] \to U$ based at $p$, and any lift $q_i$ of $p$, there exists a unique lift of $\gamma$, denoted by $\widetilde{\gamma_i} : [0, 1] \to V$ so that $\widetilde{\gamma_i}(0) = q_i$. The endpoint of $\widetilde{\gamma_i}$ is well defined up to homotopy of $\gamma$. Therefore, we have an action of $\pi_1(U, p)$ on the set $\{q_1, \cdots, q_d \}$ so that the equivalence class of homotopic loops $[\gamma] \in \pi_1(U, p)$ sends $q_i$ to the endpoint of the lifted arc $\widetilde{\gamma_i}$. With respect to the fixed numbering, this gives a homomorphism $\pi_1(U, p) \to \Sigma_d$, sometimes referred to as the monodromy representation. The image of this homomorphism is called the monodromy group of the covering map $\pi : V \to U$. Note that, a priori, the monodromy group may depend on the choice of the Zariski open subset $U$, although it actually does not.

The following lemmas which we reproduce from \cite{Ha} comes in handy to prove that the monodromy group is transitive and to show that there exists a simple transposition in the monodromy group:

\begin{lemma} \label{transitive}
Let $X, Y$ be two algebraic varieties of the same dimension over $\mathbb C$ with $X$ irreducible, and $\pi : Y \to X$ a generically finite map of degree $d$. If $Y$ has only one irreducible component of maximum dimension, then the monodromy group acts transitively on the fibre.
\end{lemma}
\begin{proof}
Let the irreducible component of maximum dimension of $Y$ be $Y_0$. Then $\pi|_{Y_0}: Y_0 \to X$ is also generically finite and $Y \setminus Y_0 \to X$ is not dominant. Therefore, the monodromy of $\pi$ is the same as the monodromy of $\pi|_{Y_0}$. So we may assume $Y$ itself is irreducible.

Let $U \subset X$ be a sufficiently small Zariski open set so that $\pi$ is an unbranched covering map of degree $d$ restricted to $V = \pi^{-1}(U)$. Let $p \in U$. $\pi^{-1}(\Delta)$ is open in $Y$, therefore connected since $Y$ is irreducible. So take a path $\widetilde{\gamma} : [0, 1] \to V$ connecting some two points $q_1,q_2 \in \pi^{-1}(p)$. The image $\pi(\widetilde{\gamma})$ will be a loop $\gamma$ whose equivalence class $[\gamma] \in \pi_1(U, p)$ will give an element of the monodromy group sending $q_1$ to $q_2$.  
\end{proof}

\begin{lemma} \label{simpletrans}
Let $\pi : Y \to X$ be a holomorphic map of degree $d$. Suppose there exists a point $p \in X$ such that the fiber of $Y$ over $p$ consists of exactly $d-1$ distinct points, i.e., simple points $q_1, . . . , q_{d-2}$ and a double point $q_{d-1} = q_d$. Suppose furthermore $Y$ is locally irreducible at $q_{d-1}$. Then the monodromy group $M$ of $\pi$ contains a simple transposition.
\end{lemma}
\begin{proof}
Take a small neighborhood $\Delta$ of $X$ such that $\pi^{-1}(\Delta)$ consists of $d- 1$ disjoint irreducible components $\Delta_i$, where $\Delta_i$ is a small neighborhood of $q_i$, and $\pi$ restricted to $\Delta_i$ gives an isometry for $i = 1, \cdots, d- 2$ and a double cover branched only at $q_{d-1}$ for $i = d- 1$. Take a point $\tilde p \in \Delta$ with $\tilde p \neq p$. Then $\pi^{-1}(\tilde p)$ consists of $d$ distinct points $\tilde q_i$ so that $\tilde q_i \in \Delta_i$ for $i = 1, \cdots , d - 2$, and $\tilde q_d,  \tilde q_{d-1} \in \Delta_{d-1}$. Let $\tilde \gamma$ be any arc connecting $\tilde q_d,  \tilde q_{d-1}$ in $\Delta_{d-1}$. Then the monodromy representation of $\gamma = \pi(\tilde \gamma)$ gives a simple transposition exchanging $d$ and $d - 1$.
\end{proof}

We follow the notation of the introduction. We assume $g =3$ or $4$ throughout. We recall from the introduction that we denote the monodromy group of $\pi: J_g \to W_g$ by $\Pi_g$ and that proving $\Pi_g \cong S_{n_g}$ is enough to prove that $\Gamma_g \cong S_{n_g}$.

\subsection{Transitivity}
As remarked before, to prove that $\Pi_g$ is transitive it suffices to prove that $J_g$ has only one irreducible component of maximum dimension. We prove this by studying the fibers of $\eta$. Let the fiber of $\eta$ over $H$ be $Y_H$. Then $Y_H$ has a map to the space $Z$ of rational curves of genus $g$ which are canonically embedded in $H \cong \mathbb P^{g-1}$ by taking $(S,H) \mapsto S \cap H$. Then we prove irreducibility results about the fiber of this map as well as $Z$.

\subsection{2-Transitivity}

Now, fix a hyperplane $H$ in $\mathbb P^g$ and a rational integral curve $C$ of genus $g$ canonically embedded in $H$. Assume that singularities of $C$ are either all nodal singularities or one simple cusp and others nodal.

Let $W' = \{ S \ | \ S \cap H = C \} \subseteq W$. 
 
 Let $J' = \{ (S,H')  \ | \ S \cap H'$ is integral rational, $H'\neq H, \ S\cap H = C \} \subseteq W' \times ((\mathbb P^g)^{\vee} - H)$. Let $\pi ' : J' \to W',  \ \eta ' : J' \to ((\mathbb P^g)^{\vee} - H)$ be the projection maps.
 Let the fiber of $\eta '$ above a fixed $H'$ be $T_{H'} = \{ S \ | \ S\cap H = C, \ S\cap H' \textup{ is integral rational}\}$.
 
 Now for any fixed finite scheme $Y \in \mathbb P^{g-1}$, let $R_Y$ be the space of genus $g$ integral rational curves $C'$ canonically embedded in $\mathbb P^{g-1}$ which contain $Y$. Then we have a map $\tau : T_{H'} \to R_{C \cap H'}$, which sends $S \mapsto S \cap H'$.

We will show that the monodromy of $\pi'$ is transitive, which will prove that $\Pi_g$ is 2-transitive. (Since we may take $C$ to be $S \cap H$ for a general $(S,H)$ and such a $C$ will be integral nodal)

First, we show that $W'$ contains a $K3$ surface if $C$ has only nodes or simple cusps. Next, we prove that the fibers of $\tau$ are all irreducible of the same dimension. Then we prove bounds on dimensions of $R_Y$ for certain $Y$ of length $2g-2$. Note that $Y$ being equal to $C \cap H'$ imposes certain restrictions on $Y$, for example if $g=4$, $Y$ must be the intersection of a conic with a cubic in $\mathbb P^2$, and (although we don't need this case) if $g=5$, $Y$ must be the intersection of three quadrics in $\mathbb P^3$. If $Y$ is the disjoint union of $2g-2$ general points (constrained by the restrictions as noted above), then we show that $R_Y$ is irreducible of dimension equal to the bound. Finally, since a general hyperplane section of $C$ will be $2g-2$ distinct points so as long as we are able to show that there are $2g-2$ general points (constrained by the restrictions as before) on $C$ we will be through.

\subsection{Simple transposition}

Following \cite{Ha} (see Lemma \ref{simpletrans}), to prove that $\Pi_g$ admits a simple transposition, it is enough to show that there is a point $S$ in $W_g$ such that the fiber of $\pi : J_g \to W_g$ above $S$ is $y_1,\cdots , y_n$ satisfying: 
\begin{enumerate}
\item $y_1$ corresponds to a rational curve having $g-1$ nodes and $1$ simple cusp and $y_2,\cdots , y_n$ are points corresponding to rational nodal curves. 
\item $n = \deg \pi - 1$. 
\item $J_g$ is locally irreducible at $y_1$. 
\end{enumerate}

In the proof of Harris\cite{Ha} where he establishes the existence of a simple transposition in the monodromy group of bitangents to a degree $d$ curve in $\mathbb P^2$ one studies flex bitangents to the curve. In our case, this corresponds to a rational curve with one simple cusp and rest of the singularities being nodes. Also, the Plucker formula gets replaced by the Yau-Zaslow formula as proven by Beauville \cite{Be}.

By Beauville's Yau-Zaslow formula (See Proposition \ref{YZF} and \ref{Mult}), we see that if property 1 is satisfied, then the curve $y_1 $ will be counted with multiplicity $\binom{5}{2}  /5 = 2$ and the rest of the curves being nodal curves will be counted with multiplicity $1$, which means that $n$ is exactly $1$ less than the number of rational curves in the linear system $|\mathcal O(1)|$ of a general surface, which is $\deg \pi$. 

Finally, if there exists a $K3$ surface in $W_g$ with a hyperplane section as a rational simple cuspidal curve $C$, then by using the transitivity of the monodromy of $\pi'$ for this $C$ (which we proved before) we will get an $S$ so that one hyperplane section is $C$ and the others are rational nodal. This concludes the outline of the proof.

\section{Preliminaries}

\subsection{A space of maps with specified images}

\begin{definition}[Immersion]
Let $k$ be a field, $C$ be a curve over $k$, and $X,T$ be schemes over $k$. Let $\phi: C_T \to X_T$ be a morphism, and $D \subset C_T$ be a relative effective divisor. We say $\phi$ is an immersion at $D$, if for every geometric point $t$ of $T$, the map $\phi_t : C_t \to X_t$ induced on the fiber at $t$ is an immersion at the points of $D_t = D \cap C_t$ i.e. the map on tangent spaces at those points is injective.
\end{definition}

\begin{definition}[Normal Sheaf]
Let $k$ be a field, $C$ be a curve over $k$, and $X$ be a scheme over $k$. Let $p_1, \cdots, p_n$ be points of $C$. Let $g: C \to X$ be a morphism which is an immersion $p_i$ if $p_i$ is a smooth point of $C$, and locally an embedding at $p_i$ if $p_i$ is a singular point of $C$. The normal sheaf $N_g$ of $(g,p_1, \cdots, p_n)$ is then defined to be the dual of $\mathscr I_C/\mathscr I_C^2$ in a neighbourhood of $p_i$ of the $p_i$ which are singular points and is the the cokernel of $T_C \to g^*T_X$ elsewhere. Here $\mathscr I_C$ is the ideal sheaf of the image $g(C)$ of $C$ in $X$. 

Suppose we have divisors $D_1, \cdots, D_n$ of $C$ so that each of them is supported at a single point. If $D_i$ is supported at $p_i$, then the normal sheaf of $(g, D_1, \cdots, D_n)$ is defined to be the normal sheaf of $(g, p_1, \cdots, p_n)$ as above. 
\end{definition}

\begin{lemma} \label{mainlemma}
Let $k$ be a field, $C$ be a local complete intersection projective curve/k, $X$ be another scheme/k, and for $i=1, \cdots, m$ we have $Y_1, \cdots, Y_m$ finite connected curvilinear subschemes of $X$ of length $n_1, \cdots, n_m$ supported at (not-necessarily distinct) points $y_1, \cdots, y_m$ respectively. Let $Y$ be a finite subscheme of $X$ supported at $y_1, \cdots, y_m$. 
\smallskip

Consider the functor $F = F(C,X,Y_1, \cdots, Y_m) :$ Schemes/k $\to$ Sets given by 
\begin{equation*}
F(T) = \left \{ (\phi, D_1, \cdots, D_m) : \  
\begin{aligned}
& \phi : C_T \to X_T \textup{ is flat.} \ \textup{For all } i, \ D_i \subset C_T \textup{ is a degree } n_i \textup{ relative} \\ & \textup{effective cartier divisor with } \phi(D_i) = (Y_i)_T, \  \phi|_{D_i}: D_i \to (Y_i)_T
\\ & \textup{is an isomorphism, and } \phi \textup{ is an immersion at } D_i.
\end{aligned} \right \}
\end{equation*} 


We also consider the functor $\tilde F = \tilde F(C,X,Y,n_1, \cdots, n_m) :$ Schemes/k $\to$ Sets given by 
\begin{equation*} 
\tilde F(T) = \left \{ (\phi, D_1, \cdots, D_m, i_1, \cdots, i_m) : \ 
\begin{aligned} 
& \phi : C_T \to X_T \textup{ is flat.} \ \textup{For all } j, \\ 
& i_j: T \times \textup{Spec}(k[x]/(x^{n_j})) \to Y_T \textup{ is an embedding with}
\\ & \textup{image } Z_j, \ D_j \subset C \textup{ is a degree } n_j \textup{ relative effective}
\\ &  \textup{cartier divisor with } \phi(D_j) = Z_j, \ \phi|_{D_j}: D_j \to Z_j \textup{ is}
\\ &  \textup{an isomorphism, and } \phi \textup{ is an immersion at } D_j. 
\end{aligned} \right \}
\end{equation*} 


Assume now that $X$ is a projective surface which is smooth at the points $y_1, \cdots, y_m$. Then the functors $ F, \tilde F$ are representable. If $X$ is a projective surface, then for a point $(g,D_1, \cdots, D_m) \in F(k)$ so that $X$ is smooth at the points of $g(C)$, $g$ is an immersion at the smooth points of $C$, and locally an embedding at singular points of $C$, the tangent space to $F$ at $(g,D_1, \cdots, D_m) \in F(k)$ is given by the inverse image of $H^0(C, \mathscr I_{D_1} \cdots \mathscr I_{D_m} N_g)$ under the map $H^0(C, g^* T_X) \to H^0(C, N_g)$, where $N_g$ is the normal sheaf of $(g,D_1, \cdots, D_m)$. 

\end{lemma}
\begin{proof}
For any projective $k-$schemes $S,S'$, let $\text{Hom}(S',S)$ and $\text{Em}(S',S)$ be the Hom-scheme of $S'$ to $S$ and the subscheme of the Hom-scheme consisting of embeddings of $S'$ into $S$, respectively. Let $T_n = \text{Spec}(k[x]/(x^n))$.

Then $F$ is a subfunctor of 
$$\text{Hom}(C,X) \times \text{Em}(Y_1, C) \times \cdots \times \text{Em}(Y_m, C).$$
The subfunctor of $\text{Hom}(C,X) \times \text{Em}(Y_i, C)$ of $(\phi, \varphi)$ such that $\phi$ is an immersion at the point in the image of $\varphi$ is an open subscheme, and the subfunctor of $\text{Hom}(C,X) \times \text{Em}(Y_i, C)$ of $(\phi,\varphi)$ with $\phi \circ \varphi$ is a fixed morphism (corresponding to $Y_i \hookrightarrow X$) is a closed subscheme. Thus, $F$ is representable.
Similarly, $\tilde F$ is a subfunctor of 
$$\text{Hom}(C,X) \times  \text{Em}(T_{n_1}, C) \times \cdots \times \text{Em}(T_{n_m}, C) \times \text{Em}_{y_1}(T_{n_1}, Y) \times \cdots \times \text{Em}_{y_m}(T_{n_m}, Y) $$
where $\text{Em}_{y}(T_n, Y)$ is the space of embeddings supported at $y$. As before, the subfunctor consisting of $(\phi, D_1, \cdots, D_m, i_1, \cdots, i_m)$ such that $\phi$ is an immersion at $D_i$ for every $i$ is an open subscheme. Also, the subfunctor of $\text{Hom}(C,X) \times \text{Em}(T_n, C) \times \text{Em}(T_n, C)$ of $(\phi,\varphi,i)$ with $\phi \circ \varphi = i$ is a closed subscheme. Thus, $\tilde F$ is representable.

Every element of the tangent space of $F$ at $(g,D_1, \cdots, D_m)$ may be seen as an element of the tangent space $H^0(C,g^* T_X)$ of Hom$(C,X)$ at $g$ because over $k[\varepsilon]/ \varepsilon^2$, if we fix a morphism $g'$ which equals $g$ over the special fiber then the divisors $D_i'$ which will restrict to $D_i$ over the special fiber and map isomophically onto $Y_i$ will be uniquely determined by $g'$ (if they exist) since $g'$ is an immersion at the points of $D_i$. So we need to only check that the elements of tangent space of $F'$ are exactly the set of elements of $H^0(C,g^* T_X)$ which map to $H^0(C,\mathscr I_{D_1} \cdots \mathscr I_{D_m} N_g)$. We need only check this at the point $y_i$ of $Y_i$.

Since the finite subschemes of a scheme which are supported at a point $p$ are 1-1 correspondence with the finite subschemes of the Spec of the completion of the local ring at $p$, we may argue at the level of completions. We may choose formal coordinates so that the map corresponding to $g$ at the completion level is $\alpha : k \llbracket x,y \rrbracket \to k \llbracket x,y \rrbracket/(f)$. Note that here we are using the fact that $g$ is an immersion at the smooth points of $C$, and locally an embedding at singular points of $C$, and that $C$ is an lci curve. The ideal sheaf of our divisor will be given by an ideal $I$ which contains $f$.

Then we need to consider formal power series of the form $f + \varepsilon f_1$ so that $f + \varepsilon f_1$ is contained in the extended ideal $I^e$ . This happens if and only if $f_1 \in I$, which implies that the element of $H^0(g^*T_X)$ corresponding to $f+ \varepsilon f_1$ goes to an element of $H^0(\mathscr I_{D_i} N_g)$.

This local calculation remains unchanged for $F'$ so the same calculation proves it in that case also.
\end{proof}

We will also need a calculation of the ways to embed $\mathrm{Spec}(k[x]/(x^n))$ in a curve at a specified point:

\begin{lemma} \label{embedcalc}
Let $C$ be an integral curve inside a surface $X$. Let $P \in C$, and suppose that $P$ is a smooth point of $X$. Let $D$ be the finite curvilinear connected scheme of length $n$, i.e. $D \cong \mathrm{Spec}(k[x]/(x^n))$. Consider the scheme $\textup{Em}_P(D,C)$ parametrizing the embeddings $\psi: D \hookrightarrow C$ so that $\psi(D)$ is supported at $P$. Then we have

1. If $P$ is a smooth point of $C$ then $\dim \textup{Em}_P(D,C) = 0$.

2. If $P$ is a node of $C$ then $\dim \textup{Em}_P(D,C) = 0$ if $n =1$, $\dim \textup{Em}_P(D,C) = 1$ if $n \ge 2$.

3. If $P$ is a simple cusp of $C$ then there is no embedding if $n \ge 4$, $\dim \textup{Em}_P(D,C) = 0$ if $n =3$, $\dim \textup{Em}_P(D,C) = 1$ if $n = 2$, $\dim \textup{Em}_P(D,C) = 0$ if $n =1$.
\end{lemma}
\begin{proof}
We may do a local calculation at the point $P$. Note that our embeddings are all contained in the curve $C$ inside the surface $X$ and $P$ is a  smooth point of $X$.

1. Smooth point: There is only a unique subscheme of $C$ of length $n$ supported at $P$, so 0-dimensional.

2. Node: We can do a local calculation at the point. Say we have a length $n$ curvilinear subscheme corresponding to $I \subset k \llbracket x,y \rrbracket$ such that $xy \in I$.

Since $I$ is curvilinear and has colength $n$, we may assume $I = (f,g)$ where $f$ has non-zero linear term and all terms of $g$ are deg $\ge n$.

Now $xy \in I$, so the linear term of $f$ has to divide $xy$ (Assuming $n \ge 2$). Without loss of generality, let the linear term of $f$ be $x$.

Thus we have $f = x+f_1$, and so we may assume (after repeated substitution of $-f_1$ for $x$ in $f,g$) that $g = y^n$ and $f = x + f_1(y)$. Also we may assume degree of $f_1$ is $\le n-1$.

Now $af + bg = xy$ for some $a,b \in k \llbracket x,y \rrbracket$. We know that deg $g \ge 2$ and $g$ has no term involving $x$ therefore $a = y$. And hence $f_1 = c \cdot y^{n-1}$ for some constant $c$.

Thus we get a one dimensional space for $n \ge 2$.

3. Cusp: Let $n \ge 3$. Similar reasoning as above allows us to assume $I = (f,g)$, $x^2 - y^3 \in I$. In this case the linear term of $f$ divides $x^2$, so again we may assume $f = x + f_1(y)$ and $g = y^n$. 

Now, $af + bg = x^2 - y^3$ implies $a =x, f_1(y) = 0 , n=3, b = -1$.

Hence for $n \ge 4$ no embedding, $n =3$ zero dimensional space, $n = 2$ one dimensional space.
\end{proof}

\begin{lemma} \label{multiplicity}
Let $C$ be a curve in $\mathbb P^n$, with $p \in C$ being a smooth point. Let $L$ be the tangent line to $C$ at $p$. Then for a general hyperplane $H$ containing $L$, $\textup{mult}_p(H \cap C) = \textup{mult}_p(L \cap C)$.
\end{lemma}
\begin{proof}
We may work in the completion of the local rings of $C$ and $\mathbb P^n$ at $p$. The latter is $k \llbracket x_1, \cdots, x_n \rrbracket$. Hyperplanes containing $p$ correspond to linear polynomials in the $x_i$.

Let $C$ be given by the vanishing of polynomials $f_1, \cdots, f_{n-1}$. $p$ is a smooth point, so we have that the linear terms of $f_1, \cdots, f_{n-1}$ are linearly independent. Thus, we may choose coordinates $x_1, \cdots, x_n$ so that the linear term of $f_i$ is $x_i$ for $i=1, \cdots, n-1$. Then the tangent line $L$ is given by $x_1=0, \cdots, x_{n-1} = 0$. The hyperplanes containing $L$ correspond to linear polynomials in $x_1, \cdots, x_{n-1}$ i.e. the coefficient of $x_n$ is $0$.

Now, we note that there exist $g_1, \cdots, g_{n-1}$ of the form $g_i = x_i + h_i(x_n)$ such that the ideal generated by them $(f_1, \cdots, f_{n-1}) = (g_1, \cdots, g_{n-1})$ are the same. To see this, note that $f_i = x_i - f_{i,1}(x_1, \cdots, x_{n-1})$, where $f_{i,1}$ has degree $\ge 2$, so we may substitute $x_j$ for $f_{j,1}$ repeatedly in $f_{i,1}$ to get the desired power series $h_i(x_n)$ in $x_n$.

Now to finish off the proof, it suffices to notice that 
\begin{align*}
\textup{mult}_p(L \cap C) &= \textup{length} \left( \frac{ k \llbracket x_1, \cdots, x_n \rrbracket} {(x_1, \cdots, x_{n-1}, f_1, \cdots, f_{n-1})} \right) \\ &= \textup{length} \left( \frac{k \llbracket x_1, \cdots, x_n \rrbracket} {(x_1, \cdots, x_{n-1},g_1, \cdots, g_{n-1})} \right) \\ &= \textup{length} \left( \frac{k \llbracket x_n \rrbracket }{(h_1, \cdots, h_{n-1})} \right) \\ &= \min_i (\deg h_i).
\end{align*}
and for a hyperplane $H$ given by $ a_1 x_1+ \cdots + a_{n-1} x_{n-1}$,
\begin{align*}
\textup{mult}_p(H \cap C) &= \textup{length} \left(\frac{ k \llbracket x_1, \cdots, x_n \rrbracket} {( a_1 x_1+ \cdots + a_{n-1} x_{n-1}, f_1, \cdots, f_{n-1})} \right) \\ &= \textup{length} \left(\frac{k \llbracket x_1, \cdots, x_n \rrbracket} {( a_1 x_1+ \cdots + a_{n-1} x_{n-1},g_1, \cdots, g_{n-1})} \right) \\ &= \textup{length} \left(\frac{k \llbracket x_n \rrbracket }{(a_1h_1+ \cdots+ a_n h_{n-1})} \right) \\ &= \deg (a_1h_1 + \cdots + a_n h_n).
\end{align*}
For general $a_1, \cdots, a_{n-1}$, $ \min_i(\deg h_i) = \deg (a_1h_1 + \cdots + a_n h_n)$, and this finishes the proof.
\end{proof}

\begin{lemma} \label{locallengthcalc}
Let $C$ be a curve in $\mathbb P^n$, and let $p \in C$ be a singularity of $C$ which is a simple node or a simple cusp. Then the general hyperplane of $\mathbb P^n$ passing through $p$ intersects $C$ in a scheme of length $2$ at $p$.
\end{lemma}
\begin{proof}
Let $A_1 = k \llbracket x,y \rrbracket/ (xy)$ and $A_2 = k \llbracket x,y \rrbracket /(x^3 - y^2)$. We first prove that $A_i/(f)$ has length $2$ if $f = ax + by +$ higher terms with $a,b \neq 0$. 

$i=1$: Multiplying $f$ by $x^n$, we have $$ax^{1+n} = a'x^{2+n} + \textup{ higher terms} \pmod{(f,xy)}$$ for all $n \ge 1$, so we have that $x^2 = 0 \pmod{(f,xy)}$. Similarly, $y^2 = 0 \pmod{(f,xy)}$. So, $A_1/(f) = k\llbracket x,y \rrbracket/(ax+by, x^2, y^2, xy)$ is of length $2$.

$i=2$: Let $c = -b/a$. Then $x = cy + \textup{ higher terms} \pmod{(f,y^2 - x^3)}$. So, $$x^2 = c^2y^2 = c^2x^3 = 0 \pmod{(f,y^2 - x^3,(x,y)^3)}.$$ Similarly,  $$xy = cy^2 = cx^3 = 0 \pmod{(f,y^2 - x^3,(x,y)^3)}.$$ Therefore, $m^2 = m^3$ where $m$ is the maximal ideal of $A_2/(f)$. So, $m^2 =0$ by Nakayama Lemma. Hence, $A_1/(f) = k\llbracket x,y \rrbracket/(ax+by, x^2, y^2, xy)$ is of length $2$.

Coming back to the main problem, we take the completion of the local ring at $p$ to get $$\widehat{\mathcal O_p} = k \llbracket x_1, \cdots, x_n \rrbracket/ (f_1, \cdots, f_{n-1}) \cong A_i$$ for $i=1$ or $i=2$. Note that here the $x_1, \cdots, x_n$ denote the coordinates of an affine space $\mathbb A^n$ containing $p$. By the above calculation it is enough to note that the isomorphism on $m/m^2$ for the two rings sends a general linear polynomial on the left to a general linear polynomial on the right.
\end{proof}

\subsection{A space of rational curves of genus $g$}

\begin{definition}[Canonical map]
Let $C$ be a Gorenstein projective curve of genus $g$ over a field $k$, and let $\omega$ denote its canonical sheaf. Then $\omega$ is invertible and globally generated and hence gives a map
$$\kappa : C \to \mathbb P^{g-1}$$
which we call as the \textit{canonical map} of $C$.
\end{definition}

\begin{definition}[Canonically embedded curve]
Let $C$ be a projective curve over a field $k$ embedded in $\mathbb P^n$. $C$ is said to be \textit{canonically embedded in $\mathbb P^n$} if $C$ is Gorenstein, with canonical sheaf $\omega$ satisfying $\mathcal O_{\mathbb P^n}(1)|_C = \omega$, and $i: C \hookrightarrow \mathbb P^n$ is the canonical map. (in particular, this means that the genus of $C$ is $n+1$)
\end{definition}

Let $g \ge 3$. Let $C$ be a integral rational curve of genus $g$ having only nodes or simple cusps as singularities, and consider a normalization map $\alpha : \mathbb P^1 \to C$. Let $P_1, \cdots, P_g$ be the $g$ singular points of $C$, and let $V_C = \coprod \alpha^{-1}(P_i)$. Then $V_C$ is a union of $g$ length-2 schemes which are pairwise disjoint. Now, $V_C$ determines $C$ up to isomorphism, and two subschemes $V,V'$ determine isomorphic $C$ if and only if there is an automorphism of $\mathbb P^1$ which takes $V$ to $V'$.

Consider the scheme $(\mathbb P^1)^{2g}$ and consider the open set $X_0 \subset (\mathbb P^1)^{2g}$ corresponding to 2g-tuples $(p_1, \cdots, p_{2g})$ of points so that no three points are the same, and if two points $p_i$ and $p_j$ are the same with $i<j$ then $i$ is odd and $j = i+1$. Note that $X_0$ is irreducible since it an open set of $(\mathbb P^1)^{2g}$. Also, for $1 \le c \le g$, let $X_c$ be the closed subset of $X_0$ where $c$ pairs of points are equal.

Now, consider the projection map $\Lambda: \mathbb P^1 \times (\mathbb P^1)^{2g} \to (\mathbb P^1)^{2g}$. Then $\Lambda$ has $2g$ sections $\alpha_1, \cdots, \alpha_{2g}$ corresponding to the coordinates of the points. We glue the pairs of sections $(\alpha_1(X_0), \alpha_2(X_0))$, $(\alpha_3(X_0), \alpha_4(X_0))$, $\cdots ,(\alpha_{2g-1}(X_0), \alpha_{2g}(X_0))$ to get a scheme $Y_0$ so that $\Lambda$ factors as $\beta \circ \lambda$:

$$\mathbb P^1 \times (\mathbb P^1)^{2g} \xrightarrow[]{\lambda_0}  Y_0 \xrightarrow[]{\beta} (\mathbb P^1)^{2g}$$

Now, note that each fibre of $\beta$, $C = \beta^{-1}(x)$ comes equipped with a normalization map $\lambda_x: \Lambda^{-1}(x) = \mathbb P^1 \to C$. The fibres of $\beta$ over $X_0$ will be rational curves of genus $g$ having only nodes or simple cusps as singularities, and every rational curve of genus $g$ having only nodes or simple cusps as singularities will occur as a fibre over the point of $X_0$ corresponding to the points of $V_C$ (with a point occurring twice when it is occurring with multiplicity two in $V_C$).

Consider the subgroup $\Omega$ of the permutation group generated by elements of the form $(2i-1 \ 2i)$ and $(2i-1 \ 2j-1) \cdot (2i \ 2j)$ for $1 \le i,j \le g$. Then we have that $\Omega$ acts on $(\mathbb P^1)^{2g}$ by permuting the elements and acts on $\mathbb P^1 \times (\mathbb P^1)^{2g}$ by acting on the second coordinate, and the map $\Lambda$ is $\Omega$-equivariant. The action of $\Omega$ restricts to an action on $X_0$. Also the diagonal actions of $\mathrm{Aut}(\mathbb P^1)$ on $(\mathbb P^1)^{2g}$ and $\mathbb P^1 \times (\mathbb P^1)^{2g}$ make the map $\Lambda$ to be $\mathrm{Aut}(\mathbb P^1)$-equivariant. So we have an action of $G = \Omega \times \mathrm{Aut}(\mathbb P^1)$ on $(\mathbb P^1)^{2g}$ and $\mathbb P^1 \times (\mathbb P^1)^{2g}$ so that $\Lambda$ is $G$-equivariant. Finally, we note that the action of $G$ on $(\mathbb P^1)^{2g}$ restricts to an action on $X_0$, and that $Y_0$ is obtained as the pushout of the maps $\alpha_{even}, \alpha_{odd}: X_0 \coprod \cdots \coprod X_0 \to \mathbb P^1 \times (\mathbb P^1)^{2g}$ where $\alpha_{even}$ corresponds to the even sections and $\alpha_{odd}$ corresponds to the odd sections. These maps are interchanged or mapped to themselves by $G$. Thus we get an action of $G$ on $Y_0$ so that $\beta$ is $G$-equivariant.

Two fibres $\beta^{-1}(x), \beta^{-1}(x')$ are isomorphic for $x, x' \in X_0$ iff there is an element of $G$ mapping $x$ to $x'$. Then let $X_0'$ be the open set of $X_0$ of $2g-$tuples such that no non-identity automorphism of $\mathbb P^1$ fixes it. Then $\mathrm{Aut}(\mathbb P^1)$ acts freely on $X_0'$. Also note that no non-identity automorphism of $\mathbb P^1$ fixes $4$ general points of $\mathbb P^1$, so there will be points in $X_0'$ corresponding to curves with $c$ simple cusps for every $0 \le c \le g$ if $g \ge 4$ and for every $0 \le c \le 2$ if $g =3$. Thus, if we define $X_c' = X_0' \cap X_c$ then it will be a non-empty open subset of $X_c$ for every $0 \le c \le g$ if $g \ge 4$ and for every $0 \le c \le 2$ if $g =3$.

Recall that a curve over a field $k$ is called a \emph{hyperelliptic curve} if there is a degree-2 morphism $\lambda : C \to \mathbb P^1$. If there is no such map we call it \emph{non-hyperelliptic}. 

We want to focus attention to points of $X_0'$ which represent non-hyperelliptic curves. Let $\mathfrak H_g$ be the subvariety of $X_0$ corresponding to points with fibers as hyperelliptic curves. Then we have the following lemma:
\begin{lemma} \label{hyper}
Let $g \ge 3$. We have:
\begin{enumerate}

    \item $\dim \mathfrak H_g = g+1$.
    \item $\dim \mathfrak H_g \cap X_c$ is empty for $c \ge 3$, $\dim \mathfrak H_g \cap X_1 = g$ and $\dim \mathfrak H_g \cap X_2 = g-1$.
\end{enumerate}

\end{lemma}

\begin{proof}
 Let $C'$ and $C''$ be two copies of $\mathbb P^1$. Any integral hyperelliptic curve $C$ admits a degree $2$ map $\beta: C \to C'' = \mathbb P^1$. If $C$ is moreover rational then we may compose $\beta$ by the normalization map $\alpha : C' = \mathbb P^1 \to C$ to get a degree $2$ map $\gamma: C' = \mathbb P^1 \to C'' = \mathbb P^1$. Then $\gamma$ will be ramified at $2$ points by Riemann-Hurwitz Formula. 
 
 The map $\gamma$ will factor as $\beta \circ \alpha$ for some nodal $C$ if and only if the $2g$ points of $C'$ we get as the inverse images of the singular points of $C$ are obtained as the inverse image of $g$ points of $C''$ via $\gamma$. 
 
 Also, $\gamma$ will factor as $\beta \circ \alpha$ for some $C$ which has $c$ simple cusps if and only if the $2g-2c$ points of $C'$ we get as the inverse images of the nodes of $C$ are obtained as the inverse image of $g-c$ points of $C''$ via $\gamma$, and the points of $C'$ corresponding to the cusps of $C$ are points of ramification. Thus, $c$ must satisfy $c \le 2$.

Thus, the subspace of $\mathfrak H_g$ corresponding to nodal $C$ may be identified with $g+1$ points in $\mathbb P^1$ and hence $g+1$ dimensional. The subspace of $\mathfrak H_g$ corresponding to $C$ having one (resp. two) simple cusps may be identified with $g$ (resp. $g-1$) points in $\mathbb P^1$, and hence $g$ (resp. $g-1$) dimensional.

\end{proof}

So let $X = X_0' \setminus \mathfrak H_g$ be the points of $X_0'$ where fibers of $\beta$ are non-hyperelliptic. Then from Lemma \ref{hyper}, we get that $X$ has dimension $2g$ and contains points corresponding to curves with $c$ simple cusps for every $0 \le c \le g$ for $g \ge 4$ and for every $0 \le c \le 2$ if $g = 3$. Let $Y = \beta^{-1}(X)$ be the pre-image of $X$ under $\beta$.

The action of $G  = \Omega \times \mathrm{Aut}(\mathbb P^1)$ restricts to an action on $X$. As seen before, $\mathrm{Aut}(\mathbb P^1)$ acts freely on $X$. Also, $\Omega$ acts freely on the set of tuples with $2g$ distinct points, and the points fixed by some element of $\Omega$ are the tuples with two points repeated. 

\smallskip

Let $B$ be the scheme representing the functor $\mathrm{Isom}_{X}(P, \mathbb P^g)$ where $P$ is the projective bundle associated
to the vector bundle $\beta_* \omega_{Y}$, where $\omega_{Y}$ is the relative dualizing line bundle. Then $B$ is a smooth scheme which is a $PGL_{g}$-bundle over $X$ with $\mathbb C$-points given by $(x,\varphi)$ where $x \in X(\mathbb C)$, and isomorphisms $ \varphi: \mathbb P(H^0(C, \omega_C)) \to \mathbb P^{g-1}$ where $C = \beta^{-1}(x)$ is the rational curve of genus g corrresponding to $x \in X(\mathbb C)$. 

The action of $G$ on $X$ and $Y$ makes $\beta$ a $G$-equivariant map, and hence induces an action of $G$ on $P$ which induces an action of $G$ on $B$. 



Note that every $C$ corresponding to a point of $X$ is non-hyperelliptic and Gorenstein, so the canonical map $\kappa_{|\omega|}: C \to \mathbb P(H^0(C, \omega_C))$ is an embedding by \cite{Kl}. So, we get a map from $B$ to the Hilbert scheme of curves in $\mathbb P^{g-1}$. The image is of course contained in the open set $\widetilde Z$ of the Hilbert scheme corresponding to geometrically integral rational curves having only nodes or simple cusps as singularities. Thus we have a map $\Phi_0 : B \to \widetilde Z$.

The image of $\Phi_0$ is closed: To prove this, we only need to prove that it is stable under specializations. So it is enough to prove that if you have a DVR $T$ and a curve over it (embedded in $\mathbb P^{g-1}$) such that the fibres are geometrically integral, the generic fibre is rational with only nodes and simple cusps and is canonically embedded and the special fibre also only has nodes or simple cusps then the special fibre is also canonically embedded (and rational). Of course, rational is clear, so the main point is to prove canonically embedded. We have that the special fibre is (geometrically) irreducible. So, the Picard scheme of the family exists and is separated (see Theorem 5 on p.212 of \cite{BLR}), so the canonical bundle can only specialize to the canonical bundle.

So, let $Z_g$ be the scheme theoretic image of $B$ under $\Phi$. Then $Z_g$ is a closed subscheme of $\widetilde Z$, and every point of $Z_g$ corresponds to a canonically embedded rational curve in $\mathbb P^{g-1}$ with nodes or cusps. So we have the restriction map $\Phi : B \to Z_g$. $Z_g$ is irreducible since $B$ is irreducible. The fiber of $\Phi$ over a point of $Z_g$ corresponding to the curve $C$ will be given by the $G$-orbit of the point in $B$ corresponding to the point of $X$ corresponding to $C$ together with isomorphisms $\mathbb P(H^0(C, \omega_C)) \to \mathbb P^{g-1}$, thus the dimension of $Z_g$ is $g^2 + 2g - 4$.

Let $U_1$ be the irreducible component of $X_1$ where $p_1 = p_2$, let $X_{(1)}$ be the intersection $X_1 \cap X$ and let $U_{(1)}$ be the intersection $U_1 \cap X$. Then $U_{(1)}$ and $X_{(1)}$ have codimension 1 in $X$. If $q:B \to X$ denotes the natural map from $B$ to $X$, then let $B_{(1)}$ be the inverse image of $X_{(1)}$ under $q$ and let $B_U$ be the inverse image of $U_{(1)}$ under $q$. Then $B_U$ is also irreducible and therefore the image of $B_U$ under $\Phi$ which we denote by $Z_{g-1,1}$ is also irreducible. Note that every point of $Z_{g-1,1}$ corresponds to a canonically embedded rational curve in $\mathbb P^{g-1}$ with nodes or cusps and at least one cusp. Therefore the inverse image of $Z_{g-1,1}$ is equal to $B_{(1)}$. Now, $B_{(1)}$ has codimension 1 in $B$ so we will have the dimension of $Z_{g-1,1}$ as $g^2 + 2g -5$.

$\Phi$ factors as $\phi \circ t$:
$B \xrightarrow{t} B/ \Omega \xrightarrow{\phi} Z_g$.
$B$ is smooth, hence normal which implies that $B/\Omega$ is normal. Also, any fibre of $\phi$ is isomorphic to $PGL_2$, since the fiber of $\Phi$ was seen to be exactly a $G$-orbit. Thus, they are geometrically irreducible.

We recall the definition of \emph{weakly-proper morphism} from \cite{Ko} (where it is stated as morphisms satisfying the weak lifting property for DVRs). A map $f : X \to Y$ of noetherian schemes is called a \emph{weakly-proper morphism} if given any DVR $T$ a morphism $\textup{Spec}(T) \to Y$ there exists a DVR $T'$, a dominant morphism from $\textup{Spec}(T')$ to $\textup{Spec}(T)$ and a morphism from $\textup{Spec}(T')$ to $X$ so that the following diagram commutes
\[ \begin{tikzcd}
\textup{Spec}(T') \arrow{r}{\rho} \arrow[swap]{d}{u} & X \arrow{d}{f} \\%
\textup{Spec}(T) \arrow{r}{\varphi}& Y
\end{tikzcd}
\]





\begin{lemma}
The map $\phi: B/\Omega \to Z_g$ is weakly proper.
\end{lemma}
\begin{proof}
We will first prove that if we have a curve over $\textup{Spec}(T)$ where $T$ is a DVR which is flat with integral rational fibres having only nodes or cusps as singularities then up to replacing $T$ by a finite extension $T'$ the curve can be constructed by considering $\mathbb P^1$ over $T'$ and gluing 2g sections appropriately.

To prove this, we first check that the normalisation of any such curve (over $T$) is isomorphic to $\mathbb P^1$ after a finite base change. Since the fibres are rational this amounts to proving that the normalisation is smooth.

Checking smoothness is a local question, so we may assume that $T$ is
$k[[t]]$ and the $g$ singularities correspond to sections: in general
the singularities need not correspond to rational points but
after a base change we can assume this. Now we need to understand the structure of the singularity locally. There are two cases, where the singularity on the special fibre has a node or a cusp.

In the nodal case, the local structure is given by $y^2 = (x -a(t))^2$ for some $a(t) \in T = k[[t]]$ (because if there is a node on the special fibre then there must also be a node on the generic fibre). By changing coordinates this
just becomes $y^2 = x^2$ and the normalisation of $T[x,y]/(y^2 -x^2)$ is isomorphic to $T[z] \times T[z]$ which is smooth.

In the cuspidal case, the local structure will be given by $y^2 =  x^3 + a(t)x + b(t)$ where $a$ and $b$ in $k[[t]]$ have $0$ constant term. Furthermore, we have that the discriminant $4a^3 + 27b^2$ is zero since there is also a singularity on the generic fibre. So, there exists $c(t)$ in $k[[t]]$ so that $a = -3c^2, b=2c^3$ because comparing valuations, we have that $3v(a) = 2v(b)$, and then it suffices to observe that if two units $u,u'$ satisfy $u^3 = u'^2$ then $u = (u'/u)^2, u' = (u'/u)^3$.

Now, $x^3 - 3c^2 x + 2c^3$ factorizes as $(x-c)^2 (x+2c)$ so we may replace $x-c$ by $x$ and consider the normalization of $T[x,y]/(y^2 = x^2(x-a))$ where $T = k[[t]]$. This has normalization $T[y/x, x]$, with $y/x$ and $x$ satisfying $(y/x)^2 = x-a$ which is isomorphic to $T[z]$ and hence smooth.

Now, any element of $Z_g(T)$ corresponds to the data of a curve $\tau: C \to \mathrm{Spec}(T)$ over $\mathrm{Spec}(T)$ which is flat with integral rational fibres having only nodes or cusps as singularities, and a canonical embedding of the curve in $\mathbb P^{g-1}_T$. We proved above that after going to a finite extension $T'$ the curve can be constructed by considering $\mathbb P^1$ over $T'$ and gluing 2g sections. These sections are well defined only up to an action of $\Omega$, so we get a point $\overline x$ of $(X/ \Omega)_T$ corresponding to these $2g$ sections. The data of the canonical embedding of the curve in $\mathbb P^{g-1}_T$ gives an isomorphism $\varphi: \mathbb P(\omega_C) \to \mathbb P^{g-1}$ where $\omega_C$ is the relative dualizing bundle of $C$ over $T$. So, we get a lift to $B/ \Omega$ corresponding to $\overline x$ and $\varphi$.
\end{proof}

\begin{lemma} \label{weakprop}
Let $f:X \to Y$ be a weakly proper, surjective map of finite type $k$-schemes, and suppose that $X$ is locally irreducible and the fibers of $f$ are geometrically connected. Then $Y$ is locally irreducible.
\end{lemma}
\begin{proof}
We will use the following fact
\cite[\href{https://stacks.math.columbia.edu/tag/054F}{Tag 054F}]{stacks-project}: Let $A$ be a noetherian local domain, then there exists a map $A \to R$ where $R$ is a DVR such that the closed point of
$\textup{Spec}(R)$ maps to the closed point of $\textup{Spec}(A)$ and the
same for the generic points.

Let $X'$ be the normalization of $X$ and $Y'$ be the normalization of $Y$. Then the map $f: X \to Y$ induces a map $f' : X' \to Y'$. To prove that $Y$ is locally irreducible it suffices to show that $Y' \to Y$ is a bijection. Note that $f$ is weakly proper implies $f'$ is also weakly proper since any map from $\textup{Spec}(R) \to X$ factors through $X'$.

First observe that $f'$ is surjective. This is because if $y' \in Y$ then there exists $T = \textup{Spec}(R)$ for some DVR $R$ so that the generic point of $T$ maps to the generic point of $Y$ and the special point of $T$ maps to $y'$. And then by the lifting property we get that $f'$ is surjective.

Finally, since $X' \to X$ is a bijection, so the fibers of $X' \to Y$ will also be geometrically connected. But this factors through the finite morphism $Y' \to Y$, so $Y' \to Y$ must be a bijection, which finishes the proof.
\end{proof}

\begin{prop} \label{locirr}
$Z_g$ is irreducible of dimension $g^2+2g-4$, $Z_{g-1,1}$ is irreducible of dimension $g^2+2g-5$, and $Z_g$ is locally irreducible at every point.
\end{prop}
\begin{proof}
Irreducibility and dimension of $Z_g$ was seen before, so we need to only show local irreducibility. The map $\phi: B/ \Omega \to Z_g$ is weakly proper, $B/ \Omega$ is normal, and the fibers of $\phi$ are irreducible. So by Lemma \ref{weakprop} we have that $Z_g$ is locally irreducible.
\end{proof}

\section{The case $g= 3$}
\subsection{Preliminary lemmas}

\begin{lemma}[Existence of a K3 surface in $\mathbb P^3$ with given hyperplane sections] \label{lemma2}
Let $C$ be a degree $4$ irreducible, reduced plane curve in a hyperplane $H$ of $\mathbb P^3$, and let $C'$ be a degree $4$ irreducible, reduced plane curve in a hyperplane $H'$ of $\mathbb P^3$ with $H \neq H'$. Assume that 

1. $C' \cap H = C \cap H'$ as subschemes of $H \cap H'$,

2. $\textup{Sing }C \cap \textup{Sing }C' = \O$. (For any variety $X$, $\textup{Sing }X$ is the set of singular points of $X$)

Then there exists a smooth quartic surface $S \subset \mathbb P^3$ such that $S \cap H = C$ and $S \cap H' = C'$. 
\end{lemma}
\begin{proof}
Let $[x:y:z:t]$ denote the coordinates of $\mathbb P^3$. Without loss of generality, we may asssume $H$ to be the hyperplane $\{ x=0\}$, $H'$ to be the hyperplane $\{ t=0\}$. Let the degree $4$ integral curve $C$ in $H$ be given by the equation $g(y,z,t) = 0$, and $C'$ in $H'$ be given by the equation $h(x,y,z)= 0$. The condition $C' \cap H = C \cap H'$ as subschemes of $H \cap H'$ gives that $g(y,z,0) = \lambda h(0,y,z)$ for some constant $\lambda$.

Thus, we want to consider homogeneous polynomials $f(x,y,z,t)$ of degree $4$ such that $f(x,y,z, 0) = \lambda_1 g$ and $f(0,y,z,t) = \lambda_2 h$ for some constants $\lambda_1, \lambda_2$, and $f$ is smooth.
  This is a linear system with base locus $C \cup C'$, so using Bertini theorem, we get that for a general $f$, the only singularities may occur on $C \cup C'$.  
  
  Let $P$ be a singular point, and let $P \in C$, i.e. $x(P) = 0$. Then $f_x(P) = f_y(P) = f_z(P) = f_t(P) = 0$, where $f_x$ is the partial derivative of $f$ wrt $x$, and similar for $y,z,t$. On the other hand, $x(P) =0$ means that $f_y(P) = \lambda_1 g_y(P), f_z(P) = \lambda_1 g_y(P), f_t(P) = \lambda_1 g_t(P)$, thus $P$ is a singular point of $C$. Similarly, if $P \in C'$, then $P$ must be a singular point of $C'$. 
 
 Let $f = \lambda_1 g + x f_1(y,z,t) +$ higher terms in $x$, where $f_1(y,z,t)$ is a degree $3$ homogeneous polynomial in $y,z,t$. If $P \in C$, $f_t(P) = f_1(P)$. Note that the condition on $f_1(y,z,t)$ is that $f_1(y,z,0)$ is fixed up to a constant. Let $f_1(y,z,t) = a \cdot t^3 + $ lower degree terms in $t$. Thus, we may choose $a$ so that $f_1(P)$ is non-zero for the singular points of $C$ not lying on $C'$. Thus, $f_t(P)$ is non-zero at these points for this choice of $a$. Similarly, we may choose the coefficient $b$ of $x^3t$ so that $f_x(P)$ is non-zero for the singular point of $C'$ not lying on $C$. Thus, we may choose $f$ so that the points of $C$ not lying on $C'$ and the points of $C'$ not lying on $C$ are non-singular points of $f$.

Let $P$ be in the intersection of $C$ and $C'$. Then $P$ is a singular point of $f$ only if $P$ is a singular point of both $C$ and $C'$, which is an empty set.

Thus, we get an $f$ which is smooth at every point, and hence defines a smooth quartic surface satisfying the required properties.
\end{proof}

\begin{lemma} \label{integral3}
 Let $C$ be a non-degenerate degree $4$ integral curve in a hyperplane $H$ of $\mathbb P^3$. Then there exists a smooth quartic surface $S \subset \mathbb P^3$ such that $S \cap H = C$. The general such $S$ will satisfy $S \cap H''$ is integral for all hyperplanes $H'' \subset \mathbb P^3$.
\end{lemma}
\begin{proof}
The first part follows from Lemma \ref{lemma2} because there exists a hyperplane $H'$ which intersects $C$ transversally, and we can find a smooth degree $4$ curve in $H'$ passing through any 4 points on a line.

Let the coordinates of $\mathbb P^3$ be $[x:y:z:w]$, $H$ be the hyperplane $\{w = 0\}$, and $C$ in $H$ be given by the quartic equation $\{f_0(x,y,z)=0 \}$. Note that $C$ is integral and non-degenerate, so $C \cap H'$ will be a length $4$ scheme for any hyperplane $H' \neq H$ in $\mathbb P^3$.

Let $f(x,y,z,w) = f_0(x,y,z) + w f_1(x,y,z,w)$ be the general quartic equation satisfying $f(x,y,z,0) = f_0(x,y,z)$. We say that $f$ \textit{extends} $f_0$. The space of such $f$ is isomorphic to $\mathbb A^{20}$. Let $S = V(f) \subset \mathbb P^3$.

\begin{itemize}
    \item Suppose $S$ contains the line $L = \{y=z=0 \}$. Let $H'$ be the hyperplane $\{z = 0\}$. We assume that $L \cap H \subset C \cap H'$. Then setting $y=z=0$, one gets $f_1(x,0,0,w) = 0$. Therefore, the coefficients of $x^3, x^2w, xw^2, w^3$ in $f_1(x,y,z,w)$ are 0. So, we get that the subvariety $D_L$ of $\mathbb A^{20}$ parametrizing $f$ so that $S$ contains $L$ is isomorphic to $\mathbb A^{16}$. 
    
    Now, let the space of lines in $\mathbb P^3$ be $\mathscr L = \mathbb{G}(1,3) = \mathrm{Gr}(2,4)$. Consider the subspace $\mathscr W_1$ of $\mathbb A^{15} \times \mathscr L \times (\mathbb P^3)^{\vee}$ such that $$\mathscr W_1 = \{(f,L,H') | f \text{ extends } f_0, \ L \subset S \cap H' \}.$$ Firstly, observe that $(f,L,H') \in \mathscr W_1$ implies that $H' \neq H$, so we assume $H' \neq H$ in what follows. Note that $L \subset S \cap H'$ implies that $L \cap H \subset C \cap H'$, so the projection map $\pi_{23} : \mathscr W_1 \to \mathscr L \times (\mathbb P^3)^{\vee}$ factors through the space $(L,H')$ so that $L \subset H'$ and $L$ contains a point of $C \cap H'$. Now, the space of lines in a fixed $H'$ which contains a point of $C \cap H'$ is of dimension 1, and the hyperplanes themselves vary in a space of dimension 3, so the space $(L,H')$ so that $L \subset H'$ and $L$ contains a point of $C \cap H'$ has dimension $4$. Also, the fibre of the projection map $\pi_{34}$ over $(L,H')$ will be isomorphic to the variety $D_L$ which has dimension $11$ as computed above. Thus, the dimension of $\mathscr W_1$ is at most $16+4 = 20$. The space of hyperplanes $H'$ containing a fixed line $L$ is of dimension 1, so the dimension of image of the projection map $\pi_{1} : \mathscr W_1 \to \mathbb A^{20}$ is of dimension at most $19$. 
    
    \item Let $H'$ be the hyperplane $\{z = 0\}$. Suppose $S$ contains the conic $T$ in $H'$ given by $\{z=0, t(x,y,w) =0 \}$, where $t$ is a degree $2$ homogeneous polynomial. We assume that $T \cap H \subset C \cap H'$. Then setting $z= 0$, we get that $f(x,y,0,w) = \lambda(x,y,w) t$ for some quadratic polynomial $\lambda$, where $\lambda(x,y,0)$ is a fixed polynomial determined by $f_0$. So, we get that the subvariety $E_T$ of $\mathbb A^{20}$ parametrizing $f$ so that $S$ contains $T$ is isomorphic to $\mathbb A^{15}$.
    
    Now, let the space of integral degree 2 curves in $\mathbb P^3$ be $\mathscr T$. Consider the subspace $\mathscr W_2$ of $\mathbb A^{20} \times \mathscr T \times (\mathbb P^3)^{\vee}$ such that $$\mathscr W_2 = \{(f,T,H') | f \text{ extends } f_0, \ T \subset S \cap H' \}.$$  Firstly, observe that $(f,T,H') \in \mathscr W_2$ implies that $H' \neq H$, so we assume $H' \neq H$ in what follows. Note that $T \subset S \cap H'$ implies that $T \cap H \subset C \cap H'$, so the projection map $ \pi_{23} : \mathscr W_2 \to \mathscr T \times (\mathbb P^3)^{\vee}$ factors through the space of $(T,H')$ so that $T \subset H'$ and $T \cap H \subset C \cap H'$. Note $C \cap H' \subset H \cap H'$ is curvilinear, so there are only finitely many length $2$ subschemes of $C \cap H'$, each lying on the line $\ell = H \cap H'$. Thus, if we fix $T \cap H'$ we also fix $\ell = H' \cap H$. The space of planes $H'$ containing $\ell$ is of dimension 1, and the space of conics $T$ in $H'$ so that $T \cap \ell$ is fixed is of dimension 3. Thus, the space $(T,H')$ so that $T \subset H'$ and $T \cap H \subset C \cap H'$ has dimension $4$. So, the image of $\pi_{34}$ is of dimension at most $4$ and the fibre of $\pi_{23}$ over $(T,H')$ is isomorphic to $E_T$ which is $15$ dimensional, thus $\mathscr W_2$ is of dimension at most $15 + 4 = 19$. So the dimension of image of the projection map $\pi_{1} : \mathscr W_2 \to \mathbb A^{20}$ is of dimension at most $19$. 
\end{itemize}
Now, note that for any hyperplane $H'$, $S \cap H'$ is a degree $4$ curve in $H'$, so if it is reducible then it must contain either a degree $1$ curve in $H'$, which is a line or a degree $2$ curve in $H'$. Hence, the surfaces $S$ for which $S \cap H'$ is reducible must correspond to $f$ which are in the image of $\mathscr W_1$ or in the image of $\mathscr W_2$. Since both of these are not dominant, so we get that the general $f$ which extends $f_0$ must satisfy that $S \cap H'$ is integral for all $H'$.   
\end{proof}

\begin{lemma} \label{nobadcurves3}
 Let $H$ be a hyperplane in $\mathbb P^3$, and let $C \subset H$ be a canonically embedded integral curve of genus $3$. Suppose that either $C$ represents a general point of $Z_3$ or a general point of $Z_{2,1}$. Then there exists a smooth quartic K3 surface $S \subset \mathbb P^3$ such that $S \cap H = C$. The general such $S$ will satisfy that
 \begin{enumerate}
     \item for all hyperplanes $H' \subset \mathbb P^3$, $C' = S \cap H'$ is integral.
     \item If $C' = S \cap H'$ is rational integral and has a cuspidal singularity for a hyperplane $H' \subset \mathbb P^3$ different from $H$, then $C'$ has exactly one simple cusp and 2 simple nodes as singularities.
 \end{enumerate}
\end{lemma}

\begin{proof}
 By Lemma \ref{integral3}, we already know the existence, and the part $(1)$  i.e. for all hyperplanes $H' \subset \mathbb P^3$, $C' = S \cap H'$ is integral. So from now on we will only consider integral curves.
 
 For part $(2)$, we claim that for any hyperplane $H'$ of $\mathbb P^3$, inside the space $Z_{H'}$ of canonically embedded rational integral curves of genus $3$ in $H'$, we have that the subspace $Z'_{H'}$ of curves having either (a) a higher order cusp at some point, or (b) a cusp at some point and a non-simple nodal singularity at some other point, has codimension $ \ge 2$.
 
 To see this, let $U$ be the subspace of $\textup{Hom}(\mathbb P^1,H')$ consisting of maps $g: \mathbb P^1 \to H'$ so that $g^* \mathcal O_{H'}(1) = \mathcal O_{\mathbb P^1}(4)$ and $g(\mathbb P^1)$ is canonically embedded in $H'$. Consider the map $\beta : U \to Z_{H'}$ which sends a map to its image. This is a dominant map with the fibre over a curve $C$ being given by Aut$(\mathbb P^1) \times$Aut$(C)$ which is 3-dimensional.
 Now, we consider the space $V$ of tuples $(g,P,x)$ consisting of maps $g: \mathbb P^1 \to \mathbb P^2$ so that the image of $g$ is in $\mathcal O(4)$ and points $P \in \mathbb P^1$ and $x \in X$ so that $g(P) = x$. Then $V$ is an etale-local fibre bundle over $\mathbb P^1 \times \mathbb P^2$. If $g$ maps $[a:b] \mapsto [p(a,b): q(a,b): r(a,b)]$ where $p,q,r$ are degree $4$ homogeneous polynomials with no common zero, then the condition for $g([0:1]) = [0:1:0]$ is that $p_0 = r_0 = 0$, where $p(a,b) = p_0b^4 + p_1 ab^3 + p_2 a^2b^2 + p_3 a^3b + p_4 a^4$ and similar for $q,r$. Now, the condition for $g$ to be not an immersion at $[0:1]$ is that $p_1 = r_1 = 0$. Thus, we get that codimension 2 irreducible subspace of the fibre over $([0:1], [0:1:0])$ corresponding to non-immersions at $P$. So, we get a codimension 2 irreducible subspace of $V$ itself corresponding to $(g,P,x)$ with $g$ not being an immersion at $P$. Now, the first projection from $V \to U$ has one-dimensional fibres, so we must have the image as a codimension 1 irreducible subspace of $U$. Now, to prove that the general point has a simple cusp and simple nodes at other points, we can work inside the subspace where $p_0 = p_1 = r_0 = r_1 = 0$, and prove that for general $p,q,r$ satisfying this condition, that the image is a curve with one simple cusp and other singularities as nodal. We have that in a neighbourhood of $[0:1:0]$, the map $g$ is
 $$t \mapsto \left( \frac{p(t,1)}{q(t,1)}, \frac{r(t,1)}{q(t,1)} \right) = \left( \frac{t^2(a+bt+b't^2)}{q(t,1)}, \frac{t^2(c+dt + d't^2)}{q(t,1)} \right)$$
 Now, $q(t,1)$ do not vanish at $0$, so we can choose an inverse for it at the completion level, so that the map at the completion becomes $$t \mapsto (a_0t^2 + a_1 t^3 + \cdots, b_0t^2 + b_1t^3 + \cdots)$$ where $$a_0 = \frac a {q_0}, \  b_0 = \frac c {q_0}, \ a_1 = \frac {bq_0 - aq_1}{q_0^2}, \ b_1 = \frac {dq_0 -cq_1}{q_0^2}.$$ Thus, for general $a,b,c,d,q_0,q_1$, we will have that $a_0b_1 \neq a_1b_0$. Hence, we may choose local coordinates so that the map becomes $t \mapsto (t^2, t^3)$. This gives us a simple cusp.
 
 To prove that a general $g$ which is not an immersion at $[0:1]$ is an immersion at every other point, by a similar argument as before, it is enough to show that in the space of maps $g$ such that $g([0:1]) = [0:1:0]$ and $g([1:0]) = [1:0:0]$ with $g$ not being an immersion at $[0:1]$, the subspace of $g$ not being an immersion at $[1:0]$ has codimension 2. But this is simple to verify: the given conditions are that $p_0 = p_1 = r_0 = r_1 = 0$ and $q_4 = r_4 = 0$, and the condition for $g$ to not be an immersion at $[1:0]$ is that $q_3 = s_3 = 0$ so we clearly get a codimension 2 subspace.
 
 Finally, we want to prove that $g(C)$ only has simple nodes at other singular points. Using a similar argument as before, one may assume that $g$ is general so that $g([0:1]) = [0:1:0]$ with $g$ not being an immersion at $[0:1]$ and  also $g([1:0]) = [1:0:0]$ and $g([1:1]) = [1:0:0]$ and proceed to prove that the image has a simple node at $[1:0:0]$. A similar calculation as before gives that the map $g$ in a neighbourhood of $[1:0:0]$ is
  $$t \mapsto \left( \frac{t(t-1)(a+bt+b't^2)}{p(1,t)}, \frac{t(t-1)}{p(1,t)} \right)$$
  This gives us a simple node if $a,b,b'$ and the coefficients of $p$ are general.
  
  Now, let $Y_{H'}$ be the subspace of $W$ so that $S \cap H'$ is rational. Then $Y_{H'}$ has codimension $4$ in $W$. Now, we have a map $\psi_{H'} : Y_{H'} \to Z_{H'}$ which sends $S \mapsto S \cap H'$. This map is smooth. Therefore, the inverse image of $Z'_{H'}$ under $\psi$ has codimension 2 inside $Y_{H'}$, and hence it has codimension $6$ inside $W$. Therefore, this subspace intersects $Y_{H}$ in subspace of codimension at least $2$, and hence it cannot map under $\psi_H$ onto any codimension 1 subspace of $Z_H$. Since the curve $C$ that we consider either corresponds to a general point of $Z_H$ or a general point of a codimension 1 subspace of $Z_H$, so we have the result.
 
 \end{proof}

\subsubsection{Rational Curves containing a prescribed finite subscheme}

\begin{lemma} \label{calc}
Following the notation of Lemma \ref{mainlemma}, let $C = \mathbb P^1$, and $T_{\phi}(F)$ denote the tangent space of $F = F(C,X,Y_1, \cdots, Y_m)$ at the point $(\phi,D_1, \cdots, D_m) \in F(k)$.

If $X = \mathbb P^2$, $D_i$ is of length $n_i$, $\sum n_i = 4$ and $\phi : \mathbb P^1 \to \mathbb P^2$ is a degree 4 map (i.e. $\phi^* \mathcal O(1) = \mathcal O(4)$) which is an immersion, then $\dim T_\phi (F)= 10$.
\end{lemma}
\begin{proof}
Using Lemma \ref{mainlemma}, the tangent space of $F$ is given by $\gamma^{-1}(H^0(C,\mathcal I_D N_\phi))$, where $N_\phi$ is the normal bundle corresponding to $\phi$, $\gamma$ is the map $H^0(C,\phi^* T_X) \rightarrow H^0(C,N_\phi)$, and $\mathcal I_D = \mathscr I_{D_1} \cdots \mathscr I_{D_m} = \mathcal O(-4)$ (since $\sum n_i = 4$).
 
 So we need to calculate the dimension of this space for maps $\phi : \mathbb P^1 \to X$.

We know that $X = \mathbb P^2 $, $\phi : \mathbb P^1 \to \mathbb P^2$ is a degree 4 map. We have,
 $$0 \to T_{\mathbb P^1} \to \phi^* T_X \to N_\phi \to 0 $$
 
 Now $\det T_X = (\omega_X)^\vee = \mathcal O(3)$ implies that $\det \phi^* T_X = \mathcal O(12)$, since $\phi $ is a degree 4 map. Thus, taking determinants, we get that $N_\phi = \mathcal O(10)$ and $\mathcal I_D N_\phi = \mathcal O(6)$. So $H^0(\mathcal I_D N_\phi)$ is a codimension-$4$ space inside the $11$ dimensional space $H^0(N_\phi)$.
 
 
 Now, $H^1(\mathbb P^1, T_{\mathbb P^1}) = 0$, so taking the long exact sequence on cohomology we get
 
 $$0 \to H^0(\mathbb P^1,T_{\mathbb P^1}) \to H^0(\mathbb P^1,\phi^* T_X) \to H^0(\mathbb P^1,N_\phi) \to 0 $$
 
 Now, $H^0(\mathbb P^1,\mathcal I_D N_\phi)$ has codimension-$4$ in $H^0(\mathbb P^1,N_\phi)$, so, the dimension of the required tangent space is $11 + 3 -4 = \boxed{10}$.
 \end{proof}
 
 \begin{lemma} \label{cuspcalc3}
Let $C$ be the unique (up to isomorphism) rational curve of genus $1$ with 1 simple cusp. Let $Y_1, \cdots, Y_m$ be finite length curvilinear subschemes of $\mathbb P^2$, with their lengths totalling $4$.

Consider the functor $F = F(C,\mathbb P^2,Y_1, \cdots, Y_m)$, and let $(g,D_1, \cdots, D_m) \in F(k)$ be a point. Suppose that $g: C \to \mathbb P^2$ is an immersion, and the image $g(C)$ is a degree $4$ curve in $\mathbb P^2$. Also we assume that $g$ is an embedding at the cusp of $C$, that $g(C)$ only has 2 simple nodes and 1 simple cusp and that at most one of the $D_i$ are supported at the cusp of $C$.
Then if $(\dim F)_g$ denotes the dimension of $F$ at the point $(g,D_1, \cdots, D_m)$, then we have
\begin{itemize}
    \item $(\dim F)_g \le 8$ if none of the $D_i$ are not supported at the cusp,
    \item $(\dim F)_g \le 9$ if length $D_i \le 2$ if $D_i$ is supported at the cusp,
    \item $(\dim F)_g \le 10$ otherwise.
\end{itemize}
\end{lemma}

\begin{proof}
We have the following exact sequence on the tangent spaces:
$$0 \to T_C \to g^* T_{\mathbb P^2} \xrightarrow{\phi} N_g$$
where $N_g$ is the dual of $\mathscr I_C/\mathscr I_C^2$ in a neighbourhood of the cusp and is the usual normal bundle elsewhere. Then this gives us a map on global sections:
$$0 \to H^0(C,T_C) \to H^0(C, g^* T_{\mathbb P^2}) \xrightarrow{H^0(\phi)} H^0(C, N_g)$$
We want the inverse image of $H^0(C, \mathscr I_D N_g)$ under $H^0(\phi)$ where $\mathscr I_D = \mathscr I_{D_1} \cdots \mathscr I_{D_m}$. 

We first prove that the degree of $N_g$ is 12. There is a canonical map from $N_g$ to the pullback of the normal bundle of $g(C)$ in $X$ which is an isomorphism outside the $4$ points lying above the $2$ nodes. The latter has degree $4 \times 4 = 16$ so $\deg N_g = 12$ will follow by showing that the cokernel of the map referred to above is a finite length sheaf supported at these $4$ points and computing the length. To compute this, all we need to do is work out what all these sheaves and maps are locally i.e. when $g(C)$ is the curve $xy = 0$ in the plane and $C$ is its normalisation (so a union of two lines). Now, let $A = k\llbracket x,y \rrbracket, I = (xy), B_1 = k \llbracket u \rrbracket, B_2 = k \llbracket v \rrbracket, B = B_1 \oplus B_2$ and we have the map 
\begin{align*}
    A &\xrightarrow{ \ g \ } B \\
    x & \mapsto u \\
    y & \mapsto v
\end{align*}
Now, if $T_R$ denotes the dual of $\Omega^1_R$ of any ring $R$, then $T_B = T_{B_1} \oplus T_{B_2} \cong B_1 \oplus B_2$, $T_A \cong A \oplus A$ with the maps
\begin{align*}
    T_B = B_1 \oplus B_2 &\xrightarrow{ \ d/dg \ } B \oplus B  = T_A \otimes B \\
    (1, 0) & \mapsto ((1,0),(0,0)) \\
    (0,1) & \mapsto ((0,0),(0,1))
\end{align*}
So the cokernel $N_g$ of $d/dg$ is isomorphic to $B$ and $T_A \otimes B \to N_g$ is given by $((a,b),(c,d)) \mapsto (c,b)$. Now, the map from $T_A$ to $N = \textup{Hom}(I/I^2,A) \cong A$ is given by 
\begin{align*}
    T_A = A \oplus A &\xrightarrow{ \ \psi \ } A = \textup{Hom}(I/I^2,A) \\
    (1, 0) & \mapsto y \\
    (0,1) & \mapsto x
\end{align*}
Thus, the canonical map $N_g \cong B \to B \cong N \otimes B$ will be given by $(1,0) \mapsto (u,0)$ and $(0,1) \mapsto (0,v)$, which means that the cokernel of this map has length $2$ as desired with length $1$ supported at each of the two lines.

So, now we have that $N_g$ has degree $= 16 - 2 \times 2 = 12$. This implies that $\mathscr I_D N_g$ has degree $8$. Thus, 
$$h^0(C,N_g) = 12, h^0(C,\mathscr I_D N_g) = 8$$

\bigskip

Let $Q_g$ be the image of $g^* T_X$ in $N_g$ under $\phi$. Then, since \footnote{This may be seen using a computer software, see \url{https://faculty.math.illinois.edu/Macaulay2/doc/Macaulay2-1.16/share/doc/Macaulay2/Macaulay2Doc/html/_tangent__Sheaf_lp__Projective__Variety_rp.html} where they have done precisely this calculation.}
$$h^0(C,T_C) = 2, h^1(C,T_C) = 0$$
so we have
$$0 \to H^0(C,T_C) \to H^0(C, g^* T_X) \to H^0(C, Q_g) \to 0$$
Therefore, we want to know the dimension of $H^0(C, \mathscr I_DN_g \cap Q_g)$.

Now, $Q_g = N_g$ away from the cusp, since $\phi$ is an immersion away from the cusp and $C$ is smooth away from the cusp. Now, in a neighbourhood of the cusp, we have
\begin{align*}
    \mathscr I_C/\mathscr I_C^2 &\xrightarrow{d} \Omega_X|_C \\
    y^2 - x^3 & \mapsto d(y^2 - x^3) = 2ydy - 3x^2dx
\end{align*}
So, locally around the cusp, we have $Q_g$ is $(3x^2, 2y) \cdot N_g = (x^2,y) \cdot N_g$. Thus, $N_g/Q_g$ will be length 2 supported at the cusp. Let $\mathcal L$ be the line bundle which is the subsheaf of $N_g$ defined as $(y) \cdot N_g$ around the cusp and identical to $N_g$ away from the cusp. Then we have $\mathcal L \subset Q_g \subset N_g$, and we also have $N_g/\mathcal L$ to be of length 3. Therefore, $\mathcal L$ is a degree 9 line bundle, implying $h^1(C,\mathcal L) = 0$, and therefore 
$$0 \to H^0(C,\mathcal L) \to H^0(C, N_g) \to H^0(C, N_g/ \mathcal L) \to 0$$
is exact. In particular, we have that $H^0(C, N_g) \to H^0(C, N_g/ \mathcal L)$ is surjective, and combining this with the fact that $H^0(C, N_g/\mathcal L) \to H^0(C, N_g/Q_g)$ is surjective, we get that
$$0 \to H^0(C,Q_g) \to H^0(C, N_g) \to H^0(C, N_g/ Q_g) \to 0$$
is exact. Thus, $h^0(C,Q_g) = 10$.

Now, $\mathscr I_D \mathcal L \subset \mathscr I_D N_g \cap Q_g$, and $\mathscr I_D \mathcal L$ is a degree $3$ line bundle so $h^1(C, \mathscr I_D \mathcal L) = 0$, so a similar argument as before implies that 
$$0 \to H^0(C,\mathscr I_D N_g \cap Q_g) \to H^0(C, \mathscr I_DN_g) \to H^0(C, \mathscr I_DN_g/ \mathscr I_DN_g \cap Q_g) \to 0$$
If $D$ is not supported at the cusp, then $h^0(C, \mathscr I_D N_g \cap Q_g) = h^0(C, \mathscr I_D N_g) - 2 =6$ and we will have the dimension of the tangent space as $6+2 = 8$.

If $D$ is supported at the cusp, then locally we have that $D$ is given by $f \in R = k[[x,y]]/(y^2 - x^3)$. We claim that if $f \in (x^2,y)$ then $R/(f)$ has length $ \ge 3$. This follows because if $R/(f)$ has length $2$ then $f \in m \setminus m^2$ where $m = (x,y)$ so $f$ has linear term $y$ since $f$ is also in $(x^2,y)$. But $R/(f)$ has length $\ge 3$ for any $f$ which has linear term $y$.

So, if $\mathscr I_D N_g \cap Q_g = \mathscr I_D N_g$ then $D$ has length $3$ at the cusp. In this case, the dimension of the tangent space is $8+2 = 10$.

Finally, if $f \not \in (x^2,y)$ then $h^0(C,\mathscr I_D N_g \cap Q_g) = 7$ and we get the dimension of the tangent space is $7+2 = 9$.
\end{proof}

\begin{lemma} \label{functorcalc3}
For any fixed curvilinear finite scheme $Y \in \mathbb P^{2}$, let $R_Y$ be the space of degree $4$ rational integral curves $C'$ in $\mathbb P^{2}$ which contain $Y$. Let $R^{(1)}_Y$ be the subspace of $R_Y$ corresponding to points representing curves which do not have a cusp, and let $R^{(2)}_Y$ be the points representing curves which have a cusp.

Suppose $Y$ is a length $4$ scheme supported on a line $L$. Then we have:
\begin{enumerate}
    \item $\dim R^{(1)}_Y \le 7$.
    \item For a curve $C'$ representing a point of $R^{(2)}_Y$ which has exactly 1 cusp and 2 nodes as singularities, 
    \begin{enumerate}
        \item If $Y$ is not supported at the cusp of $C'$, then $(\dim R^{(2)}_Y)_{C'} \le 6$.
        \item If every component of $Y$ has length $\le 2$, then $(\dim R^{(2)}_Y)_{C'} \le 7$.
        \item $(\dim R^{(2)}_Y)_{C'} \le 7$ for any $Y$.
    \end{enumerate}
\end{enumerate}
\end{lemma}
\begin{proof}
Let $Y_1, \cdots, Y_m$ be the connected components of $Y$ with $\deg Y_i = n_i$. $Y_i$'s are all curvilinear since they are subschemes of $L$. Let $C'$ be a integral rational curve of degree $4$ in $\mathbb P^2$ which contains $Y$. 

 \smallskip

 If $C'$ does not have a cusp, let $\phi : \mathbb P^1 \to C' \hookrightarrow \mathbb P^2$ be a normalization map of $C'$. Then $\phi$ is an immersion. Let $D_i$ be length $n_i$ subschemes of $\mathbb P^1$ mapping to $Y_i$. Thus, $(\phi, D_1, \cdots, D_m)$ is a point of the space $F_1 = F_1(\mathbb P^1, \mathbb P^2, Y_1, \cdots, Y_m)$. We bound the dimension of $F_1$ at the point $(\phi, D_1,\cdots, D_m) \in F_1(k)$ by looking at the tangent space of $F_1$ at this point. By Lemma \ref{calc}, we have this bound on the dimension as $10$. If $F'_1$ is the component of $F$ where $\phi$ is a degree $4$ map, then we have a dominant map $\Psi_1: F'_1 \to R^{(1)}_Y$ defined as $(\phi,D_1, \cdots, D_m) \mapsto \phi(\mathbb P^1)$.
 The automorphism group of $\mathbb P^1$ is 3-dimensional, and fixes $\phi(\mathbb P^1)$ so every fibre of $\Psi_1$ is at least 3-dimesional, and so we have a upper bound of $7$ on the dimension of $R^{(1)}_Y$ at $C'$.
 
 \smallskip
 
 Now, let $C'$ have one simple cusp and 2 simple nodes. Let $\mathscr C$ be the unique (up to isomorphism) rational curve of genus 1 with 1 cusp, and let $\phi: \mathscr C \to C'$ be the blow up of $C'$ at the nodes. Let $D_i$ be length $n_i$ subschemes of $\mathscr C$ mapping to $Y_i$. Thus, $(\phi, D_1, \cdots, D_m)$ is a point of the space $F_2 = F_2(\mathscr C, \mathbb P^2, Y_1, \cdots, Y_m)$. If $F'_2$ is the component of $F_2$ where $\phi$ is a degree $4$ map, then we have a dominant map $\Psi_2: F'_2 \to R^{(2)}_Y$ defined as $(\phi,D_1, \cdots, D_m) \mapsto \phi(\mathbb P^1)$. Now, by Lemma \ref{cuspcalc3}, the dimension of $F_2$ at $(\phi,D_1, \cdots, D_m)$ is bounded by $8$ if $D_i$ are not supported at the cusp, bounded by $9$ if length $D_i \le 2$ if $D_i$ is supported at the cusp and bounded by $10$ regardless.
  The automorphism group of $\mathscr C$ is 2-dimensional, and fixes $\phi(\mathscr C)$ so every fibre of $\Psi_2$ is at least 2-dimensional, and so we have a upper bound of $6,7,$ or $8$ on the dimension of $R^{(2)}_Y$ at $C'$ depending on the above conditions.
\end{proof}

\begin{lemma} \label{genpoints3}
Let $L$ be a line in $\mathbb P^2$. 
\begin{enumerate}
    \item If $C \subset \mathbb P^2$ is a curve corresponding to a general point of $Z_3$, then $C \cap L$ consists of $4$ general points on $L$.
    \item If $C \subset \mathbb P^2$ is a curve corresponding to a general point of $Z_{2,1}$, then $C \cap L$ consists of $4$ general points on $L$.
\end{enumerate}
\end{lemma}
\begin{proof}
\textit{Observation: There exists a point of $Z_3$ (and $Z_{2,1}$) so that the curve $C$ that it represents intersects $L$ in $4$ distinct points.} We note that if $C$ is any curve corresponding to a point of $Z_3$ (or $Z_{2,1}$) then $C \cap L'$ will be the union of $6$ distinct points for a general hyperplane $L$. Now, since all lines in $\mathbb P^2$ are projectively equivalent, this implies that there will be a $\sigma \in \textup{Aut}(\mathbb P^2)$ so that $\sigma(C) \cap L$ is the union of $6$ distinct points. On the other hand, note that $\sigma(C)$ will continue to represent a point of $Z_3$ (or $Z_{2,1}$, respectively). 

\begin{enumerate}
    \item We know that $Z_3$ is irreducible of dimension $11$.
    
    Now, consider the rational map $\alpha: Z_3 \dashedrightarrow \textup{Sym} ^4 (L)$ where $C \mapsto C \cap L$. Take any $Y = \{p_1, \cdots, p_4 \}$ corresponding to a point on $ \textup{Sym} ^4 (L)$, then the the fiber of $\alpha$ over $Y$ is $R_Y$. Note that $Y$ consists of $4$ distinct points, therefore $R_Y$ has dimension $\le 7$ by Lemma \ref{functorcalc3}. Therefore, $\alpha$ has to be dominant, since the image cannot be less than $11 - 7 = 4$ dimensional.
    
    \item We know that $Z_{2,1}$ is irreducible of dimension $10$. 
    
    Now, consider the rational map $\alpha: Z_{2,1} \dashedrightarrow \textup{Sym} ^4 (L)$ where $C \mapsto C \cap L$. Take any $Y = \{p_1, \cdots, p_4 \}$ corresponding to a point on $ \textup{Sym} ^4 (L)$, then the the fiber of $\alpha$ over $Y$ is $R^{(2)}_Y$. Note that $Y$ consists of $4$ distinct points, so it cannot be supported at the cusp of any curve in the inverse image of $Y$ under $\alpha$. Thus, $R^{(2)}_Y$ has dimension $\le 6$ by Lemma \ref{functorcalc3}. Therefore, $\alpha$ has to be dominant, since the image cannot be less than $10 - 6 = 4$ dimensional.

\end{enumerate}
\end{proof}



 

\subsection{Proof of the main result}

\begin{itemize}
    \item Let $W$ be the space of quartic surfaces in $\mathbb P^3$. So $W \cong \mathbb P^{34}$.
    \item Let $J = \{ (S,H)  \ | \ S \cap H$ is rational$ \} \subseteq W \times (\mathbb P^g)^{\vee}$.
    \item Let $\pi : J \to W,  \ \eta : J \to (\mathbb P^g)^{\vee}$ be the projection maps.

\end{itemize}
\begin{prop} \label{trans3} 
 The monodromy group $\Pi_g$ of $\pi: J \to W$ is transitive for $g =3$.
\end{prop}
\begin{proof}
We know by Lemma \ref{integral3}, that the general element $S$ of $W$ has all its hyperplane sections as integral curves, so we may restrict attention to $J_0 = \{ (S,H)  \ | \ S \cap H$ is rational integral$ \}$.
It suffices to show that $J_0$ has only one irreducible component of maximum dimension. We restrict $\pi$ and $\eta$ to $J_0$ and call them $\pi_0$ and $\eta_0$.

Let the fiber of $\eta_0$ above a fixed hyperplane $H$ be $Y_H = \{ S \ | \ S \cap H \textup{ is rational integral} \}$. Consider the map $\psi:Y_H \to Z$ from $Y_H$ to the space $Z$ of rational integral curves of degree $4$ in $H \cong \mathbb P^2$ which sends $S \mapsto S\cap H$. Then $\psi$ is surjective due to Lemma \ref{lemma2}. Let $\mathrm{Hilb}_d(\mathbb P^n)$ be the Hilbert scheme of degree $d$ hypersurfaces in $\mathbb P^n$. Then we have that $\mathrm{Hilb}_d(\mathbb P^n) \cong \mathbb P^{m_{d,n}}$ where $m_{d,n} = \binom{n+d}{d} -1$. For any hyperplane $H \subset \mathbb P^n$, let $V_{d,n}$ be the subspace of $\mathrm{Hilb}_d(\mathbb P^n)$ consisting of hypersurfaces which contain $H$. Then we have a map 
$$\psi_{d,n,H} : \mathrm{Hilb}_d(\mathbb P^n) \setminus V_{d,n} \to \mathrm{Hilb}_d(\mathbb P^{n-1})$$
obtained by intersecting the hypersurface with $H$. We can also choose coordinates so that $\psi_{d,n,H}$ is seen as a projection map from $\mathbb P^{m_{d,n}} \setminus V_{d,n} \to \mathbb P^{m_{d,n-1}}$, and therefore we have that $\psi_{d,n,H}$ is flat. Now, $\psi : Y_H \to Z$ is just a base change of $\psi_{4,3,H}$ to $Z$, therefore we get that $\psi$ is also flat.

Let $C$ be a degree $4$ rational integral curve in $H$. Let $\mathbb P^3$ be parametrized by coordinates $[x:y:z:t]$, $H$ be given by $\{ t=0 \}$, and $C$ be given by the degree 4 equation $\{f(x,y,z) =0 \}$, then the fiber of $\psi$ over $C$ is parametrized by degree $4$ homogeneous polynomials $\tilde f$ such that $\tilde f(x,y,z,0) = cf$, for some $c$ non-zero scalar. Thus the fibers of $\psi$ are isomorphic to $\mathbb A^{20}$ and thus smooth and irreducible of dimension 20. Combining with the earlier observation that $\psi$ is flat, we get that $\psi$ is smooth.

Now, recall $Z_3$ from Section 3.2. This is a dense subscheme of $Z$ corresponding to rational curves with only nodes or cusps as singularities and whose normalization map does not have a non-trivial automorphism. From Lemma \ref{locirr}, $Z_3$ is irreducible of dimension $11$, so $Z$ is also irreducible of dimension $11$.

Thus, since $\psi$ is smooth, $Y_H$ is also irreducible of dimension 31. 

Finally, it suffices to note that $\eta_0: J_0 \to (\mathbb P^3)^\vee$ makes $J_0$ into an \'{e}tale locally trivial fibre bundle over $(\mathbb P^3)^\vee$, thus $J_0$ is irreducible of dimension 34. Thus we may apply Lemma \ref{transitive} to end the proof.
\end{proof}

\bigskip

Let $H'$ be a hyperplane in $\mathbb P^3$. A curve $C' \subset H'$ \textit{satisfies $(*)$} if it satisfies the following:
\begin{enumerate}
    \item $C'$ is an integral rational curve of genus $3$ and is canonically embedded in $H'$.
    \item If $C'$ has a cusp then it has exactly one simple cusp and 2 simple nodes.
\end{enumerate}

Fix a hyperplane $H$ in $\mathbb P^3$, a rational integral curve $C \subset H$ of degree 4. Assume that the singularities of $C$ are either all nodal singularities or one simple cusp and others nodal, and that $C$ corresponds to a general point of $Z_3$ in the former case and a general point of $Z_{2,1}$ in the latter case.

\begin{itemize}
    \item Let $W' = \{ S \ | \ S \cap H = C \} \subseteq W$. Note that $W'$ is irreducible of dimension $20$. Also, it always contains a smooth surface due to Lemma \ref{lemma2}. Due to Lemma \ref{integral3}, we have an open subset of $W'$ consisting of smooth $S$ so that $S \cap H'$ satisfies $(*)$ for every $H'$ for which $S \cap H'$ is rational. Denote this open subset by $W''$.
    \item Let $J' = \{ (S,H')  \ | \ S \cap H'$ is rational integral, $H'\neq H, \ S\cap H = C \} \subseteq W' \times ((\mathbb P^3)^{\vee} - H)$. Let $\pi ' : J' \to W'$ be the projection map to $W'$. 
    \item Let the fiber of $\eta ''$ above a fixed $H'$ be $T_{H'} = \{ S \ | \ S \cap H = C, \ S\cap H' \textup{ is rational integral }\}$.
    \item We have a map $\tau : T_{H'} \to R_{C \cap H'}$, which sends $S \mapsto S \cap H'$.
    \item The fibre of $\tau$ over $C'$ is given by the space $X_{C,C'} = \{ S \ | \ S \cap H = C, \ S\cap H' = C'\}$ which may be identified with the space of quartics $f(x,y,z,t)$ such that $f(x,y,z,0) = f_1(x,y,z)$ and $f(x,y,0,t) = f_2(x,y,t)$ are fixed (up to scalars). Thus it is non-empty iff $f_1(x,y,0) = f_2(x,y,0)$ up to a scalar, and in this case we may vary coefficients of monomials which involve $zt$ which is a linear space of dimension 10 hence $X_{C,C'}$ is isomorphic to $\mathbb A^{10}$ if $C \cap H' = C' \cap H$ as subschemes of $H \cap H'$.
\end{itemize}

The following is the key computation of the proof:

\begin{prop} \label{irr3}
If $Y$ is the disjoint union of $4$ points on a line in $\mathbb P^2$, $R_Y$ is irreducible of dimension $7$.
\end{prop}
 \begin{proof}
  We consider degree $4$ maps $f: \mathbb P^1 \to \mathbb P^2$ so that the image $f(\mathbb P^1)$ contains some $4$ given points on a line and we fix three of the points on $\mathbb P^1$ mapping to three of these given points. Call this space as $\textup{Maps}_Y(\mathbb P^1, \mathbb P^2)$.
 
 Such a map is given by $[a:b] \mapsto [f_1(a,b) : f_2(a,b): f_3(a,b)]$, where $f_1,f_2,f_3$ are homogeneous polynomials of degree $4$ 
 having no common zeros.
 
 Let the fixed $4$ points be
   $$ [0:1] \mapsto [0:1:0],$$ 
   $$[1:0] \mapsto [1:0:0],$$ 
   $$[1:1] \mapsto [1:1:0],$$
   $$[1:\lambda] \mapsto [1: \mu:0].$$
 Here $\mu$ is fixed, $\mu \neq 0,1$, and $\lambda$ is allowed to vary. Let $f_1(a,b) = p_0 a^4 + p_1a^3b + p_2a^2b^2 + p_3ab^3 + p_4 b^4$, and similar for $f_2$ and $f_3$ with $p_i$'s replaced by $q_i$ and $r_i$ 
 respectively.
 
 Thus, we are looking at the subspace of $\mathbb P^{15} \times \mathbb A^1$ corresponding to the conditions
 
 \begin{equation*}
 \begin{aligned}
 p_0 \neq 0, \ q_0 &= 0, \ r_0 = 0 \\
 p_4 =0, \ q_4 &\neq 0, \ r_4 = 0  \\
 p_0+p_1+p_2+p_3 &= q_1+q_2+q_3+q_4 \neq 0, \\
 \mu(p_0 + p_1\lambda + p_2 \lambda^2 + p_3\lambda^3) &= q_1\lambda + q_2\lambda^2 + q_3\lambda^3 + q_4 \lambda^4 \neq 0, \\
 r_1+r_2+r_3&= 0, \\r_1 + r_2\lambda + r_3 \lambda^2 &= 0.
 \end{aligned}
 \end{equation*} 
 
 Thus, we have $\lambda \neq 0,1$, $r_3 \neq 0$. Now, $\lambda = r_1/r_3$, substituting this and $p_0 = q_1 + q_2 + q_3 +q_4 - p_1 - p_2 - p_3$ in the fourth equation we get that the space of integral rational curves in $\mathbb P^2$ containing the points $[0:1:0], [1:0:0], [1:1:0], [1: \mu:0]$ is isomorphic to the hypersurface $F$ of $\mathbb P[r_1:r_3:p_1:p_2:p_3:q_1:q_2:q_3:q_4]$ intersected with the open sets $q_1+q_2+q_3+q_4 \neq 0$, $q_1r_3^3 + q_2r_1r_3^2 + q_3r_1^2r_3 + q_4 r_1^3 \neq 0$, $q_1 + q_2 + q_3 +q_4 - p_1 - p_2 - p_3 \neq 0$ and $q_4 \neq 0$ where $F$ is the polynomial
 $$F = \mu(r_3^4(q_1 + q_2 + q_3 +q_4 - p_1 - p_2 - p_3) +p_1r_1^3r_3 + p_2r_1^2r_3^2 + p_3r_1^3r_3) - q_1r_1r_3^3 - q_2r_1^2r_3^2 - q_3r_1^3r_3 - q_4r_1^4$$
 
 We prove that this space is irreducible in $\mathbb P[r_1:r_3:p_1:p_2:p_3:q_1:q_2:q_3:q_4] \cong \mathbb P^8$ of dimension 7.
 
 First we prove that $F$ is irreducible. Suppose $F$ is reducible, say $F = GH$. Now, $F$ is linear in $p_1$, so one of $G,H$ has to be linear in $p_1$, and the other has to have no 
 terms containing $p_1$. Similar for $p_2,p_3,q_1,q_2,q_3,q_4$. But note that we do not have a term involving $p_ip_j$ or $p_iq_j$, thus if $G$ 
 is linear in $p_1$, then it must be linear in all the $p_i, q_j$. Hence $H = H(r_1,r_3, \mu)$.
 
 
 Now, $H$ must divide the coefficient of $p_1$ which is $\mu(r_1r_3^3 - r_3^4) = \mu r_3^3(r_1-r_3)$. $\mu \neq 0$, so this implies that $H$ must be divisible by $r_3$ or $(r_1 - r_3)$. 
 But this implies that $F$ is divisible by $r_3$ or $(r_1 - r_3)$. $F$ is not divisible by $r_3$ since we have a non-zero coefficient $q_4$ of $r_1^4$. Also, $F$ is not divisible by $(r_1 - r_3)$ since substituting $r_1 = r_3$ in $F$ we get $\mu(q_1 + q_2 + q_3 +q_4) - q_1+ q_2 + q_3 + q_4$ which is non-zero since $\mu \neq 1$.
 
 Thus, the space $\{ F=0 \}$ is irreducible of dimension $7$. Now the space under consideration is an open subset of $\{F =0 \}$ so it suffices to prove that it is non-empty. Now the open sets under consideration $q_1+q_2+q_3+q_4 \neq 0$, $q_1r_3^3 + q_2r_1r_3^2 + q_3r_1^2r_3 + q_4 r_1^3 \neq 0$, $q_1 + q_2 + q_3 +q_4 - p_1 - p_2 - p_3 \neq 0$ and $q_4 \neq 0$ are all complements of irreducible hypersurfaces, so it suffices to observe that since $F$ is irreducible and not equal to any one of these hypersurfaces so it must not be contained in the union of these hypersurfaces. 
 
 Thus, $\textup{Maps}_Y(\mathbb P^1, \mathbb P^2)$ is irreducible of dimension $7$. Now, $\textup{Maps}_Y(\mathbb P^1, \mathbb P^2)$ surjects onto $R_Y$ by sending a map to its image in $\mathbb P^2$. The fiber over any point is a finite subset of $S_3$, since any $\phi \in \textup{Maps}_Y(\mathbb P^1, \mathbb P^2)$ is a normalization map of the image $\phi(C)$ and there is only the trivial automorphism of $\mathbb P^1$ which fixes $[1:0],[1:1],[0:1]$, so the only choice is to choose which points of $Y$ correspond to $[1:0],[1:1],[0:1]$.
 
 Thus, $R_Y$ is irreducible of dimension $7$.
 \end{proof}
\begin{prop} \label{twotrans3}
 The monodromy group of $\pi'$ is transitive. 
\end{prop}
\begin{proof}

 We have the following calculations:
 
 \begin{itemize}
     \item Let $Y = C \cap H'$. So, $Y$ is a length $4$ scheme supported on the line $H \cap H'$. If $Y$ is 4 distinct points, then by Proposition \ref{irr3} we have $R_Y$ is irreducible of dimension 7. By Lemma \ref{functorcalc3}, if every component of $Y$ has length $\le 2$ then $R_Y$ has dimension $\le 7$ at every point corresponding to a curve satisfying $(*)$. If some component of $Y$ has length $\ge 3$ then $R_Y$ has dimension $\le 8$ at every point corresponding to a curve satisfying $(*)$. 
     \item The fiber $X_{C,C'}$ of $\tau$ over $C'$ is either empty or irreducible of dimension $10$. If $H'$ does not pass through any singular point of $C$, then $\tau$ is surjective by Lemma \ref{lemma2}. So, if $H'$ intersects $C$ transversally, we have that $\dim T_{H'} = 17$ and it has only one irreducible component of dimension $17$.
     \item A general line passing through any singular point $P$ of $C$ only intersects $C$ at a length 2 scheme at $P$ by Lemma \ref{locallengthcalc}. Also, the tangent line to $C$ at a general smooth point of $C$ intersects $C$ at $P$ with multiplicity $2$. Thus, if $H'$ is either a general hyperplane containing a singular point of $C$ or a general hyperplane tangent to some point of $C$ (by Lemma \ref{multiplicity}), then it intersects $C$ at $Y$ with each component of $Y$ having length $1$ or $2$. So, $(\dim R_{C \cap H'})_{C'} \le 7$ for such $H'$ and for $C'$ satisfying $(*)$, and hence $(\dim T_{H'})_{S} \le 17$ for $[S] \in W''$ and $H'$ in the codimension 1 subspace of $(\mathbb P^3)^\vee$ of general tangents or general hyperplanes passing through a singular point of $C$.
     \item Finally, if $H'$ is in the codimension 2 subspace corresponding to the hyperplanes so that $C \cap H'$ has a length $3$ component, then we have $(\dim R_{C \cap H'})_{C'} \le 8$ for such $H'$ and for $C'$ satisfying $(*)$. So, $(\dim T_{H'})_{S} \le 18$ for $[S] \in W''$ over this codimension 2 subspace of $(\mathbb P^3)^\vee$.
     \item Thus, $(\mathbb P^3)^\vee -H = U \cup V_1 \cup V_2$ where $U$ consists of $H'$ which meets $C$ transversally, $V_1$ consists of $H'$ so that every component of $C \cap H'$ has length $\le 2$, and $V_2$ is the rest. Note that $V_1,V_2$ are of codimension 1,2 respectively. We have that $T_{H'}$ has only one irreducible component of maximum dimension, $\dim T_{H'}=17$ over $H' \in U$, $(\dim T_{H'})_S \le 17$ for $[S] \in W'', H' \in V_1$ and $(\dim T_{H'})_S \le 18$ for $[S] \in W'', H' \in V_2$. 
     \item So if we consider the open set $\eta '^{-1}(U) \subset J'$, then we get that $$(\dim (J' \setminus \eta '^{-1}(U)))_{(S,H')} < 20$$ for $[S] \in W''$. Therefore, $(J' \setminus \eta '^{-1}(U))$ cannot dominate $W''$ and hence cannot dominate $W'$. Now, $\eta '|_{U}: \eta '^{-1}(U) \to  U$ is dominant, and its fibres have only one irreducible component of maximum dimension $17$. Let $J''$ be the closed set of $\eta '^{-1}(U)$ consisting of $x \in J'$ so that $\eta '^{-1}(\eta(x))$ has dimension $\ge 17$ at $x$. So, we get that $\eta '|_{J''}: J'' \to U$ is dominant and its fibres are irreducible of dimension $17$. Thus, since $U$ is irreducible of dimension 3, we have that $J''$ has only one irreducible component of maximum dimension 20. So this unique irreducible component of $J''$ must be the only component of $J'$ which dominates $W'$ (Note that the fibres of $\eta '$ at $\eta '^{-1}(U) \setminus J''$ are of dimension $\le 16$ so this space will be of dimension $\le 19$ and so will not dominate $W'$). Thus, we can conclude that $J'$ has only one irreducible component which dominates $W'$. Thus, $\pi'$ has transitive monodromy by Lemma \ref{transitive}.  \qedhere
 \end{itemize} 
 \end{proof}

\begin{prop}
 The monodromy group $\Pi_g$ of $\pi$ is 2-transitive for $g=3$.
\end{prop}
\begin{proof}
Let $H$ be a hyperplane in $\mathbb P^3$ and let $C \subset H$ be a canonically embedded integral rational curve of genus 3 corresponding to a general point of $Z_3$. Then by Lemma \ref{integral3}, there exists a smooth $K3$ surface $S$ in $\mathbb P^3$ such that $S \cap H = C$. This implies that for a surface $S$ corresponding to a general point of $W$, there is a hyperplane section of $S$ corresponding to a general point of $Z_3$.
 
 Using Proposition \ref{twotrans3} for this $C$ which is a general point of $Z_3$ and which is also a hyperplane section $S \cap H$ for a $S$ corresponding to a general point of $W$, we have that the monodromy group of $\pi '$ is transitive.
 Therefore, we get an element of the monodromy group which fixes $(S,H)$ and sends $(S,H')$ to $(S,H'')$ for any two points in the fiber $\pi^{-1}(S)$ different from $(S,H)$. Now, since $\Pi_g$ has already been proven to be transitive, so we get that $\Pi_g$ is 2-transitive.
\end{proof}

\smallskip

 \begin{prop} \label{ST3}
  $\Pi_g$ contains a simple transposition for $g=3$.
 \end{prop}
 \begin{proof}
From Lemma \ref{simpletrans}, to prove that $\Pi_g$ admits a simple transposition, it is enough to show that there is a point S in W such that the fiber of $\pi : J \to W$ above $S$ is $y_1,\cdots , y_n$ satisfying: 
\begin{enumerate}
\item $y_1$ corresponds to a rational integral curve with 2 nodes and 1 simple cusp as singularities, and $y_2,\cdots , y_n$ are points corresponding to rational nodal curves. 
\item $n = \deg \pi - 1$. 
\item $J$ is locally irreducible at $y_1$. 
\end{enumerate}

Let $H,H'$ be two distinct hyperplanes in $\mathbb P^3$ and let $L = H \cap H'$. Take $Y$ to be the union of $4$ general points on $L$. Then by Lemma \ref{genpoints3}, there exists a curve $C \subset H$ which is a canonically embedded integral rational curve of genus 3 corresponding to a general point of $Z_{2,1}$ such that $C \cap L = Y$. Lemma \ref{genpoints3} also gives us the existence of a curve $C' \subset H$ which is a canonically embedded integral rational curve of genus 3 orresponding to a general point of $Z_3$ such that $C' \cap L = Y$. Then by Lemma \ref{lemma2}, there exists a smooth $K3$ surface $S$ in $\mathbb P^3$ such that $S \cap H = C$ and $S \cap H' = C'$. This implies that for a surface $S$ corresponding to a general point of $W'$ (where $W'$ is defined with respect to $C$), there is a hyperplane section of $S$ corresponding to an integral rational nodal curve.
 Using Proposition \ref{twotrans3}, we have that the monodromy group of $\pi '$ is transitive. Therefore, we get that the fiber of a general $S$ in $W'$ will only consist of points $(S,H') \in J'$ with $S \cap H'$ rational nodal. So, this $S$ will satisfy property $1$.

By Proposition \ref{Mult}, we see that if property 1 is satisfied, then the curve $y_1 $ will be counted with multiplicity $\binom{5}{2}  /5 = 2$ and the rest of the curves being nodal curves will be counted with multiplicity $1$, which means that $n$ is exactly $1$ less than the number of rational curves in the linear system $|\mathcal O(1)|$ of a general surface, which is $\deg \pi$. Thus, property 2 is satisfied. 

Finally, we need to show that $J$ is locally irreducible at the point $(S,H)$ with $S \cap H = C$. This will follow from Proposition \ref{locirr} as follows:
Let $Y_H$, $Z$, $\psi: Y_H \to Z$ be as in the proof of Proposition \ref{trans3}. $Z$ has $Z_3$ as an open subset and hence locally irreducible at $C$ by Proposition \ref{locirr} (Note that $C$ corresponds to a point of $Z_3$). Recall that $\psi$ is smooth. Therefore $Y_H$ is also locally irreducible at the points of the fiber over $C$, and finally so is $J$ since $\eta: J \to (\mathbb P^3)^\vee$ makes $J$ into an \'{e}tale locally trivial fibre bundle over $(\mathbb P^3)^\vee$.

\end{proof}

\section{The case $g= 4$}
\subsection{Preliminary lemmas}

We will prove some lemmas similar to the ones in the genus 3 case.

\begin{lemma} \label{genusbound}
Let $C$ be an integral curve, and $\phi: C' \to C$ be a normalization of $C$. Let $p \in C$ be a point so that $C$ can be embedded in a smooth surface locally at $p$, and suppose that $\phi^{-1}(p)$ consists of $d$ distinct points of $C'$ for $d \ge 2$. Then $g(C) \ge g(C') + d(d-1)/2$, where $g(C)$ denotes the (arithmetic) genus of $C$.
\end{lemma}
\begin{proof}
We know that $$g(C) = g(C') + \sum_{P \in C} \textup{length}(\widetilde{ \mathcal O_P}/ \mathcal O_P)$$ where $\mathcal O_P$ is the completion of the local ring of $C$ at $P$ and $\widetilde{ \mathcal O_P}$ is its integral closure. So it is enough to prove that the term for $P = p$ is $\ge d(d-1)/2$.
Now, we have $$\mathcal O_p \to \widetilde{ \mathcal O_p} \cong \bigoplus_{i=1}^{d} k\llbracket x_i \rrbracket.$$ Then, since $C$ can be embedded in a smooth surface locally at $p$, we have $\mathcal O_p = k \llbracket x,y \rrbracket/ (f)$ for some $f$. In particular, we have that $\dim_k(\mathcal O_p/m_p^{d-1}) \le d(d-1)/2$. On the other hand, $\dim_k(k[x]/x^{d-1}) = d-1$ so we have $\dim_k(\widetilde{ \mathcal O_p}/m^{d-1}) = d(d-1)$. Therefore, $\textup{length}(\widetilde{ \mathcal O_p}/ \mathcal O_p) \ge d(d-1)/2$.
\end{proof}

\begin{lemma} \label{nodecusp}
Let $C$ be a curve in $\mathbb P^3$ which is the complete intersection of two hypersurfaces $X,Y$ of degrees $d,e$, where $d,e \ge 2$. Let $p \in C$ be such that it is a singular point of both $X$ and $Y$. Then $C$ does not have a planar singularity at $p$.
\end{lemma}

\begin{proof}
 Without loss of generality, we may assume that $\mathbb P^3 = \mathbb P[x:y:z:t]$, and $p$ is the point $[0:0:0:1]$, so we have $$\mathcal O_{C,P} \cong \frac {k[x,y,z]_{(x,y,z)}} {(f,g)} $$ with $f,g$ being polynomials of degree $d,e \ge 2$ with no linear term (since they have a singularity at $(0,0,0)$).
 
 But now note that $\dim m_P/m_P^2 = 3$, whereas for planar singularities we have this dimension always equal to $2$. The assertion follows.
\end{proof}

\begin{lemma} \label{prop2}
Let $C$ be an integral, genus $4$, canonically embedded curve in $\mathbb P^3$ of degree $6$. Then $C = Q \cap R$, where 

1. $Q,R$ are irreducible hypersurfaces of degrees $2,3$ respectively, in $\mathbb P^3$.

2. $Q$ is uniquely determined by $C$, and $R$ is uniquely determined up to a linear multiple of $Q$. 

3. $Q$ has at most one singular point, $\textup{Sing }R \subseteq \textup{Sing }C$ has at most $4$ singular points.

4. If $C$ has only planar singularities then $\textup{Sing }Q \cap \textup{Sing }R \cap \textup{Sing }C  = \O$.
\end{lemma}
\begin{proof}
We have $\omega_C \cong \mathcal O(1)|_C$, so $h^0(C, \mathcal O(2)|_C) = h^0(C, 2\omega) = 9$, whereas $h^0(\mathbb P^3, \mathcal O(2)) = 10$. Thus, there is a quadric hypersurface $Q$ which contains $C$. If $Q$ is reducible, then it has to be a union of hyperplanes, but that would mean that $C$ is a degenerate curve contradicting the fact that it is canonically embedded. Thus, $Q$ is irreducible, and it is unique because the intersection with any other quadric would have to have degree $4$ which is less than the degree of $C$.

On the other hand, we have $h^0(C, \mathcal O(3)|_C) = h^0(C, 3\omega) = 15$, whereas $h^0(\mathbb P^3, \mathcal O(3)) = 20$. Therefore, there is a cubic hypersurface $R$ containing $C$ which is not the union of $Q$ and a hyperplane. Now, $R$ is irreducible since $C$ is not contained in any hyperplane, and $Q$ is the only quadric hypersurface $C$ is contained in. Since $Q \cap R$ has degree $6$ and contains $C$, so it must equal $C$. Any irreducible quadric in $\mathbb P^3$ has at most one singular point, so $Q$ has at most one singular point.

Now, to prove that $R$ is uniquely determined up to a linear multiple of $Q$, we need only prove that the map $H^0(\mathbb P^3, \mathcal O(3)) \to H^0(C, \mathcal O(3)|_C)$ is surjective. Now, note that we have the exact sequence
$$ 0 \to \mathcal I_C(3) \to \mathcal O(3) \to \mathcal O(3)|_C \to 0$$
so it is enough to show that $H^1(\mathbb P^3, \mathcal I_C(3)) = 0$. But $C$ is the intersection of a quadric and cubic, so we have 
$$ 0 \to \mathcal O(-5) \to \mathcal O(-2) \oplus \mathcal O(-3) \to \mathcal I_C \to 0$$
Now, $H^1(\mathbb P^3, \mathcal O(n)) = H^2(\mathbb P^3, \mathcal O(m)) = 0$ for any $n,m$, so taking the long exact sequence on cohomology (after twisting by $\mathcal O(n)$), we get $H^1(\mathbb P^3, \mathcal I(n)) = 0$ for any $n$, so in particular $H^1(\mathcal I_C(3)) = 0$.

So, we may vary the cubic $r_0(\underline x)$ defining $R$ by a linear multiple of $q_0(\underline x)$ to obtain all cubics containing $C$. Thus, we may take the linear system $r(\underline x) = r_0(\underline x) + l(\underline x) q_0(\underline x)$. The base locus is all $\underline x$ such that $r_0(\underline x) = 0$ and $q_0(\underline x) = 0$, i.e. the points of $C$. Hence, by Bertini's theorem, a general such $R$ only has singularities on $C$. Singular points of $R$ lying on $C$ will also be a singular point of $C$ (as may be seen from the Jacobian criterion). So we get an $R$ which may only have singularities on singular points of $C$.

Finally, the last statement about the intersection of singularities follows directly from Lemma \ref{nodecusp}.
\end{proof}

\begin{lemma}[Existence of a K3 surface in $\mathbb P^4$ with given hyperplane sections] \label{prop3}
Let $H,H'$ be two distinct hyperplanes in $\mathbb P^4$, and let $C \subset H$, $C' \subset H'$ be integral curves of genus $4$ which are respectively canonically embedded in $H$ and $H'$. Let $Q$(resp. $Q'$) be the unique quadric hypersurface in $H$(resp. $H'$) containing $C$(resp. $C'$). Assume that 

1. $C' \cap H = C \cap H'$ as subschemes of $H \cap H'$, and $Q' \cap H = Q \cap H'$ as subschemes of $H \cap H'$. 

2. $C,C'$ have only planar singularities.

3. $\textup{Sing }Q' \cap \textup{Sing }Q = \O$ and $\textup{Sing }C \cap \textup{Sing }C'= \O$.

Then there exists a smooth K3 surface $S \subset \mathbb P^4$ which is the intersection of a quadric and a cubic, and such that $S \cap H = C, \ S\cap H' = C'$.
\end{lemma}

\begin{proof}
Let $\mathbb P^4 = \mathbb P[x_0 : x_1:x_2:x_3:x_4]$. We may assume without loss of generality that $H = \{x_0 = 0 \}, H' = \{ x_4 = 0 \}$.

We have $Q \cap H' = Q' \cap H$. So, $Q = V(g_1(x_1,x_2,x_3,x_4)), Q' = V(g_2(x_0, x_1, x_2, x_3)) $ with $g_1(x_1,x_2,x_3,0) = g_2(0,x_1,x_2,x_3)$ where $g_1, g_2$ are degree 2 homogeneous polynomials. 

Consider the quadric hypersurfaces $T_{1,a} = V(g) \subset \mathbb P^4$ given by quadrics $g(x_0, x_1,x_2,x_3,x_4)$ such that $g(0, x_1,x_2,x_3,x_4) = g_1, g(x_0, x_1,x_2,x_3,0) = g_2$, i.e. $g = a \cdot x_0x_4 + g_3$ if $g_3 = g_1 + g_2 - g_2(0, x_1,x_2,x_3)$.

This is a linear system with base points being the points $[x_0 : x_1:x_2:x_3:x_4]$ such that $x_0x_4 =0, g_3 = 0$, i.e. $Q \cup Q'$. Hence, by Bertini's theorem, a general such quadric will be smooth away from the points of $Q \cup Q'$. Thus, for all but finitely many $a$, $T_{1,a}$ will be smooth away from the points of $Q \cup Q'$.

Now, if a singular point $P$ of $T_{1,a}$ has $x_0 = 0$, then $$\frac{\partial g}{\partial x_i}(P) = \frac{\partial g_1}{\partial x_i}(P) = 0, \textup{ for } i = 1,2,3,4$$
Thus, if $P$ has $x_0 =0$ then $P \in \textup{Sing } Q$. Similarly, if $P$ has $x_4 =0$ then $P \in \textup{Sing } Q'$. Thus, if $x_0(P) = 0$ then $x_4(P) \neq 0$ since $\textup{Sing }Q \cap \textup{Sing } Q' = \O$. Now, $$\frac{\partial g}{\partial x_0}(P) = a \cdot x_4(P) + \frac{\partial g_3}{\partial x_0}(P) = 0,$$
therefore, we for all but finitely many values of $a$, $T_{1,a}$ has no singularities with $x_0 = 0$. Similarly, for all but finitely many values of $a$, $T_{1,a}$ has no singularities with $x_4 = 0$. 

Hence, for all but finitely many values of $a$, $T_{1,a}$ is a smooth quadric in $\mathbb P^4$.

Let $Q_{H'} = Q \cap H' = Q' \cap H$. Consider cubics $R \subset H, R' \subset H'$ containing $C,C'$ respectively. Then $R_{H'} = R \cap H'$ is a cubic in $H \cap H'$ such that $R_{H'} \cap Q_{H'}$ is a fixed length $6$ scheme $C \cap H'$. So, the degree 3 polynomial defining any other cubic with this property will differ from the one for $R_{H'}$ by a linear multiple of the polynomial defining $Q_{H'}$. Hence, $R_{H'}$ will differ from $R'_{H}$ by a linear multiple of $Q_{H'}$. The upshot is that we may add this linear multiple to the original $R$ to ensure that $R \cap H' = R' \cap H$.

So, we have $R = V(h_1(x_1,x_2,x_3,x_4)), R' = V(h_2(x_0, x_1, x_2, x_3)) $ with $h_1(x_1,x_2,x_3,0) = h_2(0,x_1,x_2,x_3)$.

Then, consider the linear system on $T_{1,a}$ given by cubics $T_{2,b,c} = V(h(x_0, x_1,x_2,x_3,x_4))$ such that $h(0, x_1,x_2,x_3,x_4) = h_1, h(x_0, x_1,x_2,x_3,0) = h_2$, i.e. $h = b \cdot x_0^2x_4 + c \cdot x_0 x_4^2 + h_3$ if $h_3 = h_1 + h_2 - h_2(0, x_1,x_2,x_3)$. 

The base locus of this linear system are points $[x_0 : x_1:x_2:x_3:x_4]$ on $T_{1,a}$ such that $x_0x_4 =0, h_3 = 0$, i.e. $C \cup C'$. Hence, by Bertini's theorem, a general such cubic $T_{2,b,c}$ will intersect $T_{1,a}$ in a surface which is smooth away from the points of $C \cup C'$.

Now, any point $P$ contained in $C$ will have $x_0 = 0$, so the Jacobian of $V(g,h)$ at $P$ will be 
\[
  \left[ {\begin{array}{ccccc}
  a \cdot x_4(P) +  \dfrac{\strut \partial g_3}{\strut \partial x_0}(P) &  \dfrac{\strut \partial g_1}{\strut \partial x_1}(P) &  \dfrac{\strut \partial g_1}{\strut \partial x_2}(P)  &  \dfrac{\strut \partial g_1}{\strut \partial x_3}(P) &  \dfrac{\strut \partial g_1}{\strut \partial x_4}(P) \\
  c \cdot x_4^2(P) +  \dfrac{\strut \partial h_3}{\strut \partial x_0}(P) &  \dfrac{\strut \partial h_1}{\strut \partial x_1}(P) &  \dfrac{\strut \partial h_1}{\strut \partial x_2}(P)  &  \dfrac{\strut \partial h_1}{\strut \partial x_3}(P) &  \dfrac{\strut \partial h_1}{\strut \partial x_4}(P) \\
  \end{array} } \right]
\]
Thus, looking at the last $4$ columns, if $P$ is a singular point of $V(g,h)$ then it must be a singular point of $Q \cap R = C$. Now $C$ is at worst nodal or cuspidal, so $P$ cannot be simultaneously a singular point of both $Q$ and $R$. Thus, there is a non-zero column among the last $4$ columns and so if $P \not \in H'$ i.e. it has $x_4 \neq 0$, then we may choose $a,c$ such that the first column is linearly independent of that column, and so if $a,b,c$ are general then no point of Sing $C$ not lying on $H'$ will be a singular point of $V(g,h)$.

Similar analysis for $x_4=0$ leads us to having that if $a,b,c$ are general then $V(g,h)$ is smooth at the points of Sing $C'$ not lying on $H$. Thus, the only singularities that may occur are on Sing $C \cap$ Sing $C'$ which is empty. Thus, for $a,b,c$ general we get that $V(g,h)$ is smooth.

So we conclude that $S = T_{1,a} \cap T_{2,b,c}$ is smooth, which concludes the proof, since $S \cap H =C, S \cap H' =C'$.  

\end{proof}

\begin{lemma} \label{hyperplanebound}
Let $C \subset \mathbb P^3$ be a canonically embedded curve of genus $4$. Suppose that $C$ has only nodes or simple cusps as singularities and that the unique quadric containing $C$ is smooth. Let $V \subseteq (\mathbb P^3)^\vee$ be the subspace of hyperplanes $H$ so that $C \cap H$ contains a component of length $\ge 3$. Then $V$ has codimension $\ge 2$ in $(\mathbb P^3)^\vee$.
\end{lemma}
\begin{proof}
Let $p \in C$ be a node or a simple cusp. Then Lemma \ref{locallengthcalc} gives us that a general hyperplane $H$ passing through $p$ will have an intersection of length $2$ at $p$. So, the hyperplanes $H$ passing through $p$ and having an intersection of length $\ge 3$ at $p$ will have codimension $\ge 2$ in $(\mathbb P^3)^\vee$. 

For a smooth point $p$ on $C$, we can talk about the multiplicity of $C \cap H$ at $p$. A hyperplane $H$ containing $p$ will have $\textup{mult}_p(C \cap H) \ge 2$ if and only if $H$ contains the tangent line $L$ to $C$ at $p$. Now, by Lemma \ref{multiplicity}, for a general hyperplane $H$ containing $L$, we have $\textup{mult}_p(H \cap C) = \textup{mult}_p(L \cap C)$. Now, let $Q$ be the unique quadric containing $C$. Then $Q$ is smooth. So, if $L \cap Q$ is a proper intersection for the tangent line $L$ to $C$ at a smooth point $p$ of $C$, then for a general hyperplane $H$ which contains $L$, $$\textup{mult}_p(H \cap C) = \textup{mult}_p(L \cap C) \le \textup{mult}_p(L \cap Q) \le 2.$$ Finally, it suffices to observe that the tangent line $L$ to $C$ is contained in $Q$ only for finitely many points, since $C$ does not contain a ruling of $Q$. This finishes the proof. 
\end{proof}

\begin{lemma} \label{integral4}
 Let $H$ be a hyperplane in $\mathbb P^4$, and let $C \subset H$ be a canonically embedded integral curve of genus $4$. Suppose that $C$ has only nodes or simple cusps as singularities and that the unique quadric surface in $H$ containing $C$ is smooth. Then there exists a smooth K3 surface $S \subset \mathbb P^4$ which is the intersection of a quadric and a cubic, and such that $S \cap H = C$. The general such $S$ will satisfy that $S \cap H'$ is integral for all hyperplanes $H' \subset \mathbb P^4$.
\end{lemma}
\begin{proof}
We can prove the existence of a smooth K3 surface $S \subset \mathbb P^4$ which is the intersection of a quadric and a cubic, and such that $S \cap H = C$ similar as in Lemma \ref{prop3}. So we need to now prove that the general such $S$ will satisfy $S \cap H'$ is integral for all hyperplanes $H' \subset \mathbb P^4$.

Let the coordinates of $\mathbb P^4$ be $[x:y:z:v:w]$, $H$ be the hyperplane $\{w = 0\}$, and $C$ be a integral non-degenerate curve in $H$ which is the intersection of the quadric and cubic in $H$ given by equations $\{f_0(x,y,z,v)=0 \}$ and $\{ g_0(x,y,z,v) = 0 \}$ respectively. We assume that $C$ has only nodes or simple cusps as singularities. Note that $C$ is integral and non-degenerate, so $C \cap H'$ will be a length $6$ scheme for any hyperplane $H' \neq H$ in $\mathbb P^4$.

Let $f(x,y,z,v,w) = f_0(x,y,z,v) + w f_1(x,y,z,v,w)$ be the general quadric equation satisfying $f(x,y,z,v,0) = f_0(x,y,z,v)$. The space of such $f$ is isomorphic to $\mathbb A^{5}$.

Let $g(x,y,z,v,w) = g_0(x,y,z,v) + w g_1(x,y,z,v,w)$ be the general quadric equation satisfying $g(x,y,z,v,0) = g_0(x,y,z,v)$. The space of such $g$ is isomorphic to $\mathbb A^{15}$. We say that $f,g$ \textit{extends} $f_0,g_0$ if they are of the above form. Let $S = V(f,g) \subset \mathbb P^4$.

\begin{itemize}
    \item Suppose $S$ contains the line $L = \{y=z=v=0 \}$. Let $H'$ be the hyperplane $\{v = 0\}$. We assume that $L \cap H \subset C \cap H'$. Then setting $y=z=v=0$, one gets $f_1(x,0,0,0,w) = 0$ and $g_1(x,0,0,0,w) = 0$. Therefore, $f_1(x,y,z,v,w) = f_2(y,z,v)$ for some homogeneous $f_2$ of degree 1 and $g_1(x,y,z,v,w) = g_2(y,z,v) + xg_3(y,z,v) + wg_3(y,z,v)$ for some homogeneous $g_2,g_3,g_4$ of degree $2,1,1$ respectively. So, we get that the subvariety $D_L$ of $\mathbb A^5 \times \mathbb A^{15}$ parametrizing $f,g$ so that $S$ contains $L$ is isomorphic to $\mathbb A^3 \times \mathbb A^{12}$. 
    
    Now, let the space of lines in $\mathbb P^4$ be $\mathscr L = \mathbb{G}(1,4) = \mathrm{Gr}(2,5)$. Consider the subspace $\mathscr W_1$ of $\mathbb A^5 \times \mathbb A^{15} \times \mathscr L \times (\mathbb P^4)^{\vee}$ such that $$\mathscr W_1 = \{(f,g,L,H') | (f,g) \text{ extends } (f_0,g_0), \ L \subset S \cap H' \}.$$ Firstly, observe that $(f,g,L,H') \in \mathscr W_1$ implies that $H' \neq H$, so we assume $H' \neq H$ in what follows. Note that $L \subset S \cap H'$ implies that $L \cap H \subset C \cap H'$, so the projection map $\pi_{34} : \mathscr W_1 \to \mathscr L \times (\mathbb P^4)^{\vee}$ factors through the space $(L,H')$ so that $L \subset H'$ and $L$ contains a point of $C \cap H'$. Now, the space of lines in a fixed $H'$ which contains a point of $C \cap H'$ is of dimension 2, and the hyperplanes themselves vary in a space of dimension 4, so the space $(L,H')$ so that $L \subset H'$ and $L$ contains a point of $C \cap H'$ has dimension $6$. Also, the fibre of the projection map $\pi_{34}$ over $(L,H')$ will be isomorphic to the variety $D_L$ which has dimension $15$ as computed above. Thus, the dimension of $\mathscr W_1$ is at most $15+6 = 21$. The space of hyperplanes $H'$ containing a fixed line $L$ is of dimension 2, so the dimension of image of the projection map $\pi_{12} : \mathscr W_1 \to \mathbb A^5 \times \mathbb A^{15}$ is of dimension at most $19$. 
    
    \item Let $H'$ be the hyperplane $\{v = 0\}$. Suppose $S$ contains the conic $T$ in $H'$ given by $\{v=0, y=0, t(x,z,w) =0 \}$, where $t$ is a degree $2$ homogeneous polynomial. We assume that $T \cap H \subset C \cap H'$. Then setting $y = v= 0$, we get that $f(x,0,z,0,w) = \lambda t$ for some non-zero constant $\lambda$ determined by $f_0$ and $g(x,0,z,0,w) = (ax + bz + cw)t$ is a (non-zero) linear multiple of $t$ where $a,b$ are determined by $g_0$. So, we get that the subvariety $E_T$ of $\mathbb A^5 \times \mathbb A^{15}$ parametrizing $f,g$ so that $S$ contains $T$ is isomorphic to $\mathbb A^2 \times \mathbb A^{10}$. 
    
    Now, let the space of integral degree 2 curves in $\mathbb P^4$ be $\mathscr T$. Consider the subspace $\mathscr W_2$ of $\mathbb A^5 \times \mathbb A^{15} \times \mathscr T \times (\mathbb P^4)^{\vee}$ such that $$\mathscr W_2 = \{(f,g,T,H') | (f,g) \text{ extends } (f_0,g_0), \ T \subset S \cap H' \}.$$  Firstly, observe that $(f,g,T,H') \in \mathscr W_2$ implies that $H' \neq H$, so we assume $H' \neq H$ in what follows. Note that $T \subset S \cap H'$ implies that $T \cap H \subset C \cap H'$, so the projection map $ \pi_{34} : \mathscr W_2 \to \mathscr T \times (\mathbb P^4)^{\vee}$ factors through the space of $(T,H')$ so that $T \subset H'$ and $T \cap H \subset C \cap H'$. We know that any degree 2 curve $T$ in $H'$ will lie in a plane $H''$ which will be unique since $T$ is integral. So we can consider the space $$X = \{(T,H'',H')| T \subset H'' \subset H', T \cap H \subset C \cap H' \}$$ seen as a subspace of $\mathscr T \times \mathbb G(2,4) \times (\mathbb P^4)^{\vee}$. Consider the third projection $p: X \to (\mathbb P^4)^{\vee}$, and fix a $H' \neq H$. 
    \begin{itemize}
        \item If $C \cap H'$ is curvilinear, then there are only finitely many length $2$ subschemes of $C \cap H'$, each lying on a unique line $\ell$. Thus, if we fix $T \cap H'$ we also fix $\ell = H'' \cap H$. The space of planes $H''$ in $H$ containing $\ell$ is of dimension 1, and the space of conics $T$ in $H''$ so that $T \cap \ell$ is fixed is of dimension 3. Thus, we get that in this case $p^{-1}(H)$ is $4$ dimensional.
        \item If $C \cap H'$ is not curvilinear, then $H'$ must be the full tangent space at some singular point of $C$. We know that there is a 1-dimensional space of embeddings of $k[x]/(x^2)$ at a singularity of $C$ (since they are either nodes or cusps) and each embedding defines a unique line, and so we get a $1$-dimensional space of lines $\ell$. The rest of computation is the same as the previous case, so we get that in this case $p^{-1}(H)$ is $5$ dimensional.
    \end{itemize}
     Now, the space of hyperplanes $H'$ containing a fixed plane in $\mathbb P^4$ is of dimension 1. So, we have that the space of hyperplanes $H'$ so that $C \cap H'$ is not curvilinear is of dimension 1. So, $p^{-1}(H')$ is $5$ dimensional over this 1-dimensional space, and $4$-dimensional over the $4$-dimensional complement, hence $X$ is of dimension at most $8$. So, the image of $\pi_{34}$ is of dimension at most $8$ and the fibre of $\pi_{34}$ over $(T,H')$ is isomorphic to $E_T$ which is $12$ dimensional, thus $\mathscr W_2$ is of dimension at most $12 + 8 = 20$. The space of hyperplanes $H'$ containing a fixed conic $T$ is of dimension 1, so the dimension of image of the projection map $\pi_{12} : \mathscr W_2 \to \mathbb A^5 \times \mathbb A^{15}$ is of dimension at most $19$. 
    
    \item Let $H'$ be the hyperplane $\{v = 0\}$. Suppose $S$ contains a rational normal curve $V$ in $H'$, which is the image of the degree 3 map $q : \mathbb P^1 \to H'$ given by $[a:b] \mapsto [q_1(a,b):q_2(a,b):q_3(a,b):q_4(a,b)]$ where $\{q_i\}$ are linearly independent degree 3 homogeneous polynomials in $a,b$. We assume that $V \cap H \subset C \cap H'$. Then plugging in these coordinates we get that $$f(q_1(a,b),q_2(a,b),q_3(a,b),0,q_4(a,b)) = 0, g(q_1(a,b),q_2(a,b),q_3(a,b),0,q_4(a,b)) = 0.$$ Note that (writing $f(q_1, \cdots ,q_n)$ as short for $f(q_1(a,b),\cdots, q_n(a,b))$ for any polynomial $f$ in $n$ variables and any polynomials $q_i$ in $a,b$) $$f_0(q_1,q_2,q_3,0) = q_4(a,b) u(a,b), g_0(q_1,q_2,q_3,0) = q_4(a,b) v(a,b)$$ are multiples of $q_4(a,b)$ since $V \cap H \subset C \cap H'$. So we get that $$f_1(q_1,q_2,q_3,0,q_4) = - u(a,b), g_1(q_1,q_2,q_3,0,q_4) = -v(a,b).$$ So, $f_1(x,y,z,0,w)$ will be a fixed linear polynomial since $q_1, \cdots, q_4$ are linearly independent polynomials. Also, if there is a solution $\tilde g_1$ to $g_1(q_1,q_2,q_3,0,q_4) = -v(a,b)$, then $g_1(x,y,z,0,w) - \tilde g_1$ will be a quadric which will vanish on $V$ and hence must be in the linear span of the three independent quadrics which contain $V$. So, we get that the subvariety of $\mathbb A^5 \times \mathbb A^{15}$ parametrizing $f,g$ so that $S$ contains $V$ is isomorphic to $\mathbb A^1 \times \mathbb A^{8}$. 
    
    Now, let the space of integral degree 3 curves in $\mathbb P^4$ be $\mathscr V$. Consider the subspace $\mathscr W_3$ of $\mathbb A^5 \times \mathbb A^{15} \times \mathscr V \times (\mathbb P^4)^{\vee}$ such that $$\mathscr W_3 = \{(f,g,V,H') | (f,g) \text{ extends } (f_0,g_0), \ V \subset S \cap H' \}.$$  Firstly, observe that $(f,g,V,H') \in \mathscr W_3$ implies that $H' \neq H$, so we assume $H' \neq H$ in what follows. Note that $V \subset S \cap H'$ implies that $V \cap H \subset C \cap H'$, so the projection map $ \pi_{34} : \mathscr W_3 \to \mathscr V \times (\mathbb P^4)^{\vee}$ factors through the space $Y$ of $(V,H')$ so that $V \subset H'$ and $V \cap H \subset C \cap H'$. Consider the map 
    \begin{align*}
    p: Y &\to \mathrm{Hilb}_3(\mathbb P^4) \times (\mathbb P^4)^{\vee} \\     
    (V,H') & \mapsto (V \cap H, H')
    \end{align*}
    Fix a $H' \neq H$. 

    \begin{itemize}
        \item If we impose $V \cap H'$ equal some fixed three distinct points, then we get that the space of rational normal curves containing these points is of dimension $6$ as can be seen as follows: We may assume that $H' = \{v=0 \}$ and the three points are given by $A=[1:0:0:0:0], B =[0:1:0:0:0], C =[0:0:1:0:0]$ if they are not collinear, and if they are collinear then we may assume they lie on the line $\{z=v=w=0\}$ and are given by $A =[1:0:0:0:0], B = [0:1:0:0:0], C = [1:1:0:0:0]$. So, we can identify this with the space of degree 3 maps 
        \begin{align*}
            q: \mathbb P^1 &\to \mathbb P^3 \\ 
            [a:b] & \mapsto [q_1(a,b): q_2(a,b) : q_3(a,b) : q_4(a,b)] \\
            [1:0] & \mapsto A\\
            [0:1] & \mapsto B\\
            [1:1] & \mapsto C
        \end{align*}
        (Here we identify $A,B,C$ with their projections to $H'$). If $q_i(a,b) = \sum_j q_{i,j}a^jb^{3-j}$, then we get the conditions $q_{4,0} = q_{4,3} = 0, q_{3,0} = q_{3,3} = 0, q_{2,0} = 0, q_{1,3} = 0$ from the first two points, and the conditions $q_{1,0} + q_{1,1} + q_{1,2} =0, q_{2,1} + q_{2,2} + q_{2,3} =0, q_{4,1} + q_{4,2} = 0$ for $C = [0:0:1:0:0]$, and the conditions $q_{1,0} + q_{1,1} + q_{1,2} = q_{2,1} + q_{2,2} + q_{2,3}, q_{3,1} + q_{3,2} = 0, q_{4,1} + q_{4,2} = 0$ for $C = [1:1:0:0:0]$.
        
        So, we get 3 independent conditions on the coefficients of $q_i$ for every point, and hence we get a $6$ dimensional space.
        \item If we impose $V \cap H'$ equal a fixed length 2 scheme supported at a point $A$ plus another fixed point $B$, then we get that the space of rational normal curves containing these points is of dimension $7$ as can be seen as follows: We may assume that $H' = \{v=0 \}$ and the points are given by $A=[1:0:0:0:0], B =[0:1:0:0:0]$. So, we can identify this with the space of degree 3 maps
        \begin{align*}
            q: \mathbb P^1 &\to \mathbb P^3 \\ 
            [a:b] & \mapsto [q_1(a,b): q_2(a,b) : q_3(a,b) : q_4(a,b)] \\
            [1:0] & \mapsto [1:0:0:0]\\
            [0:1] & \mapsto [0:1:0:0]
        \end{align*}
        with $q_4(1,b)$ having a double root at $0$. If $q_i(a,b) = \sum_j q_{i,j}a^jb^{3-j}$, then we get the conditions $q_{4,0} =q_{4,1} = q_{4,3} = 0, q_{3,0} = q_{3,3} = 0, q_{2,0} = 0, q_{1,3} = 0$. Also, we need to go modulo the 1-dimensional automorphisms of $\mathbb P^1$ which fix $[1:0]$ and $[0:1]$.
        So, we get a $7$ dimensional space.
        \item If we impose $V \cap H'$ equal a fixed length 3 scheme supported at a point $A$, then we get that the space of rational normal curves containing these points is of dimension $8$ as can be seen as follows: We may assume that $H' = \{v=0 \}$ and the point is given by $A=[1:0:0:0:0]$. So, we can identify this with the space of degree 3 maps
        \begin{align*}
            q: \mathbb P^1 &\to \mathbb P^3 \\ 
            [a:b] & \mapsto [q_1(a,b): q_2(a,b) : q_3(a,b) : q_4(a,b)] \\
            [1:0] & \mapsto [1:0:0:0]
        \end{align*}
        with $q_4(1,b)$ having a triple root at $0$. If $q_i(a,b) = \sum_j q_{i,j}a^jb^{3-j}$, then we get the conditions $q_{4,0} = q_{4,1} = q_{4,2} = 0, q_{3,0}, q_{2,0} = 0$. Also, we need to go modulo the 2-dimensional automorphisms of $\mathbb P^1$ which fix $[1:0]$.
        So, we get a $8$ dimensional space. 
    \end{itemize}
    Note that the image of $p$ is contained in the space $G = \{(U,H')\}$ where $H'$ is a hyperplane in $\mathbb P^4$ and $U$ is a length 3 curvilinear scheme contained in $C \cap H'$. We can break this space $G$ into three subschemes: $U$ is the union of $3$ distinct points, $U$ is the union of a point and a connected length 2 scheme, and $U$ is a connected length 3 scheme. Call these $G_1,G_2,G_3$ respectively. Then we know that the fibre of $p$ has dimension $6$ over $G_1$, dimension $7$ over $G_2$ and dimension $8$ over $G_3$. 
    \begin{itemize}
        \item Consider the second projection of $G_1$, $p_1: G_1 \to (\mathbb P^4)^\vee $. Then $p_1$ has 0 dimensional fibres since there are only finitely many ways to choose 3 points in $C \cap H'$. Thus, $G_1$ has dimension at most $4$.
        
        \item Consider the second projection of $G_2$, $p_2: G_2 \to (\mathbb P^4)^\vee $. Then the image of $p_2$ consists of $H'$ which are either tangent to $C$, or pass through a singular point of $C$. So, the image of $p_2$ is at most $3$ dimensional. Also, $p_2$ has a 0 dimensional fibre over $H'$ if $C \cap H'$ is curvilinear, and $p_2^{-1}(H')$ is at most 1 dimensional if $C \cap H'$ is not curvilinear (by Lemma \ref{embedcalc}). The space of hyperplanes $H'$ so that $C \cap H'$ is not curvilinear is of dimension 1. Hence, $G_2$ has dimension at most $3$.
        
        \item Consider the second projection of $G_3$, $p_3: G_2 \to (\mathbb P^4)^\vee $. Then the image of $p_3$ consists of $H'$ which intersect $C$ at some point $p$ with multiplicity $\ge 3$. So, by Lemma \ref{hyperplanebound}, the image of $p_3$ is at most $2$ dimensional. Also, $p_3$ has a 0 dimensional fibre over $H'$ if $C \cap H'$ is curvilinear, and $p_3^{-1}(H')$ is at most 1 dimensional if $C \cap H'$ is not curvilinear (by Lemma \ref{embedcalc}). The space of hyperplanes $H'$ so that $C \cap H'$ is not curvilinear is of dimension 1. Hence, $G_3$ has dimension at most $2$.
    \end{itemize}
    Thus, $Y$ has dimension at most $10$, and the fibre of $\pi_{34}$ over $(V,H')$ was seen to be 9 dimensional, so we get that $\mathscr W_3$ is at most $19$ dimensional. So, the image of the projection map $\pi_{12}: \mathscr W_3 \to \mathbb A^5 \times \mathbb A^{15}$ has dimension at most $19$.
\end{itemize}    
Now, note that for any hyperplane $H'$, $S \cap H'$ is a degree $6$ non-degenerate curve in $H'$, so if it is reducible then it must contain either a degree $1$ curve in $H'$, which is a line or a degree $2$ curve in $H'$, which is a conic in a plane in $H'$ or a degree $3$ non-degenerate curve in $H'$, which is a rational normal curve. Hence, the surfaces $S$ for which $S \cap H'$ is reducible must correspond to $(f,g)$ which are in the image of $\mathscr W_1$ or in the image of $\mathscr W_2$ or in the image of $\mathscr W_3$. Since all three of these are not dominant, so we get that the general $(f,g)$ which extends $(f_0,g_0)$ must satisfy that $S \cap H'$ is integral for all $H'$.
\end{proof}

\begin{lemma} \label{nobadcurves}
 Let $H$ be a hyperplane in $\mathbb P^4$, and let $C \subset H$ be a canonically embedded integral curve of genus $4$. Suppose that either $C$ represents a general point of $Z_4$ or a general point of $Z_{3,1}$. Then there exists a smooth K3 surface $S \subset \mathbb P^4$ which is the intersection of a quadric and a cubic, and such that $S \cap H = C$. The general such $S$ will satisfy that
 \begin{enumerate}
     \item for all hyperplanes $H' \subset \mathbb P^4$, $C' = S \cap H'$ is integral.
     \item for any hyperplane $H' \subset \mathbb P^4$ so that $C' = S \cap H'$ is rational integral and has a cuspidal singularity, $C'$ has exactly one simple cusp and three simple nodes as singularities and the unique quadric surface $Q'$ in $H'$ containing $C'$ is smooth.
 \end{enumerate}
\end{lemma}

\begin{proof}
 By Lemma \ref{integral4}, we already know the existence, and the part $(1)$  i.e. for all hyperplanes $H' \subset \mathbb P^4$, $C' = S \cap H'$ is integral. So from now on we will only consider integral curves.
 
 For part $(2)$, we claim that for any hyperplane $H'$ of $\mathbb P^4$, inside the space $Z_{H'}$ of canonically embedded rational integral curves of genus $4$ in $H'$, we have
 \begin{itemize}
     \item \textit{Claim 1}: The subspace $Z'_{H'}$ of curves having a cusp at some point is irreducible of codimension 1 and the curve corresponding to a general point of $Z'_{H'}$ has a simple cusp and 3 simple nodes as its singularities.
     \item \textit{Claim 2}: The subspace $Z''_{H'}$ of curves having a cusp and being contained in a singular quadric has codimension $ \ge 2$.
\end{itemize} 
 To see this, we note that if $\mathcal Q$ is the space of irreducible quadric hypersurfaces in $H'$ then there is a map $\alpha: Z_{H'} \to \mathcal Q$ sending the curve to the unique quadric hypersurface containing it. Then $\alpha$ is a dominant map. Let $R_X$ be the fibre of $\alpha$ over a quadric $X$. Now, let $U_X$ be the subspace of $\textup{Hom}(\mathbb P^1,X)$ consisting of maps $g: \mathbb P^1 \to X$ so that $g^* \mathcal O_X(1) = \mathcal O_{\mathbb P^1}(6)$ and $g(\mathbb P^1)$ is canonically embedded in $H'$. Consider the map $\beta : U_X \to R_X$ which sends a map to its image. This is a dominant map with the fibre over a curve $C$ being given by Aut$(\mathbb P^1) \times$Aut$(C)$ which is 3-dimensional.
 
 For the second claim, we begin by observing that the subspace of $\mathcal Q$ consisting of singular $X$ has codimension 1. Now we want to show that the subspace of $R_X$ corresponding to curves having a cusp at some point has codimension $1$ for singular $X$. To prove this we will show that the inverse image of this subspace under $\beta$ has codimension $1$ in $U_X$.
 
 $X$ is an irreducible singular quadric in $H' \cong \mathbb P^3$, so we may choose coordinates $[x:y:z:t]$ so that $X = V(x^2 - yz)$. Now, if the map $g: \mathbb P^1 \to \mathbb P^3$ is given by $[a:b] \mapsto [p(a,b): q(a,b): r(a,b): s(a,b)]$ where $p,q,r,s$ are degree $6$ homogeneous polynomials with no common zero then the condition for the image of $g$ to be contained in $X$ is that $p(a,b)^2 = q(a,b)r(a,b)$. If the gcd of $q,r = u$ (well defined up to a scalar), then $q = uv^2$, $r = uw^2$ for some polynomials $v,w$ and $p = uvw$. So, for $d=0,1,2,3$, we get different components by fixing the degree of $u = 2d$, and considering the space of polynomials $u,v,w,s$ so that $u,v,w,s$ are homogeneous of degree $2d,3-d,3-d,6$ respectively, with $\{v,w \}$ having no common zero and $\{uvw, s \}$ having no common zero. Note that the space of such polynomials has dimension $16$. The space $U_X$ is the obtained by going modulo the action of $\mathbb G_m \times \mathbb G_m$ where $(\lambda, \mu) \cdot (u,v,w,s) = (\lambda u, \mu v, \mu w, \lambda \mu^2 s)$. Let $$u = \sum_{i+j = \deg u} u_i a^i b^j$$ and similar for $v,w,s$. 
 
  
 Now, we check what the condition on $g$ is for it to be not an immersion at a point $P$. Let $P = [0:1]$, then the conditions are (we assume $u_1 = 0$ if $d=0$) 
 \begin{enumerate}
     \item $(u_1v_0w_0 + u_0v_1w_0 + u_0v_0w_1)s_0 = u_0v_0w_0 s_1$.
     \item $    (u_1v_0^2 + 2u_0v_0v_1)s_0 = u_0v_0^2s_1$.
     \item $    (u_1w_0^2 + 2u_0w_0w_1)s_0 = u_0w_0^2s_1$.
 \end{enumerate}
If $d, v_0, w_0 \neq 0$, then 
$$    (u_1v_0^2 + 2u_0v_0v_1)s_0 = u_0v_0^2s_1 \iff (u_1v_0 + 2u_0v_1)s_0 = u_0v_0s_1 $$
$$    (u_1w_0^2 + 2u_0w_0w_1)s_0 = u_0w_0^2s_1 \iff (u_1w_0 + 2u_0w_1)s_0 = u_0w_0s_1$$
These two are independent irreducible conditions and the equation (1) is dependent on these two, so we get a codimension 2 space. If $d \neq 0$, but $v_0 = 0$, then $w_0 \neq 0$, so we have the conditions $u_0v_1s_0 = 0$ and $(u_1w_0 + 2u_0w_1)s_0 = u_0w_0s_1$. The latter defines an irreducible space and is not contained in the former, so we get a codimension 3 space in this case. Similarly $w_0=0$ also gives us a codimension 3 space. If $d = 0$, then $u_0 \neq 0$ and $u_1 = 0$, so the conditions are
 \begin{enumerate}
     \item $(v_1w_0 + v_0w_1)s_0 = v_0w_0 s_1$.
     \item $    2v_0v_1s_0 = v_0^2s_1$.
     \item $    2w_0w_1s_0 = w_0^2s_1$.
 \end{enumerate}
 and similar arguments as before will give us a codimension 2 space. Since $P$ is allowed to vary in $\mathbb P^1$, we get a codimension 1 space of maps which are not immersions. This proves the second claim.
 
 Now, for the first claim, it is enough to consider $U_X$ for smooth quadrics $X$ and prove that the subspace consisting of maps which are not immersions is irreducible and that the general point of this subspace is a map which is not an immersion at a single point and the image has a simple cusp at the image of that point and the only other singularities are simple nodes. Now, $X$ is a smooth quadric surface in $\mathbb P^3$ so $X \cong \mathbb P^1 \times \mathbb P^1$. Now, we consider the space $V$ of tuples $(g,P,x)$ consisting of maps $g: \mathbb P^1 \to X$ so that the image of $g$ is in $\mathcal O(3) \boxtimes \mathcal O(3)$ (this is needed for it to be canonically embedded) and points $P \in \mathbb P^1$ and $x \in X$ so that $g(P) = x$. Then $V$ is an etale-local fibre bundle over $\mathbb P^1 \times X$. If $g$ maps $[a:b] \mapsto [p(a,b): q(a,b)] \times [r(a,b): s(a,b)]$ where $p,q,r,s$ are degree $3$ homogeneous polynomials with no common zero, then the condition for $g([0:1]) = ([0:1] \times [0:1])$ is that $p_0 = r_0 = 0$, where $p(a,b) = p_0b^3 + p_1 ab^2 + p_2 ab^2 + p_3 a^3$ and similar for $q,r,s$. Now, the condition for $g$ to be not an immersion at $[0:1]$ is that $p_1 = r_1 = 0$. Thus, we get that codimension 2 irreducible subspace of the fibre over $([0:1], [0:1] \times [0:1])$ corresponding to non-immersions at $P$. So, we get a codimension 2 irreducible subspace $V'$ of $V$ itself corresponding to $(g,P,x)$ with $g$ not being an immersion at $P$. Now, the first projection restricted to $V'$: $V' \to U_X$ has zero-dimensional fibres, so we must have the image as a codimension 1 irreducible subspace of $U_X$. 
 
 Now, to prove that the curve corresponding to a general point of the image of $V'$ has a simple cusp, we can work inside the subspace where $p_0 = p_1 = r_0 = r_1 = 0$, and prove that for general $p,q,r,s$ satisfying this condition, that the image is a simple cusp. We have that in a neighbourhood of $[0:1] \times [0:1]$, the map $g$ is
 $$t \mapsto \left( \frac{p(t,1)}{q(t,1)}, \frac{r(t,1)}{s(t,1)} \right) = \left( \frac{t^2(a+bt)}{q(t,1)}, \frac{t^2(c+dt)}{s(t,1)} \right)$$
 Now, $q(t,1),s(t,1)$ do not vanish at $0$, so we can choose inverses for them at the completion level, so that the map at the completion becomes $t \mapsto (a_0t^2 + a_1 t^3 + \cdots, b_0t^2 + b_1t^3 + \cdots)$ where $a_0 = a/q_0, b_0 = c/s_0, a_1 = b/q_0 - aq_1/q_0^2, b_1 = c/s_0 - cs_1/s_0^2$. Thus, for general $a,b,q_0,q_1,s_0,s_1$, we will have that $a_0b_1 \neq a_1b_0$. Hence, we may choose coordinates so that the map becomes $t \mapsto (t^2, t^3)$. This gives us a simple cusp.
 
 To prove that a general $g$ which is not an immersion at $[0:1]$ is an immersion at every other point, by a similar argument as before, it is enough to show that in the space of maps $g$ such that $g([0:1]) = ([0:1] \times [0:1])$ and $g([1:0]) = ([1:0] \times [1:0])$ with $g$ not being an immersion at $[0:1]$, the subspace of $g$ not being an immersion at $[1:0]$ has codimension 2. But this is simple to verify: the given conditions are that $p_0 = p_1 = r_0 = r_1 = 0$ and $q_3 = s_3 = 0$, and the condition for $g$ to not be an immersion at $[1:0]$ is that $q_2 = s_2 = 0$ so we clearly get a codimension 2 subspace.
 
 Finally, to show that the image of $g$ has a simple node at other singular points, (again by a similar argument as before) it is enough to show that for a general $g$ so that $g([0:1]) = ([0:1] \times [0:1]), g([1:0]) = ([1:0] \times [1:0]), g([1:1]) = ([1:0] \times [1:0])$ and $g$ is not an immersion at $[0:1]$, we have that there is no other point which is mapped to $([1:0] \times [1:0])$ under $g$ apart from $[1:0]$ and $[1:1]$ and the tangent directions at $([1:0] \times [1:0])$ coming from $[1:0]$ and $[1:1]$ are different. This also follows from the simple observation that the conditions are that $p(a,b) = a^2(p_2b + p_3a), r(a,b) = a^2(r_2b + r_3a)$ and that $q(a,b) = b(b-a)(q_4b + q_5a), s(a,b) = b(b-a)(s_4b + s_5a)$ for some $q_4,q_5,s_4,s_5$. There is no other point which is mapped to $([1:0] \times [1:0])$ under general such $g$ apart from $[1:0]$ and $[1:1]$ since $q_4s_5 \neq q_5s_4$ for general $q_4,q_5,s_4,s_5$ so $q(a,b),s(a,b)$ will not share three common roots. The tangent directions at $([1:0] \times [1:0])$ coming from $[1:0]$ and $[1:1]$ are $q_5r_3/s_5p_3$ and $(q_4+q_5)(r_2+r_3)/(s_4 + s_5)(p_2+p+3)$ respectively. These are not the same for general $p_2,p_3,r_2,r_3,q_4,q_5,s_4,s_5$, and so we are done.
 
 \bigskip 
 
 Let $Z'''_{H'}$ be the subspace of $Z'_{H'}$ of points not satisfying that they have 3 simple nodes and one simple cusp. So $Z'''_{H'}$ has codimension 2 in $Z_{H'}$. Let $Y_{H'}$ be the subspace of $W$ so that $S \cap H'$ is rational. Then $Y_{H'}$ has codimension $4$ in $W$. We have a map $\psi_{H'} : Y_{H'} \to Z_{H'}$ which sends $S \mapsto S \cap H'$. This map is smooth. Therefore, the inverse image of $Z'''_{H'} \cup Z''_{H'}$ under $\psi$ has codimension 2 inside $Y_{H'}$, and hence it has codimension $6$ inside $W$. Therefore, this subspace intersects $Y_{H}$ in subspace of codimension at least $2$, and hence it cannot map under $\psi_H$ onto any codimension 1 subspace of $Z_H$. Since the curve $C$ that we consider either corresponds to a general point of $Z_H$ or a general point of a codimension 1 subspace of $Z_H$, so we have the result.

\end{proof}

\begin{lemma} \label{noncurvilinear}
 Let $H$ be a hyperplane in $\mathbb P^4$, and let $C \subset H$ be a canonically embedded integral curve of genus $4$. Suppose that either $C$ represents a general point of $Z_4$ or a general point of $Z_{3,1}$. Let $H'$ be another hyperplane in $\mathbb P^4$ different from $H$. 
 
 Let $Q$ be the unique quadric surface in $H$ containing $C$. Then $Q$ is smooth. If the intersection $C \cap H'$ is non-curvilinear at a point $P$, then $H \cap H'$ is the tangent plane to $Q$ at $p$, so $Q \cap H'$ is a union of two distinct lines $L_1$ and $L_2$ intersecting at $P$. Also, the scheme $C \cap H'$ is a union of a length $4$ subscheme of $L_1 \cup L_2$ at $P$, together with one point on $L_1$ and another point on $L_2$.
\end{lemma}
\begin{proof}
$Q$ is smooth follows from the proof of Lemma \ref{nobadcurves} since we got that inside the the space $Z_{H'}$ of canonically embedded rational integral curves of genus $4$ in $H'$, we have the subspace $Z''_{H'}$ of curves having a cusp and being contained in a singular quadric has codimension $ \ge 2$ (this settles it for a general point of $Z_{3,1}$) and the subspace of curves contained in a singular quadric has codimension $1$ (for a general point of $Z_4$).

If $C \cap H'$ is non-curvilinear at $P$, then it is clear that $H \cap H'$ is the tangent plane to $Q$ at $P$ and therefore $Q \cap H'$ is a union of two lines $L_1 \cup L_2$ intersecting at $P$. Now, $C$ is the intersection of $Q$ with a cubic surface in $H$, so each of these two lines $L_1$ and $L_2$ will intersect $C$ in a length 3 scheme.

Let us identify $H \cap H'$ as $\mathbb P^2[x:y:z]$, $P = [0:0:1]$ and let $Q \cap H = V(xy)$. Now, consider a cubic $V(f)$ which has a singularity at $p$, so $f(x,y,1) = ax^2 + bxy + cy^2 + $ higher degree terms. Multiplying $f$ by $xy^n$ we get $cy^{n+2} + $ higher degree terms in $y = 0 \pmod{(xy,f(x,y,1))}$. Thus, if $a,c \neq 0$ then in the local ring $k \llbracket x,y \rrbracket/ (xy, f(x,y,1))$, we have that $x^n = 0, y^n = 0$ for all $n \ge 3$. Thus, 
$$k \llbracket x,y \rrbracket/ (xy, f(x,y,1)) \cong k \llbracket x,y \rrbracket/ (xy, ax^2 + cy^2)$$ which is a length 4 scheme. It is also clear that if $a,c \neq 0$ then the two lines of $Q \cap H$ intersect $C$ with multiplicity $2$ at $P$. Therefore, since these lines intersect $C$ in a length 3 scheme, so we must have one point each lying on $L_1$ and $L_2$ which is also in $C \cap H'$.

Thus, to prove that the component of $C \cap H'$ at $P$ is of length $4$, it suffices to fix $P,Q$ and show that the equation of the cubic at $p$ has non-zero $x^2$ and $y^2$ coefficients (if $Q \cap H'$ is given by $\{ xy = 0 \}$).

So we identify $Q = \mathbb P^1 \times \mathbb P^1$ and take $p = [0:1] \times [0:1]$. Then $C$ is the image of a map $$\phi: \mathbb P^1 \to \mathbb P^1 \times \mathbb P^1$$ which is given by $$[x:y] \mapsto [p(x,y): q(x,y)] \times [r(x,y): s(x,y)]$$ where $p,q,r,s$ are degree $3$ homogeneous polynomials in $x,y$ with $p,q$ having no common zero and $r,s$ having no common zero. We have three cases: \begin{enumerate}
    \item $\phi$ is general with the property that $P_1 = [0:1], P_2 = [1:0]$ maps to $P= [0:1] \times [0:1]$.
    \item $\phi$ is general with the property that $P_1 = [0:1], P_2 = [0:1]$ maps to $P= [0:1] \times [0:1]$ and $Q_3 = [1:1]$ maps to $P' = [1:0] \times [1:0]$ and $\phi$ is not an immersion at $Q_3$.
    \item $\phi$ is general with the property that $P_1 = [0:1]$ maps to $P= [0:1] \times [0:1]$ and $\phi$ is not an immersion at $Q$.
\end{enumerate}

Let $q(x,y) = q_0y^3 + q_1xy^2 + q_2x^2y + q_3 x^3 $ and $s(x,y) = s_0y^3 + s_1xy^2 + s_2x^2y + s_3 x^3 $. Also, let $\mathbb P^1 \times \mathbb P^1 = [u:w] \times [v:t]$.

\bigskip

In the first case, we have $p(x,y) = xy(p_0y + p_1x)$ and $r(x,y) = xy(r_0y + r_1x)$. Then in the open set at the point $P$ where $w \neq 0$ and $t \neq 0$, the image of $\phi$ is parametrized by $(u,v) = (p(x,y)/q(x,y), r(x,y)/s(x,y))$. If $f(u,v)$ is a polynomial vanishing at these parameters, then we need to show that the $u^2$ and $v^2$ terms of $f(u,v)$ are non-zero. 
To see this it is enough to observe that if $a_0u^2 + a_1 uv + a_2 v^2$ is the degree $2$ term of $f(u,v)$ then the only condition that $a_0, a_1, a_2$ satisfy is that $$x^3y^3 | a_0 p(x,y)^2 s(x,y)^2 + a_1 p(x,y) q(x,y) r(x,y) s(x,y) + a_2 r(x,y)^2 q(x,y)^2.$$ So, we have that
$$(p_0^2s_0^2) \cdot a_0 + (p_0q_0r_0s_0) \cdot a_1 + (r_0^2q_0^2) \cdot a_2 = 0$$ and
$$(p_1^2s_3^2) \cdot a_0 + (p_1q_3r_1s_3) \cdot a_1 + (r_1^2q_3^2) \cdot a_2 = 0$$
In order to prove that $a_2 \neq 0$ it suffices to observe that $(p_0^2s_0^2) \cdot (p_1q_3r_1s_3) \neq (p_1^2s_3^2) \cdot (p_0q_0r_0s_0)$ since $p_0, p_1, r_0, r_1, q_0, q_3, s_0, s_3$ are general. Similarly to prove that $a_0 \neq 0$ it is enough to check that $(r_0^2q_0^2) \cdot (p_1q_3r_1s_3) \neq (r_1^2q_3^2) \cdot (p_0q_0r_0s_0)$ which is again true since $p_0, p_1, r_0, r_1, q_0, q_3, s_0, s_3$ are general.

\bigskip

In the second case, we just have additionally that $q(x,1)$ and $r(x,1)$ has a double root at $[1:1]$. This happens if and only if $q_0 + q_1 + q_2 + q_3 =0$ and $3q_0 + 2q_1 + q_2 = 0$ so $q_1 = q_3 - 2q_0$ and $q_2 = q_0 - 2q_3$ and similar for $s_i$. So $p_0, p_1, r_0, r_1, q_0, q_3, s_0, s_3$ are general for this case also, and hence the same computation as the first case shows that we are through in this case also.

\bigskip

In the third case, we have $p(x,y) = x^2(p_0y + p_1x)$ and $r(x,y) = x^2(r_0y + r_1x)$. So in this case the only condition satisfied by $a_0,a_1,a_2$ is that $$x^6 | a_0 p(x,y)^2 s(x,y)^2 + a_1 p(x,y) q(x,y) r(x,y) s(x,y) + a_2 r(x,y)^2 q(x,y)^2.$$ So, we have that
$$(p_0^2s_0^2) \cdot a_0 + (p_0q_0r_0s_0) \cdot a_1 + (r_0^2q_0^2) \cdot a_2 = 0$$ and
$$p_0s_0(2p_1s_0 + 2p_0s_1) \cdot a_0 + (p_1r_0q_0s_0 + p_0r_1q_0s_0 + p_0r_0q_1s_0 + p_0r_0q_0s_1) \cdot a_1 + r_0q_0(2r_1q_0 + 2r_0q_1) \cdot a_2 = 0$$
In order to prove that $a_2 \neq 0$ it suffices to observe that $$(p_0^2s_0^2) \cdot (p_1r_0q_0s_0 + p_0r_1q_0s_0 + p_0r_0q_1s_0 + p_0r_0q_0s_1) - (p_0q_0r_0s_0) \cdot p_0s_0(2p_1s_0 + 2p_0s_1)$$
$$ = p_0^2s_0^2( - p_1r_0q_0s_0 + p_0r_1q_0s_0 + p_0r_0q_1s_0 - p_0r_0q_0s_1) \neq 0$$ since these coefficients are general.
Similarly, $a_0 \neq 0$ and so we are done.
\end{proof}

\subsubsection{Rational Curves containing a prescribed finite subscheme}

\begin{lemma}\label{uniqueconic}
Let $Y$ be a length $6$ subscheme of $\mathbb P^2$ which is obtained as the intersection of a (not-necessarily irreducible) conic curve and a cubic curve in $\mathbb P^2$. Then there is a unique conic curve in $\mathbb P^2$ containing $Y$.
\end{lemma}
\begin{proof}
Let $Q$ be a conic and $R$ be a cubic in $\mathbb P^2$ so that $Y = Q \cap R$. Thus, $Q$ and $R$ do not share a common component. Suppose, if possible, that there is another conic $Q'$ distinct from $Q$ which contains $Y$. If $Q$ and $Q'$ do not share a common irreducible component, then $Q \cap Q'$ is a length $4$ scheme by Bezout's Theorem, which is impossible because $Y \subseteq Q \cap Q'$ and $Y$ has length $6$. So $Q, Q'$ share a common component $L_1$ where $L_1$ is a line. Note that $Q$ and $Q'$ cannot both be non-reduced, since that would imply that they are both equal to the double line $L_1$. Therefore, if $Q$ is the double line $L_1$ then $Q'$ is reduced, so $Q = L_1 \cup L_1'$ where $L_1'$ is a line distinct from $L_1$. This implies that for any line $L_2$ passing through the intersection point of $L_1$ and $L_1'$, the conic $L_1 \cup L_2$ will contain $Y$. The general $L_2$ will not be a component of $R$, so we may assume that $Q = L_1 \cup L_2$ for such an $L_2$.

Therefore, $Q = L_1 \cup L_2$ is the union of two distinct lines $L_1$ and $L_2$ in $\mathbb P^2$. We may further assume that $L_1$ is the common component in $Q$ and $Q'$ so $L_2$ is not an irreducible component of $Q'$. Now, $L_2 \cap Y = L_2 \cap R$ is of length $3$ and $L_2 \cap Q'$ is of length $2$ by Bezout's Theorem. On the other hand, $L_2 \cap Y \subseteq L_2 \cap Q'$ but the former has length $3$ while the latter has length $2$, a contradiction.

\end{proof}

\begin{lemma} \label{curvelemma}
 Let $Y$ be a length $n$ curvilinear subscheme of a surface $X$ supported at a single point $p \in X$ with $n \ge 2$. Assume that $Y$ is contained in a smooth curve in $X$, and that $X$ is contained in a smooth threefold. Let $X'$ be the blowup of $X$ at $p$. Then there exists a curvilinear subcheme $Y'$ in $X'$ of length $n-1$ such that for any smooth curve $C$ (with blowup $C'$), $C$ contains $Y$ iff $C'$ contains $Y'$.
\end{lemma}

 \begin{proof}
  Let $Y$ be contained in a smooth curve $D$. Now, if $D'$ is the blowup of $D$ at $p$, then $D'$ maps isomorphically to $D$ since $D$ is smooth. Let $Y_D$ be the inverse image of $Y$ under this isomorphism. Now, let $Y'$ be the unique length $n-1$ subscheme of $Y_D$ (uniqueness and existence follows from the fact that $Y_D \cong Y \cong \text{Spec } k[x]/x^n$). We claim this is the $Y'$ that we want.
 
  \smallskip
  First we deal with the case when $X$ is smooth.
 
  Let $C$ be any smooth curve such that its blowup $C'$ contains $Y'$. Let $m$ be the intersection multiplicity of $C$ and $D$ at $p$. Since $Y$ is the unique length $n$ subscheme of $D$ supported at $p$ we only need to show that $m \ge n$. But now note that the intersection multiplicity of $C'$ and $D'$ at $p'$ will be $m-1$ (where $p'$ is the support of $Y'$). But both $C'$ and $D'$ contain $Y'$, so $m-1 \ge n-1$, and hence $m \ge n$.
 
  For the converse, if $C$ contains $Y$ then the intersection multiplicity of $C$ and $D$ at $p$ is at least $n$, and so the intersection multiplicity of $C'$ and $D'$ at $p'$ is at least $n-1$, and since $Y'$ is the unique length $n-1$ subscheme of $D'$ supported at $p'$, so we have that $C'$ contains $Y'$.
 
  \smallskip
  Now, let $X$ be general, and let $C$ be any smooth curve such that $C'$ passes through $Y'$. Now, the question is local at $p$, and $X$ is contained in a smooth threefold, so we may work in $\text{Spec } k \llbracket x,y,z \rrbracket$. With a change of coordinates, we may assume that $D$ is given by the ideal $(x,y)$, and $C$ is given by the ideal $(f,g)$, where $f,g \in (x,y,z)$ are linearly independent modulo $(x,y,z)^2$. (Also, $f,g \in (x,y,z^n)$ for $n \ge 2$)
 
  Now, let $f = f_1 + z^kf_2$ and $g = g_1 + z^kg_2$, where $f_1, g_1 \in (x,y)$ and $f_2$ is a unit in $k \llbracket x,y,z \rrbracket$. Then consider $T = g_2f - f_2g$. Then $T \in (f,g) \cap (x,y)$. Also, $T$ will have a non-zero deg $1$ term due to $f,g$ having linearly independent deg $1$ terms. 
 
  Thus, we have proved that locally, $C,D$ are both contained in a smooth surface given by $T = 0$. Hence, due to uniqueness of $Y' \subset D'$, we can reduce to the smooth surface case and proceed as before. This completes the proof.
 \end{proof}

\begin{lemma} \label{calc3}
Following the notation of Lemma \ref{mainlemma}, let $C = \mathbb P^1$, and $T_{\phi}(F)$ denote the tangent space of $F = F(C,X,Y_1, \cdots, Y_m)$ at the point $(\phi,D_1, \cdots, D_m) \in F(k)$.
Let $X$ be an irreducible quadric surface in $\mathbb P^3$, $\sum n_i = 6$ and $\phi^* \mathcal O_X(1) = \mathcal O(6)$. Suppose that $\phi$ is an immersion. If $X$ is singular with singular point P, suppose further that length$(Y_i) \ge 2$ if $Y_i$ is supported at $P$ and that not more than four of the $Y_i$ are supported at $P$. Then $(\dim F)_\phi \le 8 + d$ where $(\dim F)_\phi$ is the dimension of $F$ in a neighbourhood of the point $(\phi,D_1, \cdots, D_m)$ and $d$ is the number of $Y_i$ which are supported at the singular point of $X$ (so $d \le 4$).

\end{lemma}
\begin{proof}
$X$ is a quadric surface in $\mathbb P^3$, $\phi : \mathbb P^1 \to X$ satisfies $\phi^* \mathcal O_X(1) = \mathcal O(6)$.

First, let us assume that $\phi(\mathbb P^1)$ is not supported at the singular point of $X$.
We have 
 $$0 \to T_{\mathbb P^1} \to \phi^* T_X \to N_\phi \to 0 $$
 
 Let $U = X_{sm}$ be the smooth locus of $X$. So $U =X$ or $U = X$ minus one point. Note that by the adjunction formula $\omega_X|_U = \mathcal O_X(-2)|_U$, so $\det (T_X)|_U = \mathcal O(2)$. Therefore, $\det \phi^* T_X = \mathcal O(12)$, since $\phi^* \mathcal O_X(1) = \mathcal O(6)$ and $\phi(\mathbb P^1) \subset U$. So, looking at determinants in the short exact sequence, we have that $N_\phi = \mathcal O(10)$. 
 
 
 Now, $\sum n_i = 6$ so if $\mathcal I_D = \mathscr I_{D_1} \cdots \mathscr I_{D_m}$, then $\mathcal I_D = \mathcal O(-6)$. Thus, $\mathcal I_D N_\phi = \mathcal O(4)$. This implies that $H^0(\mathbb P^1, \mathcal I_D N_\phi)$ is a codimension-$6$ space inside the $11$ dimensional space $H^0(\mathbb P^1, N_\phi)$.
 
 
 Now, $H^1(\mathbb P^1, T_{\mathbb P^1}) = 0$, so we have 
 
 $$0 \to H^0(\mathbb P^1, T_{\mathbb P^1}) \to H^0(\mathbb P^1, \phi^* T_X) \to H^0(\mathbb P^1, N_\phi) \to 0 $$
 
 $H^0(\mathbb P^1, \mathcal I_D N_\phi)$ has codimension-$6$ in $H^0(\mathbb P^1, N_\phi)$, so, the dimension of $T_\phi(F)$ is $11 + 3 -6 = \boxed{8}$. Thus, we have that $(\dim F)_\phi \le 8$.
 
 Now, assume that $X$ is singular and $\phi(\mathbb P^1)$ is supported at the singular point of $X$. Consider the blowup $f:X' \to X$ of $X$ at the singular point $P$ of $X$. Then $X'$ is smooth and any map $\phi: \mathbb P^1 \to X$ factors through a map $g:\mathbb P^1 \to X'$.
 
 \begin{center}
 \begin{tikzcd}
 \mathbb P^1 \arrow[dd, "\phi"'] \arrow[rr, "g"] &  & X' \arrow[lldd, "f"] \\
  &  &  \\
 X &  & 
 \end{tikzcd}
 \end{center}
 
 Now, $X'$ is a $\mathbb P^1$ bundle over $\mathbb P^1$, it is $\mathbb P_{\mathbb P^1}(\mathcal O \oplus \mathcal O(-2))$ over $\mathbb P^1$. We have $\text{Pic}(X') = \mathbb Z A \oplus \mathbb Z B$, where $A$ is a fibre of the map $X' \xrightarrow{p} \mathbb P^1$, and $B$ is a section of $p$ corresponding to $\mathcal O_{X'}(1)$.
 
 Also, we have that $f^* \mathcal O_X(1) = \mathcal O_{X'}(2A + B)$. This implies that $g^* \mathcal O_{X'}(2A + B) = \mathcal O_{\mathbb P^1}(6)$.
 Now, Euler exact sequence gives us
 
 $$0 \to \Omega_{X'/\mathbb P^1}^1 \to p^*(\mathcal O \oplus \mathcal O(-2)) \otimes \mathcal O_{X'}(-1) \to \mathcal O_{\mathbb P^1} \to 0$$
 Now, $p^*(\mathcal O \oplus \mathcal O(-2)) \otimes \mathcal O_{X'}(-1) = (\mathcal O \oplus \mathcal O(-2A)) \otimes \mathcal O(-B) = \mathcal O(-B) \otimes \mathcal O(-2A - B)$.
 
 Taking determinant, we have $\Omega_{X'/\mathbb P^1}^1 = \mathcal O(-2A -2B)$.
 
 We also have 
 $$0 \to p^* \Omega_{\mathbb P^1}^1 \to \Omega_{X'}^1 \to \Omega_{X'/\mathbb P^1}^1 \to 0$$
 where exactness on the left is since any map from a locally free sheaf to a torsion free sheaf is injective.
 
 Taking determinant, we have that $\det \Omega_{X'}^1 = \mathcal O(-4A - 2B)$, and hence $\det g^* \Omega_{X'}^1 = g^* \mathcal O(-4A - 2B) = O(-12)$, since $g^* \mathcal O(2A + B) = \mathcal O(6)$ as noted earlier.
 
 Thus, we have that $\det g^* T_{X'} = \mathcal O(12)$, and hence $N_g = \mathcal O(10)$. Now for each $i=1, \cdots, m$, define $Y_i'$ to be the inverse image of $Y_i$ under $f$ if $Y_i$ is not supported at $P$, and $Y_i'$ to be the unique subscheme of $X'$ of length 1 less than the length of $Y_i$ given by Lemma \ref{curvelemma} if $Y_i$ is supported at $P$. Now note that there is a natural map $F(\mathbb P^1, X', Y_1', \cdots, Y_m') \to F(\mathbb P^1, X, Y_1, \cdots, Y_m)$, which is dominant so $$(\dim F(\mathbb P^1, X, Y_1, \cdots, Y_m))_\phi \le (\dim F(\mathbb P^1, X', Y_1', \cdots, Y_m'))_g.$$ Now, $\mathcal I_D = \mathscr I_{D_1'} \cdots \mathscr I_{D_m'} = \mathcal O(-6-d)$ here, so $\mathcal I_D N_g = \mathcal O(4-d)$ where $D_i'$ is the unique subdivisor of $D_i$ mapping isomorphically onto $Y_i'$. Then the same calculation as in the previous case shows that the dimension of the tangent space of $F(\mathbb P^1, X', Y_1', \cdots, Y_m')$ comes out to be $\boxed{8+d}$.
 
\end{proof}

\begin{lemma} \label{cuspcalc4}
Let $C$ be the unique (up to isomorphism) rational curve of genus $1$ with 1 simple cusp, and let $X \cong \mathbb P^1 \times \mathbb P^1$ be a smooth quadric surface in $\mathbb P^3$. Let $Y_1, \cdots, Y_m$ be finite length curvilinear subschemes of $X$, with their lengths totalling $6$.

Consider the functor $F = F(C,X,Y_1, \cdots, Y_m)$, and let $(g,D_1, \cdots, D_m) \in F(k)$ be a point. Suppose that $g: C \to X$ is an immersion, and the image $g(C)$ lies in $\mathcal O(3) \boxtimes \mathcal O(3)$. We also assume that $g$ is an embedding at the cusp of $C$, that $g(C)$ only has 3 simple nodes and 1 simple cusp and that at most one of the $D_i$ are supported at the cusp of $C$.
Then if $(\dim F)_g$ denotes the dimension of $F$ at the point $(g,D_1, \cdots, D_m)$, then we have
\begin{itemize}
    \item $(\dim F)_g \le 6$ if none of the $D_i$ are not supported at the cusp,
    \item $(\dim F)_g \le 7$ if length $D_i \le 2$ if $D_i$ is supported at the cusp,
    \item $(\dim F)_g \le 8$ otherwise.
\end{itemize}
\end{lemma}

\begin{proof}
We have the following exact sequence on the tangent spaces:
$$0 \to T_C \to g^* T_X \xrightarrow{\phi} N_g$$
where $N_g$ is the dual of $\mathscr I_C/\mathscr I_C^2$ in a neighbourhood of the cusp and is the usual normal bundle elsewhere. Then this gives us a map on global sections:
$$0 \to H^0(C,T_C) \to H^0(C, g^* T_X) \xrightarrow{H^0(\phi)} H^0(C, N_g)$$
We want the inverse image of $H^0(C, \mathscr I_D N_g)$ under $H^0(\phi)$ where $\mathscr I_D = \mathscr I_{D_1} \cdots \mathscr I_{D_m}$. 

We first prove that the degree of $N_g$ is 12. There is a canonical map from $N_g$ to the pullback of the normal bundle of $g(C)$ in $X$ which is an isomorphism outside the $6$ points lying above the $3$ nodes. The latter has degree $18$ so $\deg N_g = 12$ will follow by showing that the cokernel of the map referred to above is a finite length sheaf supported at these $6$ points and computing the length. To compute this, all we need to do is work out what all these sheaves and maps are locally i.e. when $g(C)$ is the curve $xy = 0$ in the plane and $C$ is its normalisation (so a union of two lines). This was done in the proof of Lemma \ref{cuspcalc3}.

So, now we have that $N_g$ has degree $= 18 - 2 \times 3 = 12$. This implies that $\mathscr I_D N_g$ has degree $6$. Thus, 
$$h^0(C,N_g) = 12, h^0(C,\mathscr I_D N_g) = 6$$

\bigskip

Let $Q_g$ be the image of $g^* T_X$ in $N_g$ under $\phi$. Then we know that 
$$h^0(C,T_C) = 2, h^1(C,T_C) = 0$$
so we have
$$0 \to H^0(C,T_C) \to H^0(C, g^* T_X) \to H^0(C, Q_g) \to 0$$
Therefore, we need to find the dimension of $H^0(C, \mathscr I_D N_g \cap Q_g)$.

 Now,$Q_g = N_g$ away from the cusp, since $\phi$ is an immersion away from the cusp and $C$ is smooth away from the cusp. Now, in a neighbourhood of the cusp, we have
\begin{align*}
    \mathscr I_C/\mathscr I_C^2 &\xrightarrow{d} \Omega_X|_C \\
    y^2 - x^3 & \mapsto d(y^2 - x^3) = 2ydy - 3x^2dx
\end{align*}
So, locally around the cusp, we have $Q_g$ is $(3x^2, 2y) \cdot N_g = (x^2,y) \cdot N_g$. Thus, $N_g/Q_g$ will be length 2 supported at the cusp. Now, we proceed as in the proof of Lemma \ref{cuspcalc3}. Using $\mathcal L$ which we define to be the line bundle given as the subsheaf of $N_g$ defined by the section $y$ around the cusp and equal to $N_g$ away from the cusp, and proceeding as in the proof of Lemma \ref{cuspcalc3}, we get that
$$0 \to H^0(C,Q_g) \to H^0(C, N_g) \to H^0(C, N_g/ Q_g) \to 0$$
is exact. Thus, $h^0(C,Q_g) = 10$.
Now, $\mathscr I_D \mathcal L \subset \mathscr I_D N_g \cap Q_g$, and $\mathscr I_D \mathcal L$ is a degree $3$ line bundle so $h^1(\mathscr I_D \mathcal L) = 0$, so the same argument as before implies that 
$$0 \to H^0(C,\mathscr I_D N_g \cap Q_g) \to H^0(C, \mathscr I_DN_g) \to H^0(C, \mathscr I_DN_g/ \mathscr I_DN_g \cap Q_g) \to 0$$
Thus, if $D$ is not supported at the cusp, then $h^0(C, \mathscr I_D N_g \cap Q_g) = h^0(C, \mathscr I_D N_g) - 2 = 4$ and we will have the dimension of the tangent space as $4+2 = 6$.

If $D$ is supported at the cusp, then locally we have that $D$ is given by $f \in R = k[[x,y]]/(y^2 - x^3)$. We recall from the proof of Lemma \ref{cuspcalc3} that if $f \in (x^2,y)$ then $R/(f)$ has length $ \ge 3$.

So, if $\mathscr I_D N_g \cap Q_g = \mathscr I_D N_g$ then $D$ has length $3$ at the cusp. In this case, the dimension of the tangent space is $6+2 = 8$.

Finally, if $f \not \in (x^2,y)$ then $h^0(C,\mathscr I_D N_g \cap Q_g) = 5$ and we get the dimension of the tangent space is $5+2 = 7$.
\end{proof}

\begin{lemma} \label{functorcalc4}
 For any finite scheme $Y \subset \mathbb P^{3}$, and any (not-necessarily smooth) quadric surface $X \subset \mathbb P^3$ containing $Y$, let $R_{Y,X}$ be the space of canonically embedded genus $4$ rational integral curves $C'$ in $\mathbb P^{3}$ which contain $Y$ and are contained in $X$. Let $R^{(1)}_{Y,X}$ be the subspace of $R_{Y,X}$ corresponding to points representing curves which do not have a cusp, and let $R^{(2)}_{Y,X}$ be the points representing curves which have a cusp.

Suppose $Y$ is a length $6$ scheme supported on the conic $X \cap H$ for a hyperplane $H$ of $\mathbb P^3$. If $Y$ is curvilinear, we have:
\begin{enumerate}
    \item $\dim R^{(1)}_{Y,X} \le 5+d$ where $d = 1$ if $X$ is singular at a point of $Y$ and $0$ otherwise.
    \item For a curve $C'$ representing a point of $R^{(2)}_{Y,X}$ which has exactly 1 cusp and 3 nodes as singularities, and for smooth $X$,
    \begin{enumerate}
        \item If $Y$ is not supported at the cusp of $C'$, then $(\dim R^{(2)}_{Y,X})_{C'} \le 4$.
        \item If every component of $Y$ has length $\le 2$, then $(\dim R^{(2)}_{Y,X})_{C'} \le 5$.
        \item If some component of $Y$ has length $\ge 3$, then $(\dim R^{(2)}_{Y,X})_{C'} \le 6$.
    \end{enumerate}
\end{enumerate}

Now, suppose that $X \cap H$ is the union of two lines $L_1$ and $L_2$ intersecting at the point $p_0$. Let $Y$ be supported on $X \cap H$, with $Y = Y_1 \cup p_1\cup p_2$ where $Y_1$ is non-curvilinear of length $4$ supported at a point $p_0$ and $p_1,p_2$ are two other distinct points. Also, suppose that $Y \cap L_1$ and $Y \cap L_2$ are length 2 schemes. Then we have:
\begin{enumerate}
    \item $\dim R^{(1)}_{Y,X} \le 7$.
    \item For a curve $C'$ representing a point of $R^{(2)}_{Y,X}$ which has exactly 1 cusp and 3 nodes as singularities, and for smooth $X$, $(\dim R^{(2)}_{Y,X})_{C'} \le 6$.
\end{enumerate}
\end{lemma}
\begin{proof}
Note that for any curve $C'$ corresponding to a point of $R_{Y,X}$, we have $Y = C' \cap H$ since $Y$ is a length $6$ subscheme of the length $6$ scheme $C' \cap H$. 

We first analyze the dimensions of $R^{(1)}_{Y,X}$ and $R^{(2)}_{Y,X}$ when $Y$ is curvilinear: Let $Y_1, \cdots, Y_m$ be the connected components of $C \cap H'$ with $\deg Y_i = n_i$. $Y_i$'s are all curvilinear. Let $C'$ be a curve corresponding to a point of $R_{Y,X}$.
 \begin{itemize}
     \item If $C'$ does not have a cusp, let $\phi : \mathbb P^1 \to C' \hookrightarrow X$ be a normalization map of $C'$. Then $\phi$ is an immersion. Let $D_i$ be length $n_i$ subschemes of $\mathbb P^1$ mapping to $Y_i$. Thus, $(\phi, D_1, \cdots, D_m)$ is a point of the space $F^{(1)} = F(\mathbb P^1, X, Y_1, \cdots, Y_m)$. We bound the dimension of $F^{(1)}$ at the point $(\phi, D_1,\cdots, D_m) \in F^{(1)}(k)$ by looking at the tangent space of $F$ at this point. By Lemma \ref{calc3}, we have this bound on the dimension as $8 +d$ where $d = 1$ if $X$ is singular at a point of $Y$ and $0$ otherwise. So, if $F^{(1)}_0$ is the component of $F^{(1)}$ where $\phi$ is a degree $6$ map, then we have a dominant map $\Psi_1: F^{(1)}_0 \to R^{(1)}_{Y,X}$ defined as $(\phi,D_1, \cdots, D_m) \mapsto \phi(\mathbb P^1)$.
 The automorphism group of $\mathbb P^1$ is 3-dimensional, and fixes $\phi(\mathbb P^1)$ so every fibre of $\Psi_1$ is at least 3-dimensional, and so we have a upper bound of $5 +d$ on the dimension of $R^{(1)}_{Y,X}$ at $C'$ where $d = 1$ if $X$ is singular at a point of $Y$ and $0$ otherwise.
 
 \item If $C'$ has 1 cusp and 3 nodes as the only singularities, let $\mathscr C$ be the unique (up to isomorphism) rational curve of genus 1 with 1 cusp, and let $\phi: \mathscr C \to C'$ be the blow up of $C'$ at the nodes. Let $D_i$ be length $n_i$ subschemes of $\mathscr C$ mapping to $Y_i$. Thus, $(\phi, D_1, \cdots, D_m)$ is a point of the space $F^{(2)} = F(\mathscr C, X, Y_1, \cdots, Y_m)$. So, if $F^{(2)}_0$ is the component of $F^{(2)}$ where $\phi$ is a degree $4$ map, then we have a dominant map $\Psi_2: F_0 \to R^{(2)}_{Y,X}$ defined as $(\phi,D_1, \cdots, D_m) \mapsto \phi(\mathscr C)$. Now, by Lemma \ref{cuspcalc4}, the dimension of $F^{(2)}$ at $(\phi,D_1, \cdots, D_m)$ is bounded by $6$ if $D_i$ are not supported at the cusp, bounded by $7$ if length $D_i \le 2$ if $D_i$ is supported at the cusp and bounded by $8$ regardless.
  The automorphism group of $\mathscr C$ is 2-dimensional, and fixes $\phi(\mathscr C)$ so every fibre of $\Psi_2$ is at least 2-dimensional, and so we have a upper bound of $4,5,$ or $6$ on the dimension of $R^{(2)}_{Y,X}$ at $C'$ depending on the above conditions.
 
 \end{itemize}

Suppose $Y$ is not curvilinear, and let $Y = Y_1 \cup p_1\cup p_2$ where $Y_1$ is a non-curvilinear length $4$ scheme supported at a singular point $p_0$ of $C$ and $p_1,p_2$ are some two other points of $C$. In this case $X \cap H$ is the union of two distinct lines $L_1$ and $L_2$ in $H$ meeting at $p_0$ so that $L_1$ and $L_2$ both intersect $Y$ with multiplicity 2 at $p_0$. Again, let $C'$ be a curve corresponding to a point of $R_{Y,X}$. Note that $p_1, p_2$ are smooth points of $C'$ (since $Y = C' \cap H$). If $C'$ has a simple cusp at $p_0$ and 3 nodes as the only other singularities, then we know that the dimension of $R^{(2)}_{Y,X}$ at $C'$ is bounded by $6$ by the same argument as before. So assume that $C'$ does not have a cusp at $p_0$.

Let us investigate the case when $C'$ does not have a cusp and $X$ is smooth. Let $\phi: \mathbb P^1 \to C'$ be a normalization of $C'$. So $\phi$ is an immersion. Then the pullback of $C' \cap H = Y$ will be a degree $6$ divisor $D$ of $\mathbb P^1$. If $p_1', p_2'$ are the points of $\mathbb P^1$  mapping to $p_1,p_2$ respectively, then $D = p_1'+p_2' +D'$ where $D' = D_1 + D_2 + \cdots + D_m$ is supported at the $m$ points of $\mathbb P^1$ which map to $p_0$. Consider the scheme $\tilde F(\mathbb P^1, X, Y, n_1, \cdots, n_m,1,1)$ which we will simply call $\tilde F(\underline n)$, and take the component $\tilde F_0$ of $\tilde F$ where $\phi$ is a degree $6$ map and the image of $\phi$ has genus $4$. Then the union of $\tilde F_0(\underline n)$ over all partitions $(\underline n)$ of $4$ surjects onto $R_{Y,X}$. To bound the dimension of $\tilde F_0(\underline n)$ we note that it maps to $E(\underline n) = \textup Em_{p_0}(T_{n_1}, C) \times \cdots \times \textup Em_{p_0}(T_{n_m}, C)$ with fiber over the embeddings $(Y_1', \cdots, Y_m')$ as $F(\mathbb P^1,X,Y, Y_1', \cdots, Y_m', p_1, p_2)$. Recall from Lemma \ref{calc3}, that $\dim F(\mathbb P^1,X,Y, Y_1', \cdots, Y_m', p_1, p_2)_\phi \le 8$.  We have the cases:
\begin{itemize}
    \item (1,1,1,1): Due to Lemma \ref{genusbound}, such a $C'$ must necessarily have genus $\ge 6$, so this case does not arise.
    \item (2,1,1): By Lemma \ref{embedcalc}, $\dim E(\underline n) = 1$. So we have the dimension of $\tilde F_0(\underline n) \le 9$.
    \item (3,1): By Lemma \ref{embedcalc}, $\dim E(\underline n) = 1$. So we have the dimension of $\tilde F_0(\underline n) \le 9$.
    \item (2,2): By Lemma \ref{embedcalc}, $\dim E(\underline n) = 2$. So we have the dimension of $\tilde F_0(\underline n) \le 10$.
\end{itemize}
So, we get that $(\dim R^{(1)}_{Y,X})_{C'} \le 7$ in the case when $X$ is smooth.
The case when $C'$ has a cusp (not at $p_0$) is handled exactly the same way as above, with $\mathscr C$ instead of $\mathbb P^1$, and using the calculations in Lemma \ref{cuspcalc4} (note that none of the $D_i$ will be supported at the cusp), so we will get $(\dim R^{(2)}_{Y,X})_{C'} \le 6$ in this case. Also, the case when $C'$ does not have a cusp and $X$ is singular at a point different from $p_0$ is also dealt with the same way as above, and we get $(\dim R^{(1)}_{Y,X})_{C'} \le 7$ in this case.

So it remains to deal with the case of $C'$ not having a cusp and $X$ being singular at $p_0$. Let $f: X' \to X$ be the blow-up of $X$ at $p_0$. Then we will have two points $p^{(1)}_0$ and $p^{(2)}_0$ in the intersection of the exceptional of $f$ with the strict transform of $C'$ and these will be determined by $Y$ since these will correspond to the two tangent directions $L_1$ and $L_2$ at $p_0$. The upshot of this is that we only need to consider $F(\mathbb P^1, X', p^{(1)}_0, p^{(2)}_0, p_1', p_2')$ which has dimension $\le 10$ as we saw in the proof of Lemma \ref{calc3}. So, we get $(\dim R^{(1)}_{Y,X})_{C'} \le 7$ in this case.
\end{proof}

\begin{lemma} \label{genpoints4}
Let $X$ be a smooth quadric surface in $\mathbb P^3$, and let $H'$ be a plane in $\mathbb P^3$ which intersects $X$ in the conic $D = X \cap H'$. Let $R_X$ be the space of integral rational curves of genus 4 canonically embedded in $\mathbb P^3$ which lie on $X$, and let $R'_X$ be the codimension 1 irreducible subspace of $R_X$ consisting of curves in $R_X$ having a cusp. Then 
\begin{enumerate}
    \item if $C$ is a curve corresponding to a general point of $R_X$, then $C \cap H'$ consists of $6$ general points on $D$.
    \item if $C$ is a curve corresponding to a general point of $R'_X$, then $C \cap H'$ consists of $6$ general points on $D$.
\end{enumerate}
\end{lemma}
\begin{proof}
From the proof of Lemma \ref{nobadcurves}, we have that $R_X$ is irreducible of dimension $11$, and $R'_X$ is irreducible of dimension $10$, since $R_X$ (resp. $R'_X$) is the image of $U_X$ (resp. the codimension 1 irreducible subspace of $U_X$ of non-immersions) under $\beta$.

\textit{Observation: There exists a point of $R_X$ (and $R'_X$) so that the curve $C$ that it represents intersects $H'$ in $6$ distinct points.} We note that if $C$ is any curve corresponding to a point of $R_X$ (or $R'_X$) then $C \cap H$ will be the union of $6$ distinct points for a general hyperplane $H$. Now, there will be a $\sigma \in \textup{Aut}(\mathbb P^3)$ so that $\sigma(H) = H'$ and $\sigma(X) = X$. So, $\sigma(C)$ will represent a point of $R_X$ (or $R'_X$, respectively) which intersects $H'$ in $6$ distinct points. 
\begin{enumerate}
    \item  Consider the rational map $\alpha: R_X \dashedrightarrow \textup{Sym} ^6 (D)$ where $C \mapsto C \cap H'$ (this is a rational map due to the above observation). Take any $Y = \{p_1, \cdots, p_6 \}$ corresponding to a point on $ \textup{Sym} ^6 (D)$, then the the fiber of $\alpha$ over $Y$ is $R_{Y,X}$. Note that $Y$ consists of $6$ distinct points and $X$ is smooth, therefore $R_{Y,X}$ has dimension $\le 5$ by Lemma \ref{functorcalc4}. Therefore, $\alpha$ has to be dominant, since the image cannot be less than $11 - 5 = 6$ dimensional.
    
    \item Consider the rational map $\alpha: R'_X \dashedrightarrow \textup{Sym} ^6 (D)$ where $C \mapsto C \cap H'$(this is a rational map due to the above observation). Take any $Y = \{p_1, \cdots, p_6 \}$ corresponding to a point on $ \textup{Sym} ^6 (D)$, then the the fiber of $\alpha$ over $Y$ is $R^{(2)}_{Y,X}$. Note that $Y$ consists of $6$ distinct points, so cannot be supported at the cusp of $C$ for any curve $C$ in the fiber over $Y$. Thus $R^{(2)}_{Y,X}$ has dimension $\le 4$ by Lemma \ref{functorcalc4}. Therefore, $\alpha$ has to be dominant, since the image cannot be less than $11 - 5 = 6$ dimensional.
\end{enumerate}
\end{proof}

\subsection{Proof of the main result}
\begin{itemize}
    \item Let $W$ be the Hilbert scheme of complete intersections in $\mathbb P^4$ of a quadric hypersurface with a cubic hypersurface. So $W = W_4$ where $W_4$ is as in the introduction.
    \item Let $J = \{ (S,H)  \ | \ S \cap H$ is rational integral $ \} \subseteq W \times (\mathbb P^4)^{\vee}$. So $J \subset J_4$.
    \item Let $\pi : J \to W,  \ \eta : J \to (\mathbb P^4)^{\vee}$ be the projection maps. Note that by Lemma \ref{integral4}, the complement of $J$ in $J_4$ will not dominate $W$, so the monodromy group of $\pi$ will be equal to $\Pi_4$.

\end{itemize}

\begin{prop} \label{trans4}
The monodromy group $\Pi_4$ of $\pi:J \to W$ is transitive. 
\end{prop}
\begin{proof}
As with the proof of Proposition \ref{trans3}, it suffices to show that $J$ has only one irreducible component of maximum dimension.

First, we note that $\eta: J \to (\mathbb P^4)^\vee$ makes $J$ into an \'{e}tale locally trivial fibre bundle over $(\mathbb P^4)^\vee$, thus $J$ also has only one irreducible component of maximum dimension. The fiber of $\eta$ above a fixed hyperplane $H$ is $Y_H = \{ S \ | \ S \cap H \textup{ is rational integral} \}$. Note that for $S \in Y_H$, $S \cap H =C$ is an integral curve in $H$ which is the intersection of a quadric and cubic, so by the adjunction formula, $\omega_C = \mathcal O_C(1)$ and hence $C$ is canonically embedded in $H$. Thus, we may consider the map $\psi:Y_H \to Z$ from $Y_H$ to the closed subscheme $Z$ of rational curves in the Hilbert scheme of curves in $H \cong \mathbb P^3$ which are the complete intersection of a quadric surface and a cubic surface. Then $\psi$ is surjective due to Lemma \ref{prop3}. Note that due to the adjunction formula, these curves corresponding to points of $Z$ are all canonically embedded. Conversely, every canonically embedded rational curve will be a point of $Z$.

On the other hand, $Z$ contains $Z_4$ as a dense open subset corresponding to points which represent curves having only nodes and simple cusps as singularities and whose normalization morphism does not admit a non-trivial automorphism. Now, $Z_4$ is irreducible by Proposition \ref{locirr}, hence $Z$ is irreducible of dimension $20$.

Let $C$ be a genus $4$ rational integral curve canonically embedded in $H$. Let $\mathbb P^4$ be parametrized by coordinates $[x_0:x_1:x_2:x_3:x_4]$, $H$ be given by $\{ x_4=0 \}$. Write $C = Q_C \cap R_C$ as the intersection of a quadric and cubic in $H$. Let the equations of $Q_C,R_C$ be $f(x_0,\cdots,x_3) = 0$ and $g(x_0,\cdots,x_3) = 0$ respectively. 

If $S$ belongs to the fiber of $\psi$ over $C$, then $S = Q \cap R$ where $Q$ is a quadric hypersurface and $R$ is a cubic hypersurface. Now, if $Q \cap H = Q_0$ and $R \cap H = R_0$, then $Q_0 \cap R_0 = C$. Therefore, $Q_0$ must be the unique quadric in $H$ which contains $C$, thus $Q_0 = Q_C$. Thus, if the equation of $Q,R$ are $\tilde f(x_0,\cdots,x_4) = 0$ and $\tilde g(x_0,\cdots,x_4) = 0$ respectively, then we may vary the equation of $Q$ in the $\mathbb A^5$ of $\tilde f$ such that $\tilde f(x_0,x_1,x_2,x_3,0) = cf$ for some non-zero scalar $c$. Also, $R_0$ is a cubic containing $C$, so $\tilde g(x_0,x_1,x_2,x_3,0) = g$ up to a linear multiple of $f$, but on the other hand $Q \cap R$ gives the same surface when we change $g$ by a linear multiple of $f$ so we may take all $\tilde g$ so that $\tilde g(x_0,x_1,x_2,x_3,0) = g$, and identify different $\tilde g$ if they differ by a scalar times $x_4f$. Thus, we get that the fiber of $\psi$ over $C$ is isomorphic to the space $(\mathbb A^5 \times \mathbb A^{15})/ \sim$ corresponding to $(\tilde f, \tilde g)$ so that $\tilde f(x_0,x_1,x_2,x_3,0) = cf$ for some non-zero scalar $c$, $\tilde g(x_0,x_1,x_2,x_3,0) = g$ and we identify $(\tilde f, \tilde g) \sim (\tilde f, \tilde g + c'x_4f)$. Thus, we get the fiber of $\psi$ over $C$ to be isomorphic to $\mathbb A^{19}$.

We claim that $\psi$ is flat: To see this we proceed on similar lines as in the proof of Proposition \ref{ST3}. We denote the Hilbert scheme of complete intersection of type (2,3) in $\mathbb P^n$ by $\mathrm{Hilb}_{(2,3)}(\mathbb P^n)$, then we have that $\psi$ is a base change to $Z$ of the morphism from $\mathrm{Hilb}_{(2,3)}(\mathbb P^4)$ (minus the closed subset consisting of surfaces not intersecting $H$ properly) to $\mathrm{Hilb}_{(2,3)}(\mathbb P^3)$, the morphism given by intersecting with $H$. So, it is enough to note that this morphism between Hilbert schemes is flat. To see this we note that both the schemes are smooth: both the Hilbert schemes are a projective bundle over a projective space, and that the map has constant fibre dimension: the fibre of this space will also be isomorphic to $\mathbb A^{19}$ by exactly the same calculation as in the previous paragraph where we calculated the fibres of $\psi$. Thus, $\psi$ is flat. 

Thus, $\psi$ is flat and has smooth fibres so $\psi$ is smooth.
Now we know that $Z$ is irreducible of dimension 20, so combining with $\psi$ is smooth we get that $Y_H$ is irreducible of dimension 39. Finally, since $J$ is an \'{e}tale locally trivial fibre bundle over $(\mathbb P^4)^\vee$ with fibres $Y_H$, thus $J$ is also irreducible of dimension 43.
\end{proof}

Let $H'$ be a hyperplane in $\mathbb P^4$. A curve $C' \subset H'$ \textit{satisfies $(*)$} if satisfies the following:
\begin{enumerate}
    \item $C'$ is an integral rational curve of genus $4$ and is canonically embedded in $H'$.
    \item If $C'$ has a cusp then it has exactly one simple cusp and the unique quadric surface $Q'$ containing $C'$ is smooth.
\end{enumerate} 

Fix a hyperplane $H$ in $\mathbb P^4$, and a genus 4 canonically embedded rational integral curve $C \subset H$. Suppose that either $C$ is represented by a general point of $Z_4$ or a general point of $Z_{3,1}$.

Let $W' = \{ S \ | \ S \cap H = C \} \subseteq W$. Note that $W'$ is irreducible of dimension $19$, by the analysis in the proof of Proposition \ref{trans4}. Also, it always contains a smooth surface due to Lemma \ref{prop3}. 

We define $$J' = \{ (S,H')  \ | \ S\cap H = C, \ S \cap H' \textup{ is rational integral}, \  H'\neq H\} \subseteq W' \times ((\mathbb P^4)^{\vee} - H).$$ Let $\pi ' : J' \to W',  \ \eta ' : J' \to ((\mathbb P^4)^{\vee} - H)$ be the projection maps. So, by Lemma \ref{nobadcurves}, there is an open set $W'' \subseteq W'$ so that for any $(S,H') \in \pi '^{-1}(W'')$, $S \cap H'$ satisfies $(*)$. Let the fiber of $\eta '$ above a fixed $H'$ be $$T_{1,H'} = \{ S \ | \ S\cap H = C, \ S\cap H' \textup{ is rational integral } \}$$.
 
 For any curvilinear finite scheme $Y$ of length $6$ in $\mathbb P^{3}$, let $T_{2,Y}$ be the space of integral rational curves of genus $4$ which are canonically embedded in $\mathbb P^3$ and which contain $Y$. We have a map
 \begin{align*}
     \displaystyle \gamma: & T_{1,H'} \to T_{2,C \cap H'} \\  \displaystyle &S \mapsto S \cap H' 
 \end{align*}
 By Lemma \ref{prop3}, $\gamma$ is surjective. Let the fibre of $\gamma$ over $C'$ be 
 $$X_{C,C'} = \{S | S\cap H = C, \ S\cap H' =C' \}$$
 
 \begin{lemma} \label{xcc}
  For $C,C'$ such that $C \cap H' = C' \cap H$, $X_{C,C'}$ is irreducible of dimension $4$.
 \end{lemma}
 \begin{proof}
  By Lemma \ref{prop2}, we may write $C = Q_C \cap R_C$, and $C' = Q_{C'} \cap R_{C'}$, where $Q_C, Q_{C'}$ are uniquely determined quadrics in $H,H'$ respectively, and $R_C,R_{C'}$ are uniquely determined up to a linear multiple of $Q_C, Q_{C'}$ respectively. As argued in the proof of Lemma \ref{prop3}, we get $R_C, R_{C'}$ so that $R_C \cap H' = R_{C'} \cap H$.
 
 Consider $$X_1 = \{Q | \text{ Q quadric hypersurface in }\mathbb P^4 \text{ such that } Q \cap H = Q_C, \ Q \cap H' = Q_{C'} \}$$
 $$X_2 = \{R | \text{ R cubic hypersurface in }\mathbb P^4 \text{ such that } R \cap H \cap Q_C = C, \ R \cap H' \cap Q_{C'} = C' \}$$
 $X_3 = \{(R_1, R_2) |\ R_1$ cubic hypersurface in H such that $R_1 \cap Q_C = C,   \ R_2$ cubic hypersurface in $H'$ such that $R_2 \cap Q_{C'} = C'$, $R_1 \cap H' = R_2 \cap H \}$
 
 As seen in the proof of Lemma \ref{prop3}, $X_1$ is isomorphic to $\mathbb A^1$, so it is irreducible of dimension $1$. 
 
 There is a natural map $X_2 \to X_3$, $R \mapsto (R \cap H, R \cap H')$, whose fiber is a $\mathbb A^2$. $X_3$ is irreducible and 6 dimensional, since we may vary $R_1$ in a $\mathbb A^5$, and if we fix $R_1$, then the condition $R_1 \cap H' = R_2 \cap H$ means that we may we may vary $R_2$ in a $\mathbb A^1$. Hence, $X_2$ is irreducible of dimension $8$. 
 
 Also, we have a surjetive map $X_1 \times X_2 \to X_{C,C'}$, $(Q,R) \mapsto Q \cap R$, so $X_{C,C'}$ is irreducible. $(Q,R)$ and $(Q',R')$ have the same intersection iff $Q = Q'$ and $R = R'$ up to a linear multiple of $Q$, so the fibers have dimension $5$. Hence, $X_{C,C'}$ is irreducible of dimension $4$.
 
 \end{proof}
 
 Let $T_{3,Y}$ be the space of integral quadric hypersurfaces in $\mathbb P^{3}$ which contain $Y$. We have a natural map $\delta_Y: T_{2,Y} \to T_{3,Y}$ sending the curve to the unique (irreducible, reduced) quadric containing it.
 
 Recall the definition of $R_{Y,X}$ from Lemma \ref{functorcalc4}. Then we have that the fiber of $\delta_Y$ above a quadric hypersurface $X$ in $\mathbb P^3$ is $R_{Y,X}$.
 \bigskip
 
 The following is the key computation in the proof:
 
 \begin{prop} \label{irr4}
 Let $H$ be a hyperplane in $\mathbb P^3$, and let $Y$ be a disjoint union of $6$ points in $H$ which are general with the property of being contained in a conic in $H$. If $X$ is a smooth quadric surface in $\mathbb P^3$ containing $Y$, then $R_{Y,X}$ has dimension $5$ and only one irreducible component of dimension $5$.
 \end{prop}
 \begin{proof}
  We want to prove that if $X$ is smooth and $Y$ is a disjoint union of $6$ points and does not lie on a line, then $R_{Y,X}$ has only one irreducible component of dimension $5$. We may assume that $X \cong \mathbb P^1 \times \mathbb P^1$ is embedded in $\mathbb P^3$ via the Segre embedding, and may assume that $Y$ lies on the diagonal of $\mathbb P^1 \times \mathbb P^1$ (since $Y$ does not lie on a line in $\mathbb P^3$) Let the points of $Y$ be $y_1, \cdots, y_6$.
  
  The points of $R_{Y,X}$ correspond to images of maps $\alpha: \mathbb P^1 \to \mathbb P^1 \times \mathbb P^1$  so that image of $\alpha$, $C' = \alpha(\mathbb P^1)$ contains $y_1, \cdots, y_6$ and is in the linear system $(3,3)$ (since it is the intersection of a cubic in $\mathbb P^3$ with $X$). Now, $C'$ is intersects the diagonal exactly at $y_1, \cdots, y_6$. which is given by $[a:b] \mapsto ([p(a,b):q(a,b)], [r(a,b):s(a,b)])$, where $p,q,r,s$ are polynomials of degree $3$.
 
 Let the fixed $6$ points be
   $$ [0:1] \mapsto y_1 = ([0:1], [0:1]),$$ 
   $$[1:0] \mapsto y_2 = ([1:0], [1:0]),$$ 
   $$[1:1] \mapsto y_3 = ([1:1], [1:1]),$$
   $$[1:\lambda_1] \mapsto y_4 = ([1: \mu_1],[1: \mu_1]),$$
   $$[1:\lambda_2] \mapsto y_5 = ([1: \mu_2],[1: \mu_2]),$$
   $$[1:\lambda_3] \mapsto y_6 = ([1: \mu_3],[1: \mu_3]),$$
 Let $p(a,b) = p_0 a^3 + p_1a^2b + p_2ab^2 + p_3b^3$, and similar for $q,r,s$ with $p_i$'s replaced by $q_i,r_i,s_i$ respectively.
 
 Looking at these conditions as linear equations in the $p_i$ and $q_i$ gives us:
 
 \[
  \left[ {\begin{array}{cccccccc}
   0 & 0 & 0 & 1 & 0 & 0 & 0 & 0 \\
   0 & 0 & 0 & 0 & 1 & 0 & 0 & 0 \\
   1 & 1 & 1 & 1 & -1 & -1 & -1 & -1 \\
   \mu_1 & \mu_1 \lambda_1 & \mu_1 \lambda_1^2 & \mu_1 \lambda_1^3 & -1 & -\lambda_1 & -\lambda_1^2 & -\lambda_1^3 \\
   \mu_2 & \mu_2 \lambda_2 & \mu_2 \lambda_2^2 & \mu_2 \lambda_2^3 & -1 & -\lambda_2 & -\lambda_2^2 & -\lambda_2^3 \\
   \mu_3 & \mu_3 \lambda_3 & \mu_3 \lambda_3^2 & \mu_3 \lambda_3^3 & -1 & -\lambda_3 & -\lambda_3^2 & -\lambda_3^3 \\
  \end{array} } \right]
  \left[ {\begin{array}{c}
   p_0 \\ p_1 \\ p_2 \\ p_3 \\ q_0 \\ q_1 \\ q_2 \\ q_3
  \end{array} } \right] = 0
\]
and the same conditions on $r_i$ and $s_i$. Note that $\mu_i$ are fixed and distinct, and the $\lambda_i$ are also distinct and not equal to $0,1$.

Thus, $p_3 = q_0 = 0$, and
\[
  \left[ {\begin{array}{cccccc}
  1 & 1 & 1  & -1 & -1 & -1 \\
   \mu_1 & \mu_1 \lambda_1 & \mu_1 \lambda_1^2 & -\lambda_1 & -\lambda_1^2 & -\lambda_1^3 \\
   \mu_2 & \mu_2 \lambda_2 & \mu_2 \lambda_2^2 & -\lambda_2 & -\lambda_2^2 & -\lambda_2^3 \\
   \mu_3 & \mu_3 \lambda_3 & \mu_3 \lambda_3^2 & -\lambda_3 & -\lambda_3^2 & -\lambda_3^3 \\
  \end{array} } \right]
  \left[ {\begin{array}{c}
   p_0 \\ p_1 \\ p_2 \\ q_1 \\ q_2 \\ q_3
  \end{array} } \right] = 0
\]
Let us call the $4 \times 6$ matrix which is multiplied on the left in the above equation as $A$. If $A$ has full rank, we have a subspace of $\mathbb A^3 \times \mathbb P^5 \times \mathbb P^5$ corresponding to these 4 linearly independent linear conditions on $p_i,q_i$ as well as the same conditions on $r_i,s_i$, and thus we get a irreducible space of dimension 5.

Consider the $(\mathbb A^3 - \Delta) \times \mathbb A^3$ spanned by $\lambda_i$ and $\mu_j$ ($\Delta$ is the subscheme consisting of points where two of coordinates are equal or some coordinate is equal to $0$ or $1$), and consider the closed subvariety $\Gamma$ corresponding to the points where $A$ does not have full rank.
Consider the $4 \times 3$ matrix $A'$ formed by the last $3$ columns of $A$, and denote by $M_i$ the determinant of the $3 \times 3$ minors of $A'$ formed by deleting the $i^{th}$ row.
Then the determinant of the matrix formed by the $i^{th}$ column of $A$ together with the last three columns of $A$ will be 

\begin{align*}
    & M_1 + \mu_1 M_2 + \mu_2 M_3 + \mu_3 M_4 \textup{ for } i=1, \\ 
    & M_1 + \mu_1 \lambda_1 M_2 + \mu_2 \lambda_2 M_3 + \mu_3 \lambda_3 M_4 \textup{ for } i=2, \\
    & M_1 + \mu_1 \lambda_1^2 M_2 + \mu_2 \lambda_2^2 M_3 + \mu_3 \lambda_3^2 M_4 \textup{ for } i=3.
\end{align*}

These equations define $\Gamma$. These polynomials are linear in the $\mu_i$, and they are linearly independent because the matrix 
\[
  \left[ {\begin{array}{ccc}
  M_2 & M_3 & M_4 \\
   \lambda_1 M_2 & \lambda_2 M_3 & \lambda_3 M_4 \\
   \lambda_1^2 M_2 & \lambda_2^2 M_3 & \lambda_3^2 M_4 \\
  \end{array} } \right]
\]
has determinant $M_2 \cdot M_3 \cdot M_4 \cdot M_1$ which is not zero since $\lambda_i$ are distinct and not equal to $0,1$ (Each $M_i$ is the determinant of a vandermonde matrix). So we have that the projection of $\Gamma$ on to the first component is an isomorphism. 

Thus, the projection of $\Gamma$ on to the second component is either not dominant, or if it is dominant, then the general fiber is a finite scheme.

Now, the rank of $A$ is always $\ge 3$ since $A'$ has full rank. Hence, if $A$ does not have full rank then it has rank $3$. In this case, once we fix $\lambda_i$, we will get a space of dimension $2$ spanned by the $p_i,q_j$ times a space of dimension $2$ spanned by the $r_i,s_j$. But, now if $\mu_i$ are general then there can only be a finite scheme of the $\lambda_i$, hence we get that these have dimension $4$.

Thus, we have a unique component of maximum dimension $5$.
\end{proof}

\begin{prop} \label{twotrans4}
 $J'$ has dimension $17$ and has only one irreducible component of dimension $17$.
\end{prop}
\begin{proof}
We basically unwind the calculations above. Let $Q$ be the unique quadric surface in $H$ containing $C$. Then $Q$ is smooth by Lemma \ref{noncurvilinear}.
\begin{itemize}
    \item We consider the map $\delta_Y : T_{2,Y} \to T_{3,Y}$ for $Y = C \cap H'$. Note that $Y$ is the intersection of a conic and cubic in $H \cap H'$. So by Lemma \ref{uniqueconic} there will be a unique conic $D$ in $H \cap H'$ containing $Y$. Thus, $T_{3,Y}$ will be an open set of the linear system of $\{$quadrics $Q_0$ in $H'$ so that $Q_0 \cap H = D$ or $Q_0 \cap H = H \}$, and therefore irreducible of dimension 4. This linear system has base points at $D$, but $D$ has at most one singular point, so the general $Q_0$ will be smooth (by the same argument as in the proof of Lemma \ref{prop3}). Also, the subvariety of $T_{3,Y}$ of quadrics $Q_0$ with a singularity at a point of $Y$ is non-empty iff the unique conic $D = Q \cap H'$ containing $Y$ is singular at a point of $Y$ and it has codimension 1 in $T_{3,Y}$ in this case.

\item The map $\delta_Y$ is dominant if $Y$ is a union of $6$ general points on a conic : We proved in Lemma \ref{genpoints4} that on any smooth quadric $X$ in $\mathbb P^3$, and any $6$ general points on $X \cap H$, that there is a genus $4$ curve $C'$ canonically embedded in $\mathbb P^3$ which is lying on that smooth quadric and passing through these $6$ points.

\item By Proposition \ref{irr4}, if $Y$ is a union of $6$ general points on a conic and for smooth $X$, $\dim R_{Y,X} =5$ and $R_{Y,X}$ has only one irreducible component of $\dim 5$. Thus, if $Y$ is a disjoint union of $6$ general points on a conic, then $\dim T_{2,Y} =9$ and $T_{2,Y}$ has only one irreducible component of dimension $9$.

\item Let $C'$ be a curve satisfying $(*)$ represented by a point of $R_{Y,X}$. Note that the condition $(*)$ is an open condition in $R_{Y,X}$. Let $Y = Y_1 \cup \cdots \cup Y_m$, with $Y_i$ being the connected components of $Y$. By Lemma \ref{functorcalc4}, if $Y$ is curvilinear with every $Y_i$ of length $\le 2$, then we have that $(\dim R_{Y,X})_{C'} \le 5$ if $X$ is not singular at a point of $Y$ and $(\dim R_{Y,X})_{C'} \le 6$ if $X$ is singular at a point of $Y$. Therefore, $(\dim T_{2,Y})_{C'} \le 9$ for such $Y$. In general, we have $(\dim R_{Y,X})_{C'} \le 6$ regardless of $X$, so $(\dim T_{2,Y})_{C'} \le 10$ for any curvilinear $Y$.

\item If $Y$ is not curvilinear, then by Lemma \ref{noncurvilinear}, we know that $Y = Y_1 \cup p_1\cup p_2$ where $Y_1$ will be a non-curvilinear length $4$ scheme supported at a singular point $p_0$ of $C$ and $p_1,p_2$ are some two other points of $C$. Also, we know that in this case $Q \cap H'$ is the union of two distinct lines $L_1$ and $L_2$ in $H \cap H'$ meeting at $p_0$ so that $L_1$ and $L_2$ both intersect $Y$ with multiplicity 2 at $p_0$. So, by Lemma \ref{functorcalc4}, we have $(\dim R_{Y,X})_{C'} \le 7$. So, we get $(\dim T_{2,Y})_{C'} \le 11$ in this case.

\item Now, the fibers $X_{C,C'}$ of $\gamma$ are irreducible of dimension $4$ by Lemma \ref{xcc}. Thus, if $C \cap H'$ is the union of $6$ general points on the conic $Q \cap H'$, then $\dim T_{1,H'} = 13$ and it has only one irreducible component of $\dim 13$. Otherwise if $Y = C \cap H'$ and $S \in W''$, then we have $(\dim T_{1,H'})_{S} \le 13$ if $Y$ is curvilinear with every component $Y_i$ of length $\le 2$, $(\dim T_{1,H'})_{S} \le 14$ if $Y$ is curvilinear, and $(\dim T_{1,H'})_{S} \le 15$ if $Y$ is not curvilinear.

\item Thus, $(\mathbb P^4)^\vee -H = U \cup V_1 \cup V_2 \cup V_3$ where $U$ consists of $H'$ which meets $C$ in $6$ general points lying on $Q \cap H'$, $V_1$ consists of $H'$ in the complement of $U$ so that every component of $C \cap H'$ has length $\le 2$, $V_2$ consists of $H'$ in the complement of $U \cup V_1$ which have a curvilinear intersection with $C$, an $V_3$ is the space of $H'$ having non-curvilinear intersection with $C$. Note that $V_1$ has codimension 1 since the general hyperplane section will consist of $6$ general points, $V_2$ has codimension 2 due to Lemma \ref{hyperplanebound}, and $V_3$ has codimension 3 since any $H' \in V_3$ must contain the tangent plane at some singular point of $C$. We have that $T_{H'}$ has only one irreducible component of maximum dimension and $\dim T_{H'}=17$ over $H' \in U$, $(\dim T_{H'})_S \le 17$ for $[S] \in W'', H' \in V_1$ and $(\dim T_{H'})_S \le 18$ for $[S] \in W'', H' \in V_2$. 

     \item So if we consider the open set $\eta '^{-1}(U) \subset J'$, then we get that $$(\dim (J' \setminus \eta '^{-1}(U)))_{(S,H')} < 20$$ for $[S] \in W''$. Therefore, $(J' \setminus \eta '^{-1}(U))$ cannot dominate $W''$ and hence cannot dominate $W'$. Now, arguing similar as in the proof of Proposition \ref{twotrans3}, we can conclude that $J'$ has only one irreducible component which dominates $W'$. Thus, $\pi'$ has transitive monodromy by Lemma \ref{transitive}. \qedhere
\end{itemize} 
\end{proof}

\begin{prop}
 $\Pi_g$ is $2-$transitive for $g=4$.
\end{prop}
\begin{proof}
Let $H$ be a hyperplane in $\mathbb P^4$ and let $C \subset H$ be a canonically embedded integral rational curve of genus 4 corresponding to a general point of $Z_4$. Then by Lemma \ref{integral4}, there exists a smooth $K3$ surface $S$ in $\mathbb P^3$ such that $S \cap H = C$. This implies that for a surface $S$ corresponding to a general point of $W$, there is a hyperplane section of $S$ corresponding to a general point of $Z_4$.
 
 Using Proposition \ref{twotrans4} for this $C$ which is a general point of $Z_4$ and which is also a hyperplane section $S \cap H$ for a $S$ corresponding to a general point of $W$, we have that the monodromy group of $\pi '$ is transitive.
 Therefore, we get an element of the monodromy group which fixes $(S,H)$ and sends $(S,H')$ to $(S,H'')$ for any two points in the fiber $\pi^{-1}(S)$ different from $(S,H)$. Now, since $\Pi_g$ has already been proven to be transitive, so we get that $\Pi_g$ is 2-transitive.
 
 
 
 
\end{proof}

 \begin{prop}
  $\Pi_g$ contains a simple transposition for $g=4$.
 \end{prop}
 \begin{proof}
 We will follow the same proof as the proof of the $g=3$ case, i.e. Proposition \ref{ST3}. Thus, we only need to show (note that Beauville's formula holds in this case too)
 \begin{enumerate}
     \item The existence of a surface $S$ such  that $\pi^{-1}(S)$ consists of $(S,H)$ such that $S \cap H$ is rational nodal for all $H$ apart from one $H_0$ where $S \cap H_0$ has one simple cusp and rest nodal singularities.
 \item Local irreducibility of $J$ at this point $(S,H_0)$.
 \end{enumerate} 
 
Let $H,H'$ be two distinct hyperplanes in $\mathbb P^4$. Take smooth quadric surfaces $Q \subset H$ and $Q' \subset H'$ so that $Q \cap H' = Q' \cap H$ and let $Y$ to be the union of $6$ general points on $D$ where $D$ is the conic $Q \cap H' = Q' \cap H$. Then by Lemma \ref{genpoints4}, there exists a curve $C \subset H$ which is a curve corresponding to a general point of $R'_Q$ such that $C \cap H' = Y$. Lemma \ref{genpoints3} also gives us the existence of a curve $C' \subset H$ which is curve corresponding to a general point of $R_{Q'}$ such that $C' \cap H = Y$. Then by Lemma \ref{lemma2}, there exists a smooth $K3$ surface $S$ in $\mathbb P^3$ such that $S \cap H = C$ and $S \cap H' = C'$. This implies that for a surface $S$ corresponding to a general point of $W'$ (where $W'$ is defined with respect to $C$), there is a hyperplane section of $S$ corresponding to an integral rational nodal curve. Finally, we note that a curve corresponding to a general point of $Z_{3,1}$ which is canonically embedded in $H$ is contained in a smooth quadric surface, and all smooth quadric surfaces are projectively equivalent, so a general point of $R'_Q$ will also correspond to a general point of $Z_{3,1}$. Thus, we have that $C$ corresponds to a general point of 
$Z_{3,1}$.

 Using Proposition \ref{twotrans4}, we have that the monodromy group of $\pi '$ is transitive. Therefore, we get that the fiber of a general $S$ in $W'$ will only consist of points $(S,H') \in J'$ with $S \cap H'$ rational nodal. So, this $S$ will satisfy property $1$.

Finally, we need to show that $J$ is locally irreducible at the point $(S,H)$ with $S \cap H = C$ where $C$ is a general rational cuspidal curve of genus 4 canonically embedded in $\mathbb P^3$. 

Let $Y_H$, $Z$, $\psi: Y_H \to Z$ be as in the proof of Proposition \ref{trans4}. $Z$ has $Z_4$ as a dense open subset and hence locally irreducible at $C$ by Proposition \ref{locirr} (Note that $C$ is a point of $Z_4$ since it is general). $\psi$ is smooth, and hence $Y_H$ is also locally irreducible at the points of the fiber over $C$. Finally, we have that $\eta: J \to (\mathbb P^4)^\vee$ makes $J$ into an \'{e}tale locally trivial fibre bundle over $(\mathbb P^4)^\vee$ with fibres $Y_H$, thus $J$ is also locally irreducible at the points of the fiber over $C$. This ends the proof.
\end{proof}

\bibliographystyle{plain} 
\bibliography{rijul} 

\end{document}